\documentclass[12pt,fleqn]{article}

\usepackage{amsmath,amssymb,amsthm}
\usepackage{geometry} 
\usepackage[pdfpagemode=UseNone,pdfstartview=FitH,hypertexnames=false]{hyperref}
\usepackage{doi}

\newcommand{\A}{{\mathcal A}}
\newcommand{\As}[1][]{A_\sigma #1}
\newcommand{\abs}[1]{\left|#1\right|}
\newcommand{\B}{{\mathcal B}}
\newcommand{\Bs}[1][]{B_\sigma #1}
\newcommand{\bdry}[1]{\partial #1}
\newcommand{\bgdnorm}[2][]{\big\|#2\big\|_{#1}^\ast}
\newcommand{\bgset}[1]{\big\{#1\big\}}
\newcommand{\D}{{\mathcal D}}
\newcommand{\dist}[2]{\text{dist}\, (#1,#2)}
\newcommand{\dnorm}[2][]{\left\|#2\right\|_{#1}^\ast}
\newcommand{\E}{{\mathcal E}}
\newcommand{\eps}{\varepsilon}
\newcommand{\F}{{\mathcal F}}

\newcommand{\halfthin}{\kern 0.08333em}\newcommand{\loc}{\text{loc}}
\newcommand{\hquad}{\hspace{0.08in}}
\newcommand{\id}{id}
\newcommand{\incl}{\hookrightarrow}
\newcommand{\M}{{\mathcal M}}
\newcommand{\N}{\mathbb N}
\newcommand{\norm}[2][]{\left\|#2\right\|_{#1}}
\renewcommand{\O}{\text{O}}
\newcommand{\PS}[1]{$(\text{PS})_{#1}$}
\newcommand{\QED}{\mbox{\qedhere}}
\newcommand{\R}{\mathbb R}
\newcommand{\RP}{\R \text{P}}
\newcommand{\rad}{\text{rad}}
\newcommand{\restr}[2]{\left.#1\right|_{#2}}
\newcommand{\seq}[1]{\left(#1\right)}
\newcommand{\set}[1]{\left\{#1\right\}}
\newcommand{\wto}{\rightharpoonup}
\newcommand{\Z}{\mathbb Z}

\newcommand{\ds}[1]{\displaystyle #1}
\newcommand{\dint}{\ds{\int}}
\renewcommand{\o}{\text{o}}

\DeclareMathOperator{\sign}{sign}
\DeclareMathOperator{\supp}{supp}

\newenvironment{enumroman}{\begin{enumerate}

}{\end{enumerate}}

\newtheorem{corollary}{Corollary}[section]
\newtheorem{lemma}[corollary]{Lemma}
\newtheorem{proposition}[corollary]{Proposition}
\newtheorem{theorem}[corollary]{Theorem}

\theoremstyle{definition}
\newtheorem{definition}[corollary]{Definition}

\theoremstyle{remark}
\newtheorem{example}[corollary]{Example}
\newtheorem{remark}[corollary]{Remark}

\numberwithin{equation}{section}

\title{\bf A unifying zero-mass equation}
\author{\bf Kanishka Perera\\
Department of Mathematics\\
Florida Institute of Technology\\
150 W University Blvd, Melbourne, FL 32901-6975, USA\\
\em kperera@fit.edu}
\date{}

\begin{document}

\maketitle

\begin{abstract}
We introduce the nonlocal zero-mass equation
\[
- \Delta u + \Phi_u(|x|)\, |u|^{p-2}\, u = f(u) \quad \text{in } \R^N, \qquad \Phi_u(r) = a \int_r^\infty \rho^{-b}\, h_u^{q-1}(\rho)\, d\rho,
\]
where $h_u(\rho)$ is the mass of $|u|^p$ in the ball of radius $\rho$. For radial functions, this equation includes the defocusing inverse-power Schr\"{o}dinger equation, the Chern-Simons-Schr\"{o}dinger equation, and the Schr\"{o}dinger-Poisson-Slater equation as special cases. We develop the associated ball-mass Lebesgue and Sobolev spaces, which are uniformly convex, and prove sharp compact radial embeddings above a new critical exponent, a Br\'{e}zis-Lieb type splitting, and a Poho\v{z}aev identity. Exploiting a scaling invariance of the operator, we construct an unbounded sequence of eigenvalues of a scaled eigenvalue problem using the Fadell-Rabinowitz cohomological index, and obtain existence and multiplicity results in the subscaled, superscaled, and critical regimes. Our main result is a Br\'{e}zis-Nirenberg type theorem, proved by means of a new scaled linking theorem, which gives a nontrivial radial solution for every $\lambda > 0$ that is not an eigenvalue. Specializing to the three models recovers several known results in a unified way and yields new ones, including a Br\'{e}zis-Nirenberg type result for the Schr\"{o}dinger-Poisson-Slater equation in dimensions $N \ge 4$.
\end{abstract}

\begin{center}
\begin{minipage}{12cm}
\footnotesize \tableofcontents
\end{minipage}
\end{center}

\newpage

\section{Introduction}

This paper introduces a general nonlocal zero-mass equation that includes many physical models as special cases, among them the defocusing inverse-power Schr\"{o}dinger equation, Chern-Simons-Schr\"{o}dinger equation, and the Schr\"{o}dinger-Poisson-Slater equation.

Consider the equation
\begin{equation} \label{1}
- \Delta u + \Phi_u(|x|)\, |u|^{p-2}\, u = f(u) \quad \text{in } \R^N,
\end{equation}
where $N \ge 2$,
\begin{equation} \label{39}
\Phi_u(r) = a \int_r^\infty \rho^{-b}\, h_u^{q-1}(\rho)\, d\rho, \qquad h_u(\rho) = \int_{B_\rho} |u|^p\, dy,
\end{equation}
$r = |x|$, $a > 0$ is a normalization constant,
\begin{gather}
\label{33} 1 < p < \begin{cases}
\infty & \text{if } N = 2\\[5pt]
2^\ast & \text{if } N \ge 3,
\end{cases}\\[10pt]
\label{34} 1 \le q < \infty, \quad pq > 2,\\[10pt]
\label{2} 1 < b < b_\ast := \begin{cases}
1 + 2q & \text{if } N = 2\\[5pt]
1 + Nq\, (1 - p/2^\ast) & \text{if } N \ge 3,
\end{cases}
\end{gather}
$2^\ast = \frac{2N}{N-2}$ is the critical Sobolev exponent when $N \ge 3$, $B_\rho = \set{x \in \R^N : |x| < \rho}$, and $f : \R \to \R$ is a continuous function satisfying a suitable growth condition. In the radial case $u = u(r)$, special cases of equation \eqref{1} appear in different physical models. Despite the fact that some of these models have been studied extensively in recent years, the unifying ball-mass equation \eqref{1} in this general form and for the full range of parameters stated in \eqref{33}--\eqref{2} does not seem to have been considered in the literature. We will present a systematic structural study of this general unifying equation in this paper. Some of the results presented here are new even in the special cases that have appeared in the literature. First we describe some physical models.

When $q = 1$,
\[
\Phi_u(r) = a \int_r^\infty \rho^{-b}\, d\rho = \frac{ar^{-(b-1)}}{b - 1},
\]
so equation \eqref{1} with $a = b - 1$ reduces to the defocusing inverse-power Schr\"{o}dinger equation
\[
- \Delta u + \frac{|u|^{p-2}\, u}{|x|^{b-1}} = f(u) \quad \text{in } \R^N.
\]
This equation arises in the study of standing waves of the inhomogeneous nonlinear \linebreak Schr\"{o}dinger equation, which models the propagation of a laser beam through a medium of nonuniform density such as a plasma, the weight $|x|^{-(b-1)}$ accounting for the inhomogeneity of the medium. The sign of the inhomogeneous term is defocusing, so the self-interaction opposes concentration.

When $N = 2$, $p = 2$, $q = 3$, $b = 3$, and $u = u(r)$,
\[
\Phi_u(r) = 16 \pi^2 a \int_r^\infty \rho^{-3}\, \mathfrak{h}_u^2(\rho)\, d\rho,
\]
where $\mathfrak{h}_u = h_u/4 \pi$, i.e.,
\begin{equation} \label{623}
\mathfrak{h}_u(\rho) = \frac{1}{4 \pi} \int_{B_\rho} u^2\, dy = \int_0^\rho \frac{\tau}{2}\, u^2(\tau)\, d\tau.
\end{equation}
Integrating by parts assuming $u \in L^2(\R^2)$ and noting that
\begin{equation} \label{622}
\mathfrak{h}_u'(\rho) = \frac{\rho}{2}\, u^2(\rho)
\end{equation}
gives
\begin{equation} \label{621}
\Phi_u(r) = 8 \pi^2 a \left(\frac{\mathfrak{h}_u^2(r)}{r^2} + \int_r^\infty \frac{\mathfrak{h}_u(\rho)}{\rho}\, u^2(\rho)\, d\rho\right).
\end{equation}
So in this case equation \eqref{1} with $a = 1/8 \pi^2$ reduces to the Chern-Simons-Schr\"{o}dinger equation
\begin{equation} \label{3}
- \Delta u + \left(\frac{\mathfrak{h}_u^2(|x|)}{|x|^2} + \int_{|x|}^\infty \frac{\mathfrak{h}_u(\rho)}{\rho}\, u^2(\rho)\, d\rho\right) u = f(u) \quad \text{in } \R^2.
\end{equation}
The Chern-Simons-Schr\"{o}dinger system is a gauged nonlinear Schr\"{o}dinger equation in the plane describing the dynamics of a large number of nonrelativistic quantum particles interacting with a self-generated electromagnetic field whose dynamics is governed by the Chern-Simons term rather than the Maxwell term. Such planar gauge theories arise in the study of anyons, the fractional quantum Hall effect, and high-temperature superconductivity. For standing waves of zero vortex number, the gauge field can be eliminated in the Coulomb gauge, which reduces the system to the single nonlocal equation \eqref{3}.

When $N \ge 3$, $q = 2$, $b = N - 1$, and $u = u(r)$,
\[
\Phi_u(r) = a \int_r^\infty \rho^{-(N-1)}\, h_u(\rho)\, d\rho, \qquad h_u(\rho) = \omega_{N-1} \int_0^\rho \tau^{N-1}\, |u(\tau)|^p\, d\tau,
\]
where $\omega_{N-1}$ is the area of $\bdry{B_1}$. Since $\Phi_u'(r) = - ar^{-(N-1)}\, h_u(r)$ and $h_u'(\rho) = \omega_{N-1}\, \rho^{N-1}\, |u(\rho)|^p$,
\[
- \Delta \Phi_u = - r^{-(N-1)} \left(r^{N-1}\, \Phi_u'(r)\right)' = ar^{-(N-1)}\, h_u'(r) = \omega_{N-1}\, a\, |u|^p.
\]
Assuming $u \in L^p(\R^N)$, this gives
\begin{equation} \label{49}
\Phi_u(|x|) = \frac{a}{N - 2} \left(\frac{1}{|x|^{N-2}} \star |u|^p\right).
\end{equation}
So in this case equation \eqref{1} with $a = N - 2$ reduces to the Schr\"{o}dinger-Poisson-Slater equation
\begin{equation} \label{4}
- \Delta u + \left(\frac{1}{|x|^{N-2}} \star |u|^p\right) |u|^{p-2}\, u = f(u) \quad \text{in } \R^N.
\end{equation}
When $N = 3$ and $p = 2$, the convolution term is the electrostatic potential generated by the charge density $u^2$ through the Poisson equation, so equation \eqref{4} describes a system of electrons in a mean-field approximation, the nonlocal term accounting for the Coulomb repulsion among the electrons. The local term then models the exchange energy, the choice $f(u) = |u|^{s-2}\, u$ with $s = 8/3$ being the Slater approximation.

More generally, when $q = 2$ and $u = u(r)$, Fubini's theorem gives
\begin{multline*}
\Phi_u(r) = a \int_r^\infty \rho^{-b} \int_{B_\rho} |u(y)|^p\, dy\, d\rho = a \int_{\R^N} |u(y)|^p \int_{\max \set{r,|y|}}^\infty \rho^{-b}\, d\rho\, dy\\[7.5pt]
= \frac{a}{b - 1} \int_{\R^N} \frac{|u(y)|^p}{\max \set{r,|y|}^{b-1}}\, dy
\end{multline*}
since $b > 1$. Since $\max \set{|x|,\tau}^{2-N}$ is the average of $|x - y|^{2-N}$ over the sphere $|y| = \tau$ by Newton's theorem (see \cite[Theorem 9.7]{MR1817225}), this reduces to \eqref{49} when $b = N - 1$.

\subsection{Scaling property of the ball-mass operator}

The ball-mass operator
\[
\A(u) := - \Delta u + \Phi_u(|x|)\, |u|^{p-2}\, u
\]
has an important scaling property. Setting
\begin{equation} \label{37}
b_\# = N\halfthin (q - 1) + 3, \qquad \alpha = \frac{b_\# - b}{pq - 2}
\end{equation}
and
\[
u_t(x) = t^\alpha\, u(tx), \quad x \in \R^N,\, t \ge 0
\]
gives
\begin{equation} \label{28}
\A(u_t)(x) = t^{\alpha + 2}\, \A(u)(tx).
\end{equation}
Our analysis here will be based on this scaling and some scaling-based variational methods recently developed in \cite{MR5043800} (see Section \ref{section:scaling}). We note that
\begin{equation} \label{29}
b_\# - b_\ast = \frac{(N - 2)(pq - 2)}{2}
\end{equation}
and hence
\begin{equation} \label{35}
\alpha > \frac{N - 2}{2} \ge 0
\end{equation}
by \eqref{34} and \eqref{2}.

\subsection{Nonlinear eigenvalue problem}

The operator
\[
\B(u) := |u|^{\beta - 2}\, u,
\]
where
\begin{equation} \label{31}
\beta = 2 \left(1 + \frac{1}{\alpha}\right),
\end{equation}
scales the same way as $\A$, i.e.,
\[
\B(u_t)(x) = t^{\alpha + 2}\, \B(u)(tx).
\]
So
\begin{equation} \label{52}
\A(u) = \lambda\, \B(u)
\end{equation}
is a scaled eigenvalue problem for the operator $\A$ in the sense that if $u$ is an eigenfunction associated with an eigenvalue of this problem, then the entire $1$-parameter family $\set{u_t : t > 0}$ consists of eigenfunctions associated with the same eigenvalue. The abstract theory developed in \cite{MR5043800} gives a nondecreasing and unbounded sequence $\seq{\lambda_k}$ of positive eigenvalues of this problem. These scaled eigenvalues will play a central role in our results for equation \eqref{1}.

\subsection{Energy space and embeddings}

The natural energy space to look for solutions to equation \eqref{1} is the ball-mass Sobolev space $E^{p,q}_b(\R^N)$ consisting of all weakly differentiable functions $u : \R^N \to \R$ such that
\[
\int_{\R^N} |\nabla u|^2\, dx + \int_0^\infty \rho^{-b}\, h_u^q(\rho)\, d\rho < \infty.
\]
We will show that
\[
\norm{u} := \left[\left(\int_{\R^N} |\nabla u|^2\, dx\right)^{pq} + \left(\int_0^\infty \rho^{-b}\, h_u^q(\rho)\, d\rho\right)^2\right]^{1/2pq}
\]
is a norm on $E^{p,q}_b(\R^N)$ and that $E^{p,q}_b(\R^N)$ equipped with this norm is a uniformly convex Banach space (see Theorem \ref{Theorem 4}).

Denote by $E^{p,q}_{b,\rad}(\R^N)$ the closed subspace of $E^{p,q}_b(\R^N)$ consisting of radially symmetric functions. Set
\begin{equation} \label{30}
b^\ast = 2\, (N - 1)\, q + 1, \qquad 2^{p,q}_{b,\ast} = \frac{2\, [(N - 1)\, pq + b - 1]}{b^\ast - b}
\end{equation}
and note that
\begin{equation} \label{404}
p < 2^{p,q}_{b,\ast} < \begin{cases}
\infty & \text{if } N = 2\\[5pt]
2^\ast & \text{if } N \ge 3
\end{cases}
\end{equation}
by \eqref{2}. We will prove that $E^{p,q}_{b,\rad}(\R^2)$ is compactly embedded in $L^s(\R^2)$ for all $s \in (2^{p,q}_{b,\ast},\infty)$ and $E^{p,q}_{b,\rad}(\R^N)$ is embedded in $L^s(\R^N)$ continuously for all $s \in (2^{p,q}_{b,\ast},2^\ast]$ and compactly for $s \in (2^{p,q}_{b,\ast},2^\ast)$ for $N \ge 3$ (see Theorem \ref{Theorem 1}).

We have
\begin{equation} \label{630}
2^{p,q}_{b,\ast} = \frac{2\, (N - \kappa)}{N + \kappa - 2},
\end{equation}
where
\begin{equation} \label{48}
\kappa = \frac{2\, (b_\ast - b)}{(p + 2)\, q} > 0
\end{equation}
by \eqref{2}, so $2^{p,q}_{b,\ast}$ has the form of a Sobolev exponent. As $b \nearrow b_\ast$, $\kappa \searrow 0$ and hence $2^{p,q}_{b,\ast} \nearrow \infty$ if $N = 2$ and $2^{p,q}_{b,\ast} \nearrow 2^\ast$ if $N \ge 3$.

\subsection{Weak radial solutions and the variational formulation} \label{ssec:var-form}

In view of the above embeddings, a natural growth condition for the nonlinearity $f$ in equation \eqref{1} is
\begin{equation} \label{26}
|f(t)| \le C \left(|t|^{s_1 - 1} + |t|^{s_2 - 1}\right) \quad \forall t \in \R,
\end{equation}
where $C > 0$ is a constant and
\[
2^{p,q}_{b,\ast} < s_1 \le s_2 \begin{cases}
< \infty & \text{if } N = 2\\[5pt]
\le 2^\ast & \text{if } N \ge 3.
\end{cases}
\]
Under this growth condition, a weak radial solution of equation \eqref{1} is a function $u \in E^{p,q}_{b,\rad}(\R^N)$ satisfying
\[
\int_{\R^N} \nabla u \cdot \nabla v\, dx + \int_{\R^N} \Phi_u(|x|)\, |u|^{p-2}\, uv\, dx = \int_{\R^N} f(u)\, v\, dx \quad \forall v \in E^{p,q}_{b,\rad}(\R^N).
\]
Weak radial solutions coincide with critical points of the $C^1$-functional
\[
\E(u) = \frac{1}{2} \int_{\R^N} |\nabla u|^2\, dx + \frac{a}{pq} \int_0^\infty \rho^{-b}\, h_u^q(\rho)\, d\rho - \int_{\R^N} F(u)\, dx, \quad u \in E^{p,q}_{b,\rad}(\R^N), 
\]
where $F(t) = \int_0^t f(\tau)\, d\tau$ is the primitive of $f$ (see Proposition \ref{prop:c1}).

\subsection{Classification of nonlinear regimes}

As in \cite{MR5043800}, the scaling property \eqref{28} of the operator $\A$ naturally leads to the following classification of different nonlinear regimes for $f$, where $\beta$ is given in \eqref{31}:
\begin{enumroman}
\item $f$ is subscaled if
    \[
    \lim_{|t| \to \infty}\, \frac{f(t)}{|t|^{\beta - 2}\, t} = 0;
    \]
\item $f$ is asymptotically scaled if
    \[
    \lim_{|t| \to \infty}\, \frac{f(t)}{|t|^{\beta - 2}\, t} = \lambda \in (0,\infty);
    \]
\item $f$ is superscaled if
    \[
    \lim_{|t| \to \infty}\, \frac{f(t)}{|t|^{\beta - 2}\, t} = \infty.
    \]
\end{enumroman}
We will see in Section \ref{section:main} that the variational functional $\E$ has distinct geometries in these three regimes, which leads to different types of existence and multiplicity results for equation \eqref{1}.

\subsection{Related work} \label{ssec:related}

Equation \eqref{1} is a zero-mass problem in the sense that there is no linear term $\omega u$ on its left-hand side. Its solutions therefore need not decay exponentially, and $H^1(\R^N)$ is not the natural energy space. Zero-mass problems for the local equation $- \Delta u = f(u)$ go back to Berestycki and Lions \cite{MR695535,MR695536}, and the resulting loss of compactness has been addressed in various ways since then (see, e.g., \cite{MR2902133,MR2187794,MR3210961}). In the present setting it is the ball-mass term that supplies the missing control, and the ball-mass Sobolev space $E^{p,q}_b(\R^N)$ introduced in Section \ref{section:bm-spaces} takes over the role of $H^1(\R^N)$. The literature on the three models described above has developed along largely separate lines and we review it in turn.

\emph{Schr\"{o}dinger-Poisson-Slater equation.} For the physical background of equation \eqref{4} we refer to \cite{MR2013491,MR1702877,MR1836081,MR2032129,Slater}. The case $N = 3,\, p = 2$ has been studied extensively. The natural energy space, namely, the space of functions in $D^{1,2}(\R^3)$ with finite Coulomb energy, goes back to Lions \cite{MR636734}. Ruiz \cite{MR2679375} studied this space in detail, showed that its radial subspace is embedded in $L^s(\R^3)$ continuously for $s \in (18/7,6]$ and compactly for $s \in (18/7,6)$, and in no $L^s(\R^3)$ with $s < 18/7$, and obtained a positive solution of the pure power equation for $s \in (18/7,3)$. Ianni and Ruiz \cite{MR2902293} obtained a groundstate and infinitely many radial solutions for $s \in (3,6)$, a sequence of eigenvalues at the scaled exponent $s = 3$, and nonexistence for $s \ge 6$ via a Poho\v{z}aev identity. Infinitely many solutions at negative energy levels for $s \in (18/7,3)$, as well as multiplicity results for critical growth problems near the eigenvalues and for large perturbations, were obtained by Mercuri and Perera \cite{MR5043800}. Positive solutions of the corresponding critical problems were obtained earlier by Liu, Zhang, and Huang \cite{MR3912770}. For general $N$ and $p$ and Riesz potentials of general order, the Coulomb-Sobolev space was introduced by Mercuri, Moroz, and Van Schaftingen \cite{MR3568051}, who developed its functional analytic theory, established optimal interpolation inequalities and radial embeddings, and obtained groundstates (see also \cite{MR3852465,MR4292779}). Building on \cite{MR3568051} and \cite{MR5043800}, the scaled eigenvalue problem and a range of existence and multiplicity results, including results for critical growth, were extended to general $N$ and $p$ by Marinho, Mercuri, and Perera \cite{MarMePe}. For the Schr\"{o}dinger-Poisson-Slater equation with a positive mass we refer to \cite{MR2465993} and the references therein.

\emph{Chern-Simons-Schr\"{o}dinger equation.} The Chern-Simons-Schr\"{o}dinger system was introduced by Jackiw and Pi \cite{MR1084552}, and the reduction of its radial standing waves to equation \eqref{3} is due to Byeon, Huh, and Seok \cite{MR2948224}. Existence, nonexistence, and multiplicity results for this equation were obtained in \cite{MR2948224,MR3415024,MR3353806}. All of these works include a frequency $\omega > 0$, or equivalently a nonlinearity with a negative linear part near zero, and are set in $H^1_\rad(\R^2)$. The zero-mass case has been considered in \cite{MR4708596,MR4968174,MR4502773} for nonlinearities with critical or supercritical exponential growth, but there an additional term of the form $- a\, |u|^{r-2}\, u$ with $a > 0$ and $r > 2$ is added to the right-hand side in order to produce a workable energy space, namely, the space of radial functions $u$ with $\nabla u \in L^2(\R^2)$ and $u \in L^r(\R^2)$. In the present framework no such device is needed since the Chern-Simons term is itself the ball-mass energy (see \eqref{621}).

\emph{Defocusing inverse-power Schr\"{o}dinger equation.} For the physical background of equation \eqref{600} we refer to \cite{Gill,LiuTripathi} (for related equations with spatially decaying nonlinearities see \cite{MR2379460,MR2834784}). When $q = 1$ the energy space $E^{p,1}_b(\R^N)$ is a weighted Sobolev space of Caffarelli-Kohn-Nirenberg type \cite{MR768824}. Compact radial embeddings of weighted Sobolev spaces were studied by Su, Wang, and Willem \cite{MR2334597} and by Badiale, Guida, and Rolando \cite{MR3385192,MR3576582}, although in those works the weighted term is paired with a gradient term of the same order, which leads to spaces different from $E^{p,1}_b(\R^N)$. The closest predecessor is the work of Gloss, Perera, and Ribeiro \cite{GlPeRi}, where equation \eqref{600} with competing weighted nonlinearities is studied in $E^{p,1}_b(\R^N)$.

\emph{Scaling-based methods.} The variational methods used here were introduced by Mercuri and Perera \cite{MR5043800} and were applied there to the Schr\"{o}dinger-Poisson-Slater equation with $N = 3$ and $p = 2$. They were subsequently applied to the Schr\"{o}dinger-Poisson-Slater equation with general $N$ and $p$ in \cite{MarMePe} and to inhomogeneous Schr\"{o}dinger equations in \cite{GlPeRi}.

These lines of work share no common framework. The Coulomb-Sobolev space theory is confined to $q = 2$, where the interaction is generated by a Riesz potential, the Chern-Simons-Schr\"{o}dinger literature treats a single set of parameters arising from the gauge structure, and the inhomogeneous case is local. What makes a unified treatment possible here is the scaling property \eqref{28}, which holds for the whole range \eqref{33}--\eqref{2}. As a consequence, some of the results in Subsections \ref{ssec:app-ip}--\ref{ssec:app-sps} recover known results in the special cases, while others are new. The precise relationship with the literature is described in the remarks following each theorem in those subsections.

\subsection{Outline of the paper} \label{ssec:outline}

The rest of the paper is organized as follows. Section \ref{section:bm-spaces} develops the functional analytic framework for equation \eqref{1}. The ball-mass Lebesgue space $L^{p,q}_b(\R^N)$ and the ball-mass Sobolev space $E^{p,q}_b(\R^N)$ do not appear to have been studied before in the generality of \eqref{33}--\eqref{2}. The special cases $q = 1$ and $q = 2$, $b = N - 1$ are discussed at the beginning of that section, where the results that are known in those cases are identified. We introduce $L^{p,q}_b(\R^N)$ and show that it is a uniformly convex and uniformly smooth Banach space (Theorem \ref{Theorem 2}). The proof exhibits a linear isometry \eqref{510} of $L^{p,q}_b(\R^N)$ into a Lebesgue-Bochner space, which will be used repeatedly in the sequel. Theorem \ref{Theorem 3} gives a scale-explicit local embedding in $L^p_\loc(\R^N)$. We then prove a pairing inequality for the ball-mass term (Theorem \ref{thm:bm-pairing}), together with a characterization of the case of equality, and deduce the boundedness of the associated nonlocal form (Corollary \ref{cor:bm-pairing}) and the identity \eqref{73}, which is the basic computational tool throughout the paper. A Br\'{e}zis-Lieb type splitting in $L^{p,q}_b(\R^N)$ is proved in Theorem \ref{thm:bm-brezis-lieb}, which is what makes the concentration analysis in the critical case possible. We introduce $E^{p,q}_b(\R^N)$ and prove that it is uniformly convex, hence reflexive (Theorem \ref{Theorem 4}), and continuously embedded in $H^1_\loc(\R^N)$ (Theorem \ref{Theorem 5}). Proposition \ref{prop:c1} shows that the ball-mass potential is of class $C^1$ and computes its derivative, which justifies the variational formulation given in Subsection \ref{ssec:var-form}. The argument uses the uniform smoothness of $L^{p,q}_b(\R^N)$ and is not a routine differentiation under the integral sign.

The heart of Section \ref{section:bm-spaces} is the embedding theory. We first show that without any symmetry assumption the ball-mass term makes no contribution, in the sense that $E^{p,q}_b(\R^N)$ embeds in $L^s(\R^N)$ only for $N \ge 3$ and $s = 2^\ast$, and this embedding is not compact (Theorem \ref{thm:no-embedding}). Restricting to radial functions changes the picture completely. Theorem \ref{Theorem 1}, the main result of the section, gives compact embeddings of $E^{p,q}_{b,\rad}(\R^N)$ in $L^s(\R^N)$ for every $s < 2^\ast$ above the exponent $2^{p,q}_{b,\ast}$ defined in \eqref{30}, which is a new critical exponent attached to the parameters $(p,q,b)$ and has the form of a Sobolev exponent (see \eqref{630}). The proof rests on an exterior Gagliardo-Nirenberg type inequality with explicit decay rates in the radius (Lemma \ref{Lemma 1}), obtained from a dyadic annular decomposition. Theorem \ref{thm:emb-sharp} shows that $2^{p,q}_{b,\ast}$ is sharp, so the range in Theorem \ref{Theorem 1} cannot be enlarged. The remaining subsections establish further properties of $E^{p,q}_{b,\rad}(\R^N)$ that the variational arguments require. These include the weak continuity of the ball-mass operator (Theorem \ref{thm:weak-cont}), a radial decay estimate with an explicit rate (Theorem \ref{Theorem 6}), a Trudinger-Moser type inequality in the borderline dimension $N = 2$ together with the attainment of the supremum (Theorem \ref{Theorem 7}), and the density of radial test functions (Theorem \ref{Theorem 9}). The constant $4 \pi$ in Theorem \ref{Theorem 7} is sharp, and whether the inequality persists at $\nu = 4 \pi$ is left open (see Remark \ref{rmk:tm-sharp}).

Section \ref{section:scaling} is largely expository. We recall the notion of a scaling in a Banach space, scaled operators, and the associated scaled eigenvalue problems from \cite{MR5043800}, together with the $\Z_2$-cohomological index of Fadell and Rabinowitz \cite{MR0478189}, whose piercing property is essential for our purposes here and is not shared by the genus. Two results in this section are new and of independent interest. Theorem \ref{Theorem 10} is a multiplicity result for subscaled and superscaled problems, obtained by applying the index-based minimax theory to a suitable constrained functional. It produces an infinite sequence of nontrivial solutions together with the precise sign and asymptotics of their energies. Theorem \ref{Theorem 8} is a scaled linking theorem in which the linking set is built from a homotopy along the fibers of the scaling rather than from a fixed finite-dimensional subspace. It is derived from Proposition \ref{Proposition 9}, a variant of \cite[Proposition 2.22]{MR5043800}.

Our main results for equation \eqref{1} are proved in Section \ref{section:main}. We first verify that the scaling \eqref{103} induced by \eqref{28} satisfies the abstract assumptions $(A_1)$--$(A_5)$, the continuity being the only delicate point (Proposition \ref{prop:continuous}), and that the ball-mass operator and its potential $I_\sigma$ satisfy $(A_6)$, $(A_7)$, and $(A_{10})$ (Proposition \ref{prop:a7}). Along the way we record the alternative form \eqref{140} of $I_\sigma$, which is what the applications in Subsections \ref{ssec:app-ip}--\ref{ssec:app-sps} use. Theorem \ref{thm:bm-regularity} establishes the regularity of weak radial solutions together with pointwise bounds for $u$, $u'$, $u''$, and $\Phi_u$ near the origin, where the ball-mass potential is singular. Remark \ref{rmk:distributional} shows that weak radial solutions are also solutions in the sense of distributions. Theorem \ref{thm:bm-pohozaev} proves a Poho\v{z}aev identity, a Nehari identity, and a scaling identity for equation \eqref{1}. Besides verifying the abstract assumptions $(A_{11})$ and $(f_4)$, it is the source of our nonexistence result in the critical case. These results are new in the generality of \eqref{33}--\eqref{2}.

The existence results then proceed by nonlinear regime. Theorem \ref{thm:bm-eigenvalues} shows that the scaled eigenvalue problem \eqref{51} has an unbounded sequence of positive eigenvalues with radial eigenfunctions. For the pure power equation \eqref{star}, Theorem \ref{thm:pure-power} gives an infinite sequence of nontrivial weak radial solutions in both the subscaled and the superscaled regimes, with energies accumulating at $0$ from below in the first case and diverging to $+ \infty$ in the second. Subsection \ref{ssec:critical} treats critical growth. Theorem \ref{thm:pure-crit-nonex} shows that the pure critical power equation \eqref{pure-crit} has no nontrivial weak radial solution, so a perturbation is necessary. Proposition \ref{prop:local-ps} establishes a local Palais-Smale condition below the threshold $\frac{1}{N}\, S^{N/2}$, the proof combining the Br\'{e}zis-Lieb splitting of Theorem \ref{thm:bm-brezis-lieb} with the constraint $pq < 2^\ast$. Theorem \ref{thm:left-nbhd} then produces $m$ distinct pairs of solutions of equation \eqref{136} for all $\lambda$ in a suitably small left neighborhood of an eigenvalue of multiplicity $m$, and Theorem \ref{thm:large-mu} produces arbitrarily many pairs of solutions of equation \eqref{138} for all sufficiently large $\mu$.

Subsection \ref{ssec:bn} contains our main result, Theorem \ref{thm:bn}, a Br\'{e}zis-Nirenberg type theorem for the ball-mass operator that is the analog of the classical results of \cite{MR709644} and \cite{MR831041} and gives a nontrivial weak radial solution of equation \eqref{bn-eq} for every $\lambda > 0$ that is not an eigenvalue of \eqref{51}. The case $\lambda < \lambda_1$ is a mountain pass argument, but the case $\lambda_k < \lambda < \lambda_{k+1}$ requires the new scaled linking geometry of Theorem \ref{Theorem 8}, applied to a homotopy that interpolates along the fibers of the scaling between a suitably chosen compact symmetric subset $A_0$ of a sublevel set of $\widetilde{\Psi}$ and a concentrating truncated Talenti bubble. Two ingredients are specific to the ball-mass setting and have no counterparts in the classical argument, namely, the asymptotics of the ball-mass energy of the truncated bubble (Lemma \ref{lem:bubble}), where the gain $\eps^{\delta_\beta}$ coming from the subcritical term must be shown to beat the loss $\eps^\nu$ coming from the ball-mass term, and the construction of an index-$k$ set $A_0$ whose elements are supported away from the origin (Lemma \ref{lem:A0}), which is what makes the bubble and the linking set interact through disjoint supports. The resonant case $\lambda = \lambda_k$ is not covered by Theorem \ref{thm:bn} and remains open (see Remark \ref{rmk:resonant}).

Subsections \ref{ssec:app-ip}--\ref{ssec:app-sps} specialize the results of Subsections \ref{ssec:eigen-prob}--\ref{ssec:bn} to the three physical models described previously. Not all of the results apply to all three models, and which ones do is determined by the sign of $\beta - 2^{p,q}_{b,\ast}$, computed in \eqref{505}. For the defocusing inverse-power Schr\"{o}dinger equation this quantity is negative, and only the results that do not involve the scaled eigenvalues survive. For the Chern-Simons-Schr\"{o}dinger equation $\beta - 2^{p,q}_{b,\ast}$ vanishes, and there is no critical exponent since $N = 2$, so among the general results only the pure power theorem applies. The borderline exponent $4$ that emerges in \eqref{608} agrees with the one known in the literature for this model. For the Schr\"{o}dinger-Poisson-Slater equation $\beta - 2^{p,q}_{b,\ast}$ is positive and the full theory applies, subject to the further restriction \eqref{616} in the critical case. In particular, Theorem \ref{thm:bn} yields a Br\'{e}zis-Nirenberg type result for the Schr\"{o}dinger-Poisson-Slater equation in dimensions $N \ge 4$. Some of the results obtained in these three subsections are known for the individual models, while others appear to be new, and each theorem there is followed by a remark comparing it with the literature (see also Subsection \ref{ssec:related}). The situation is summarized in Table \ref{tab:models}. Finally, Subsection \ref{section:open} collects some open problems.

\begin{table}[ht]
\centering
\small
\renewcommand{\arraystretch}{1.35}
\begin{tabular}{|p{2.2cm}|p{3.7cm}|c|c|c|p{3.8cm}|}
\hline
Model & Parameters & $\alpha$ & $\beta$ & $\beta - 2^{p,q}_{b,\ast}$ & Results \\
\hline
Inverse-power & $q = 1$, $a = b - 1$, $N \ge 2$, $2 < p < 2^\ast$ ($p < \infty$ if $N = 2$), $1 < b < 1 + N - \frac{(N - 2)\, p}{2}$ & $\dfrac{3 - b}{p - 2}$ & $\dfrac{2\, (p + 1 - b)}{3 - b}$ & $< 0$ & Theorems \ref{Theorem 4.20} and \ref{Theorem 4.21} (known), Theorem \ref{Theorem 4.22} (new) \\
\hline
Chern-Simons-Schr\"{o}dinger & $N = 2$, $p = 2$, $q = 3$, $b = 3$, $a = 1/8 \pi^2$ & $1$ & $4$ & $0$ & Theorem \ref{thm:css-pure} (new) \\
\hline
Schr\"{o}dinger-Poisson-Slater & $N \ge 3$, $q = 2$, $b = N - 1$, $a = N - 2$, $1 < p < \frac{N + 2}{N - 2}$ & $\dfrac{2}{p - 1}$ & $p + 1$ & $> 0$ & Theorems \ref{thm:sps-eigenvalues}, \ref{thm:sps-left}, and \ref{thm:sps-large} (known), Theorem \ref{thm:sps-pure} (new for general $N$ and $p$), Theorem \ref{thm:sps-nonex}, Theorem \ref{thm:sps-bn} (new, $N \ge 4$) \\
\hline
\end{tabular}
\caption{The three physical models as special cases of equation \eqref{1}, with the scaling exponent $\alpha$ from \eqref{37}, the scaled exponent $\beta$ from \eqref{31}, the sign of $\beta - 2^{p,q}_{b,\ast}$ from \eqref{505}, and the results of Subsections \ref{ssec:app-ip}--\ref{ssec:app-sps} (see the remarks following each result for the precise relationship with the literature).}
\label{tab:models}
\end{table}

\section{Ball-mass spaces} \label{section:bm-spaces}

The spaces introduced in this section are new in the generality of \eqref{33}--\eqref{2}, but two special cases have appeared before (see Subsection \ref{ssec:related}), and it is useful to see how they sit inside the present scale. When $q = 1$, \eqref{602} shows that $L^{p,1}_b(\R^N)$ is the weighted Lebesgue space of functions with $|x|^{-(b-1)/p}\, u \in L^p(\R^N)$, so $E^{p,1}_b(\R^N)$ is a Sobolev space of Caffarelli-Kohn-Nirenberg type \cite{MR768824}. This space was studied in \cite{GlPeRi}, where its completeness and uniform convexity and the density of test functions were established, together with radial embeddings that coincide with those of Theorem \ref{Theorem 1} for this case \cite[Lemma 5.1]{GlPeRi}. When $q = 2$ and $b = N - 1$, \eqref{49} shows that for radial functions the ball-mass term is the Coulomb energy, so that $E^{p,2}_{N-1,\rad}(\R^N)$ is the radial Coulomb-Sobolev space of \cite{MR3568051}. In this case the completeness and uniform convexity in Theorem \ref{Theorem 4}, the density of test functions in Theorem \ref{Theorem 9}, the splitting of Theorem \ref{thm:bm-brezis-lieb}, the embeddings of Theorem \ref{Theorem 1}, and the decay estimate of Theorem \ref{Theorem 6} all correspond to results established in \cite{MR3568051} for the Coulomb-Sobolev space, and for $N = 3$ and $p = 2$ the embeddings and their sharpness are due to Ruiz \cite[Theorem 1.2]{MR2679375}. The endpoint $s = 2^{p,q}_{b,\ast}$, which Theorem \ref{Theorem 1} leaves open, behaves differently in the two cases. The corresponding embedding holds when $q = 1$ \cite[Lemma 5.1]{GlPeRi} and fails when $q = 2$ and $b = N - 1$ \cite[Theorem 4]{MR3568051}.

Two further points should be kept in mind. First, the Coulomb-Sobolev space is defined through a Riesz potential and is therefore invariant under translations, whereas the ball-mass term is anchored at the origin. By Newton's theorem the two coincide on radial functions but not in general, so Theorem \ref{thm:no-embedding} concerns $E^{p,q}_b(\R^N)$ and is not in conflict with the embeddings of the full Coulomb-Sobolev space obtained in \cite[Theorem 1]{MR3568051} and \cite[Theorem 1.5]{MR2679375}. Second, both special cases are $1$-parameter degenerations of the present scale. For $q = 1$ the ball-mass term is local and for $q = 2$ it is generated by a Riesz potential, whereas for $q \ne 1, 2$ it is neither and no framework covering this range appears to be available in the present literature. In particular, the Chern-Simons-Schr\"{o}dinger case $q = 3$ lies outside both theories.

\subsection{Ball-mass Lebesgue spaces}

For
\begin{equation} \label{5}
1 < p < \infty, \qquad 1 \le q < \infty, \qquad 1 < b < b^\# := 1 + Nq,
\end{equation}
we define the ball-mass Lebesgue space $L^{p,q}_b(\R^N)$ to be the space of all measurable functions $u : \R^N \to \R$ such that
\[
\int_0^\infty \rho^{-b}\, h_u^q(\rho)\, d\rho < \infty,
\]
where $h_u$ is given in \eqref{39}.

\begin{theorem} \label{Theorem 2}
If \eqref{5} holds, then $L^{p,q}_b(\R^N)$ equipped with the norm
\[
\norm[L^{p,q}_b(\R^N)]{u} = \left(\int_0^\infty \rho^{-b}\, h_u^q(\rho)\, d\rho\right)^{1/pq}
\]
is a uniformly convex and uniformly smooth Banach space.
\end{theorem}

\begin{proof}
First we show that $\norm[L^{p,q}_b(\R^N)]{\cdot}$ is a norm. We have
\[
\norm[L^{p,q}_b(\R^N)]{\cdot} = \big\|\!\norm[L^p(B_\rho)]{\cdot}\!\big\|_{L^{pq}((0,\infty),\,\rho^{-b}\, d\rho)}.
\]
It is immediate from this that $\norm[L^{p,q}_b(\R^N)]{\cdot}$ is positive definite and absolutely homogeneous. The triangle inequality also follows since $\norm[L^{pq}((0,\infty),\,\rho^{-b}\, d\rho)]{\cdot}$ is monotone on the positive cone.

The map
\begin{equation} \label{510}
\Lambda : L^{p,q}_b(\R^N) \to X, \quad \Lambda u(x,\rho) = u(x)\, \mathbf{1}_{B_\rho}(x),
\end{equation}
where $X$ is the Lebesgue-Bochner space $L^{pq}((0,\infty),\,\rho^{-b}\, d\rho;L^p(\R^N))$ and $\mathbf{1}_{B_\rho}$ denotes the indicator function of $B_\rho$, is a linear isometry. Since $pq \ge p > 1$ by \eqref{5}, $X$ is uniformly convex and uniformly smooth (see, e.g., \cite[Chapter 1]{MR1400007}). It follows that $L^{p,q}_b(\R^N)$ is uniformly convex and uniformly smooth.

It remains to show that $L^{p,q}_b(\R^N)$ is complete. For $u \in L^{p,q}_b(\R^N)$ and $R > 0$,
\[
\norm[L^{p,q}_b(\R^N)]{u}^{pq} \ge \int_R^\infty \rho^{-b}\, h_u^q(\rho)\, d\rho \ge h_u^q(R) \int_R^\infty \rho^{-b}\, d\rho = \frac{R^{-(b-1)}}{b - 1} \norm[L^p(B_R)]{u}^{pq},
\]
so a Cauchy sequence $\seq{u_j}$ in $L^{p,q}_b(\R^N)$ is a Cauchy sequence in $L^p(B_R)$ for every $R > 0$ and hence converges in $L^p_\loc(\R^N)$ to some measurable function $u$. Let $\eps > 0$ and let $j_0$ be such that $\norm[L^{p,q}_b(\R^N)]{u_j - u_k} < \eps$ for all $j, k \ge j_0$. For fixed $j \ge j_0$, $\norm[L^p(B_\rho)]{u_j - u_k} \to \norm[L^p(B_\rho)]{u_j - u}$ as $k \to \infty$ for every $\rho \in (0,\infty)$, so Fatou's lemma in $L^{pq}((0,\infty),\,\rho^{-b}\, d\rho)$ gives
\[
\norm[L^{p,q}_b(\R^N)]{u_j - u} \le \liminf_{k \to \infty}\halfthin \norm[L^{p,q}_b(\R^N)]{u_j - u_k} \le \eps.
\]
Hence $u \in L^{p,q}_b(\R^N)$ and $u_j \to u$ in $L^{p,q}_b(\R^N)$.
\end{proof}

We have the following local embedding theorem for $L^{p,q}_b(\R^N)$.

\begin{theorem} \label{Theorem 3}
If \eqref{5} holds, then there exists a constant $C > 0$ such that
\[
\norm[L^p(B_R)]{u} \le CR^{(b-1)/pq} \norm[L^{p,q}_b(\R^N)]{u} \quad \forall u \in L^{p,q}_b(\R^N)
\]
for all $R > 0$. In particular, $L^{p,q}_b(\R^N)$ is continuously embedded in $L^p_\loc(\R^N)$.
\end{theorem}

\begin{proof}
This was shown in the last paragraph of the proof of Theorem \ref{Theorem 2}.
\end{proof}

Recall that
\begin{multline*}
b_\ast = \begin{cases}
2q + 1 & \text{if } N = 2\\[5pt]
Nq\, (1 - p/2^\ast) + 1 & \text{if } N \ge 3,
\end{cases} \qquad b_\# = N\halfthin (q - 1) + 3, \qquad b^\# = Nq + 1,\\[7.5pt]
b^\ast = 2\, (N - 1)\, q + 1.
\end{multline*}
We have
\begin{equation} \label{109}
b^\# - b_\# = N - 2, \qquad b^\ast - b^\# = (N - 2)\, q,
\end{equation}
which together with \eqref{29} and \eqref{34} show that
\begin{equation} \label{36}
b_\ast \le b_\# \le b^\# \le b^\ast
\end{equation}
and equality holds throughout if and only if $N = 2$.

\subsection{Pairing inequality}

The linear isometry given in \eqref{510} and the H\"{o}lder inequality in $X$ give us the following pairing inequality for $L^{p,q}_b(\R^N)$.

\begin{theorem} \label{thm:bm-pairing}
If \eqref{5} holds, then
\[
\int_0^\infty \rho^{-b}\, h_u^{q-1}(\rho) \bigg(\int_{B_\rho} |u|^{p-2}\, uv\, dx\bigg) d\rho \le \norm[L^{p,q}_b(\R^N)]{u}^{pq-1} \norm[L^{p,q}_b(\R^N)]{v} \quad \forall u, v \in L^{p,q}_b(\R^N).
\]
For $u \ne 0$, equality holds if and only if $v = cu$ for some constant $c \ge 0$.
\end{theorem}

\begin{proof}
The dual of $X$ is $X^\ast = L^{(pq)'}((0,\infty),\,\rho^{-b}\, d\rho;L^{p'}(\R^N))$, where $p' = p/(p - 1)$ and $(pq)' = pq/(pq - 1)$. Set
\[
H_u(x,\rho) = h_u^{q-1}(\rho)\, |u(x)|^{p-2}\, u(x)\, \mathbf{1}_{B_\rho}(x).
\]
Since $(p - 1)\, p' = p$, we have
\[
\norm[L^{p'}(\R^N)]{H_u(\cdot,\rho)} = h_u^{q-1}(\rho) \bigg(\int_{B_\rho} |u|^{(p - 1)\, p'}\halfthin dx\bigg)^{1/p'} = h_u^{q - 1 + 1/p'}(\rho) = h_u^{q - 1/p}(\rho).
\]
Since $(q - 1/p)(pq)' = q$, this gives
\[
\norm[X^\ast]{H_u}^{(pq)'} = \int_0^\infty \norm[L^{p'}(\R^N)]{H_u(\cdot,\rho)}^{(pq)'} \rho^{-b}\, d\rho = \int_0^\infty \rho^{-b}\, h_u^q(\rho)\, d\rho = \norm[L^{p,q}_b(\R^N)]{u}^{pq},
\]
so $\norm[X^\ast]{H_u} = \norm[L^{p,q}_b(\R^N)]{u}^{pq-1}$. The H\"{o}lder inequality in $X$ now gives
\begin{multline*}
\int_0^\infty \rho^{-b}\, h_u^{q-1}(\rho) \bigg(\int_{B_\rho} |u|^{p-2}\, uv\, dx\bigg) d\rho = \int_0^\infty \bigg(\int_{\R^N} H_u(x,\rho)\, \Lambda v(x,\rho)\, dx\bigg) \rho^{-b}\, d\rho\\[7.5pt]
\le \norm[X^\ast]{H_u} \norm[X]{\Lambda v} = \norm[L^{p,q}_b(\R^N)]{u}^{pq-1} \norm[L^{p,q}_b(\R^N)]{v}
\end{multline*}
since $\Lambda$ is an isometry.

Suppose $u \ne 0$ and equality holds. Then equality holds in the H\"{o}lder inequality in $L^p(B_\rho)$ pairing $H_u(\cdot,\rho)$ with $\Lambda v(\cdot,\rho)$ for a.e.\! $\rho$. So $v = c(\rho)\, u$ a.e.\! in $B_\rho$ for some constant $c(\rho) \ge 0$ for a.e.\! $\rho$ with $h_u(\rho) > 0$. If $\rho_1 < \rho_2$ with $h_u(\rho_1) > 0$, then $v = c(\rho_1)\, u$ and $v = c(\rho_2)\, u$ a.e.\! in $B_{\rho_1}$, where $u$ does not vanish identically, so $c(\rho_1) = c(\rho_2)$. So $c(\rho) \equiv c$ for some constant $c \ge 0$ and $v = cu$ a.e.\! in $\R^N$. Conversely, if $v = cu$ with $c \ge 0$, then both sides equal $c \norm[L^{p,q}_b(\R^N)]{u}^{pq}$.
\end{proof}

Theorem \ref{thm:bm-pairing} together with Fubini's theorem gives
\begin{multline} \label{107}
\int_{\R^N} \Phi_u(|x|)\, |u|^{p-2}\, uv\, dx = a \int_{\R^N} \bigg(\int_{|x|}^\infty \rho^{-b}\, h_u^{q-1}(\rho)\, d\rho\bigg) |u|^{p-2}\, uv\, dx\\[7.5pt]
= a \int_0^\infty \rho^{-b}\, h_u^{q-1}(\rho) \bigg(\int_{B_\rho} |u|^{p-2}\, uv\, dx\bigg) d\rho,
\end{multline}
in particular,
\begin{equation} \label{73}
\int_{\R^N} \Phi_u(|x|)\, |u|^p\, dx = a \int_0^\infty \rho^{-b}\, h_u^q(\rho)\, d\rho.
\end{equation}
So we have the following corollary.

\begin{corollary} \label{cor:bm-pairing}
If \eqref{5} holds, then
\[
\int_{\R^N} \Phi_u(|x|)\, |u|^{p-2}\, uv\, dx \le a \norm[L^{p,q}_b(\R^N)]{u}^{pq-1} \norm[L^{p,q}_b(\R^N)]{v} \quad \forall u, v \in L^{p,q}_b(\R^N).
\]
For $u \ne 0$, equality holds if and only if $v = cu$ for some constant $c \ge 0$.
\end{corollary}

\subsection{Ball-mass Br\'{e}zis-Lieb lemma}

We have the following analog of the Br\'{e}zis-Lieb lemma \cite{MR699419} for $L^{p,q}_b(\R^N)$.

\begin{theorem} \label{thm:bm-brezis-lieb}
If \eqref{5} holds and $\seq{u_j}$ is a bounded sequence in $L^{p,q}_b(\R^N)$ such that $u_j \to u$ in $L^p_\loc(\R^N)$, then $u \in L^{p,q}_b(\R^N)$ and
\begin{equation} \label{16}
\norm[L^{p,q}_b(\R^N)]{u_j}^{pq} - \norm[L^{p,q}_b(\R^N)]{u_j - u}^{pq} \to \norm[L^{p,q}_b(\R^N)]{u}^{pq}.
\end{equation}
\end{theorem}

\begin{proof}
Since $u_j \to u$ in $L^p_\loc(\R^N)$,
\[
\norm[L^p(B_\rho)]{u_j} \to \norm[L^p(B_\rho)]{u}
\]
for every $\rho \in (0,\infty)$. Applying Fatou's lemma in $L^{pq}((0,\infty),\,\rho^{-b}\, d\rho)$ gives
\begin{multline*}
\norm[L^{p,q}_b(\R^N)]{u} = \big\|\!\norm[L^p(B_\rho)]{u}\!\big\|_{L^{pq}((0,\infty),\,\rho^{-b}\, d\rho)} \le \liminf_{j \to \infty}\, \big\|\!\norm[L^p(B_\rho)]{u_j}\!\big\|_{L^{pq}((0,\infty),\,\rho^{-b}\, d\rho)}\\[7.5pt]
= \liminf_{j \to \infty}\halfthin \norm[L^{p,q}_b(\R^N)]{u_j} < \infty
\end{multline*}
since $\seq{u_j}$ is bounded in $L^{p,q}_b(\R^N)$, so $u \in L^{p,q}_b(\R^N)$.

Since $\seq{u_j}$ is bounded in $L^{p,q}_b(\R^N)$ and the map $\Lambda$ in \eqref{510} is an isometry, $\seq{\Lambda u_j}$ is bounded in $X$. Since $u_j \to u$ in $L^p_\loc(\R^N)$,
\[
\norm[L^p(\R^N)]{\Lambda u_j(\cdot,\rho) - \Lambda u(\cdot,\rho)} = \norm[L^p(B_\rho)]{u_j - u} \to 0
\]
for every $\rho \in (0,\infty)$. Applying the Br\'{e}zis-Lieb lemma in $X$ gives
\[
\norm[X]{\Lambda u_j}^{pq} - \norm[X]{\Lambda u_j - \Lambda u}^{pq} \to \norm[X]{\Lambda u}^{pq},
\]
which is equivalent to \eqref{16} since $\Lambda$ is a linear isometry.
\end{proof}

\subsection{Ball-mass Sobolev space}

For $p$, $q$, and $b$ satisfying \eqref{33}--\eqref{2}, we define the ball-mass Sobolev space $E^{p,q}_b(\R^N)$ to be the space of all weakly differentiable functions $u : \R^N \to \R$ such that
\[
\norm[L^2(\R^N)]{\nabla u} + \norm[L^{p,q}_b(\R^N)]{u} < \infty.
\]

\begin{theorem} \label{Theorem 4}
If \eqref{33}--\eqref{2} hold, then $E^{p,q}_b(\R^N)$ equipped with the norm
\[
\norm[E^{p,q}_b(\R^N)]{u} = \left(\norm[L^2(\R^N)]{\nabla u}^{2pq} + \norm[L^{p,q}_b(\R^N)]{u}^{2pq}\right)^{1/2pq}
\]
is a uniformly convex Banach space.
\end{theorem}

\begin{proof}
The map
\[
\Upsilon : E^{p,q}_b(\R^N) \to L^2(\R^N;\R^N) \oplus_{2pq} L^{p,q}_b(\R^N), \quad \Upsilon u = (\nabla u,u)
\]
is a linear isometry. The $\ell^{2pq}$-direct sum $L^2(\R^N;\R^N) \oplus_{2pq} L^{p,q}_b(\R^N)$ is uniformly convex in view of Theorem \ref{Theorem 2} since $2pq > 1$. It follows that $\norm[E^{p,q}_b(\R^N)]{\cdot}$ is a uniformly convex norm.

Now we show that $\Upsilon(E^{p,q}_b(\R^N))$ is closed and hence $E^{p,q}_b(\R^N)$ is complete. Suppose $\Upsilon u_j = (\nabla u_j,u_j) \to (V,w)$ in $L^2(\R^N;\R^N) \oplus_{2pq} L^{p,q}_b(\R^N)$. Then $u_j \to w$ in $L^{p,q}_b(\R^N)$ and hence also in $L^p_\loc(\R^N)$ by Theorem \ref{Theorem 3}, while $\nabla u_j \to V$ in $L^2(\R^N;\R^N)$, so $V = \nabla w$ a.e.\! in $\R^N$. Hence $w \in E^{p,q}_b(\R^N)$ with $\Upsilon w = (V,w)$, so $(V,w) \in \Upsilon(E^{p,q}_b(\R^N))$.
\end{proof}

We have the following local embedding theorem for $E^{p,q}_b(\R^N)$.

\begin{theorem} \label{Theorem 5}
If \eqref{33}--\eqref{2} hold, then
\[
\norm[H^1(B_R)]{u} \le C_R \norm[E^{p,q}_b(\R^N)]{u} \quad \forall u \in E^{p,q}_b(\R^N)
\]
for some constant $C_R > 0$. In particular, $E^{p,q}_b(\R^N)$ is continuously embedded in $H^1_\loc(\R^N)$.
\end{theorem}

\begin{proof}
Since $\norm[L^2(\R^N)]{\nabla u} \le \norm[E^{p,q}_b(\R^N)]{u}$, it only remains to show that
\[
\norm[L^2(B_R)]{u} \le C_R \norm[E^{p,q}_b(\R^N)]{u} \quad \forall u \in E^{p,q}_b(\R^N).
\]
In view of Theorem \ref{Theorem 3}, this follows from the H\"{o}lder inequality when $p \ge 2$ and from the Gagliardo-Nirenberg inequality
\[
\norm[L^2(B_R)]{u} \le C_R \left(\norm[L^2(B_R)]{\nabla u}^\theta \norm[L^p(B_R)]{u}^{1 - \theta} + \norm[L^p(B_R)]{u}\right),
\]
where $\theta \in (0,1)$ is given by $1/2 = \theta\, (1/2 - 1/N) + (1 - \theta)/p$, when $p < 2$.
\end{proof}

\subsection{Differentiability of the ball-mass potential}

The following proposition justifies the variational formulation of equation \eqref{1} given in Subsection \ref{ssec:var-form}.

\begin{proposition} \label{prop:c1}
The functional
\[
\mathcal{P}(u) = \int_0^\infty \rho^{-b}\, h_u^q(\rho)\, d\rho, \quad u \in E^{p,q}_{b,\rad}(\R^N)
\]
is of class $C^1$ with
\[
\mathcal{P}'(u) v = \frac{pq}{a} \int_{\R^N} \Phi_u(|x|)\, |u|^{p-2}\, uv\, dx, \quad u, v \in E^{p,q}_{b,\rad}(\R^N).
\]
\end{proposition}

\begin{proof}
Since $L^{p,q}_b(\R^N)$ is uniformly smooth by Theorem \ref{Theorem 2}, its norm is of class $C^1$ away from the origin. Since $pq > 1$ by \eqref{5}, it follows that the functional
\[
\psi(u) = \norm[L^{p,q}_b(\R^N)]{u}^{pq}, \quad u \in L^{p,q}_b(\R^N)
\]
is of class $C^1$ with $\psi'(0) = 0$. Since the inclusion $\iota : E^{p,q}_{b,\rad}(\R^N) \hookrightarrow L^{p,q}_b(\R^N)$ is a bounded linear mapping and $\mathcal{P} = \psi \circ \iota$, the chain rule now shows that $\mathcal{P}$ is of class $C^1$.

For $u, v \in E^{p,q}_{b,\rad}(\R^N)$ and $t \ne 0$,
\begin{equation} \label{120}
\frac{\mathcal{P}(u + tv) - \mathcal{P}(u)}{t} = \int_0^\infty \rho^{-b} \left(\frac{h_{u+tv}^q(\rho) - h_u^q(\rho)}{t}\right) d\rho.
\end{equation}
For each $\rho > 0$, the mapping $t \mapsto h_{u+tv}(\rho)$ is $C^1$ with derivative $p\, P_\rho(u + tv,v)$, where
\[
P_\rho(u,v) = \int_{B_\rho} |u|^{p-2}\, uv\, dx,
\]
so the integrand in \eqref{120} is equal to
\[
\frac{pq\halfthin \rho^{-b}}{t} \int_0^t h_{u+sv}^{q-1}(\rho)\, P_\rho(u + sv,v)\, ds \to pq\halfthin \rho^{-b}\, h_u^{q-1}(\rho)\, P_\rho(u,v) \quad \text{as } t \to 0.
\]
For $|t| \le 1$ and $|s| \le |t|$, we have $|u + sv| \le |u| + |v| =: w$ and $|v| \le w$, so $h_{u+sv} \le h_w$ and $h_v \le h_w$. So the H\"{o}lder inequality gives
\[
|h_{u+sv}^{q-1}(\rho)\, P_\rho(u + sv,v)| \le h_{u+sv}^{q-1}(\rho)\, h_{u+sv}^{1-1/p}(\rho)\, h_v^{1/p}(\rho) \le h_w^q(\rho)
\]
and hence
\[
\abs{\frac{\rho^{-b}}{t} \int_0^t h_{u+sv}^{q-1}(\rho)\, P_\rho(u + sv,v)\, ds} \le \rho^{-b}\, h_w^q(\rho) \in L^1((0,\infty))
\]
since $w \in E^{p,q}_{b,\rad}(\R^N)$. We may now pass to the limit in \eqref{120} using the dominated convergence theorem to get
\[
\mathcal{P}'(u) v = pq \int_0^\infty \rho^{-b}\, h_u^{q-1}(\rho)\, P_\rho(u,v)\, d\rho = \frac{pq}{a} \int_{\R^N} \Phi_u(|x|)\, |u|^{p-2}\, uv\, dx
\]
by \eqref{107}.
\end{proof}

\subsection{Sharp embedding results}

In this subsection we consider embeddings of $E^{p,q}_b(\R^N)$ in Lebesgue spaces. We begin by showing that without any restrictions there are no embeddings for $N = 2$ and only the Sobolev embedding in $L^{2^\ast}(\R^N)$ for $N \ge 3$.

\begin{theorem} \label{thm:no-embedding}
Assume \eqref{33}--\eqref{2}. Then $E^{p,q}_b(\R^N)$ is continuously embedded in $L^s(\R^N)$ if and only if $N \ge 3$ and $s = 2^\ast$, and this embedding is not compact.
\end{theorem}

\begin{proof}
Suppose
\begin{equation} \label{43}
\norm[L^s(\R^N)]{u} \le C \norm[E^{p,q}_b(\R^N)]{u} \quad \forall u \in E^{p,q}_b(\R^N)
\end{equation}
for some constant $C > 0$. Let $u \in C^\infty_c(\R^N)$, let $R > 0$, and take $y \in \R^N$ such that $u_y := u(\cdot - y) \equiv 0$ in $B_R$. Then $\norm[L^s(\R^N)]{u_y} = \norm[L^s(\R^N)]{u}$ and $\norm[L^2(\R^N)]{\nabla u_y} = \norm[L^2(\R^N)]{\nabla u}$. On the other hand, $h_{u_y}(\rho) = 0$ for $\rho \le R$ and $h_{u_y}(\rho) \le \norm[L^p(\R^N)]{u}^p$ for all $\rho > 0$, so
\[
\norm[L^{p,q}_b(\R^N)]{u_y} \le \norm[L^p(\R^N)]{u} \left(\int_R^\infty \rho^{-b}\, d\rho\right)^{1/pq} = \frac{R^{-(b-1)/pq}}{(b - 1)^{1/pq}} \norm[L^p(\R^N)]{u} \to 0 \quad \text{as } R \to \infty
\]
since $b > 1$. So applying \eqref{43} to $u_y$ and letting $R \to \infty$ gives
\begin{equation} \label{44}
\norm[L^s(\R^N)]{u} \le C \norm[L^2(\R^N)]{\nabla u} \quad \forall u \in C^\infty_c(\R^N).
\end{equation}
Fixing $u \in C^\infty_c(\R^N) \setminus \set{0}$ and applying this to $u_\lambda := u(\lambda\, \cdot)$ gives
\[
\lambda^{-N/s} \norm[L^s(\R^N)]{u} \le C \lambda^{-(N-2)/2} \norm[L^2(\R^N)]{\nabla u} \quad \forall \lambda > 0,
\]
which forces $N = 2$ and $s = \infty$, or $N \ge 3$ and $s = 2^\ast$. The standard Moser sequence in $\R^2$ violates \eqref{44} and hence rules out the former.

Conversely, if $N \ge 3$, then the Sobolev inequality gives
\[
\norm[L^{2^\ast}(\R^N)]{u} \le C \norm[L^2(\R^N)]{\nabla u} \le C \norm[E^{p,q}_b(\R^N)]{u} \quad \forall u \in E^{p,q}_b(\R^N)
\]
for some constant $C > 0$.

This embedding is not compact since for any $u \in C^\infty_c(\R^N) \setminus \set{0}$ and $\seq{y_j} \subset \R^N$ with $|y_j| \to \infty$, $u_j := u_{y_j} \wto 0$ in $E^{p,q}_b(\R^N)$ since $\nabla u_j \wto 0$ in $L^2(\R^N;\R^N)$ and $\norm[L^{p,q}_b(\R^N)]{u_j} \to 0$, but $\|u_{y_j}\|_{L^{2^\ast}(\R^N)} = \norm[L^{2^\ast}(\R^N)]{u} \not\to 0$.
\end{proof}

However, restricting to radial functions gives compact embeddings into a range of Lebesgue spaces. The main result of this section is the following embedding theorem for $E^{p,q}_{b,\rad}(\R^N)$, where $2^{p,q}_{b,\ast}$ is given in \eqref{30}. Compactness of radial embeddings of this type goes back to Strauss \cite{MR0454365} and Berestycki and Lions \cite{MR695535}. Here the gradient term alone does not give compactness. It must be used together with the ball-mass term.

\begin{theorem} \label{Theorem 1}
Assume \eqref{33}--\eqref{2}. Then $E^{p,q}_{b,\rad}(\R^2)$ is compactly embedded in $L^s(\R^2)$ for all $s \in (2^{p,q}_{b,\ast},\infty)$. For $N \ge 3$, $E^{p,q}_{b,\rad}(\R^N)$ is embedded in $L^s(\R^N)$ continuously for all $s \in (2^{p,q}_{b,\ast},2^\ast]$ and compactly for $s \in (2^{p,q}_{b,\ast},2^\ast)$.
\end{theorem}

Set
\[
2^{p,q}_{b,\#} = \frac{Npq}{Nq + 1 - b} = \frac{Npq}{b^\# - b}.
\]
We have
\[
2^{p,q}_{b,\ast} - 2^{p,q}_{b,\#} = \begin{cases}
\dfrac{2\, (b - 1)}{b_\ast - b} & \text{if } N = 2\\[15pt]
\dfrac{2\, (b - 1)(b_\ast - b)}{(b^\# - b)(b^\ast - b)} & \text{if } N \ge 3,
\end{cases}
\]
so
\begin{equation} \label{106}
p < 2^{p,q}_{b,\#} < 2^{p,q}_{b,\ast}
\end{equation}
by \eqref{2} and \eqref{36}. We begin by proving the following exterior Gagliardo-Nirenberg type inequality.

\begin{lemma} \label{Lemma 1}
If \eqref{33}--\eqref{2} hold and $s \in (2^{p,q}_{b,\ast},\infty)$, then there exists a constant $C > 0$ such that
\[
\norm[L^s(\R^N \setminus B_R)]{u} \le C \left(R^{- \gamma_\ast} \norm[L^2(\R^N)]{\nabla u}^\theta \norm[L^{p,q}_b(\R^N)]{u}^{1 - \theta} + R^{- \gamma_\#} \norm[L^{p,q}_b(\R^N)]{u}\right) \quad \forall u \in E^{p,q}_{b,\rad}(\R^N)
\]
for all $R > 0$, where
\begin{gather}
\label{6} \theta = \frac{2\, (s - p)}{(p + 2)\, s} \in (0,1),\\[10pt]
\notag \gamma_\ast = \frac{(b^\ast - b)(s - 2^{p,q}_{b,\ast})}{(p + 2)\, qs} > 0, \qquad \gamma_\# = \frac{(b^\# - b)(s - 2^{p,q}_{b,\#})}{pqs} > 0.
\end{gather}
\end{lemma}

\begin{proof}
Setting $v(x) = u(Rx)$ gives
\begin{gather*}
\norm[L^s(\R^N \setminus B_1)]{v} = R^{-N/s} \norm[L^s(\R^N \setminus B_R)]{u},\\[10pt]
\norm[L^2(\R^N)]{\nabla v} = R^{-(N-2)/2} \norm[L^2(\R^N)]{\nabla u},\\[10pt]
\norm[L^{p,q}_b(\R^N)]{v} = R^{- (b^\# - b)/pq} \norm[L^{p,q}_b(\R^N)]{u},
\end{gather*}
so it suffices to prove the inequality for $R = 1$.

We use the dyadic annular decomposition
\[
\R^N \setminus B_1 = \bigcup_{j \ge 0} A_j,
\]
where $A_j = B_{2^{j+1}} \setminus B_{2^j}$. We have
\begin{equation} \label{7}
\norm[L^s(\R^N \setminus B_1)]{u} = \left(\sum_{j \ge 0} \norm[L^s(A_j)]{u}^s\right)^{1/s} \le \sum_{j \ge 0} \norm[L^s(A_j)]{u}
\end{equation}
since $s > 1$.

Write $r = |x|$ and set
\[
v_j(r) = u(2^j\halfthin r), \quad 1 \le r < 2.
\]
By the $1$-dimensional Gagliardo-Nirenberg inequality in a bounded domain, there exists a constant $C > 0$ such that
\begin{equation} \label{8}
\norm[L^s(1,2)]{v_j} \le C \left(\norm[L^2(1,2)]{v_j'}^\theta \norm[L^p(1,2)]{v_j}^{1 - \theta} + \norm[L^p(1,2)]{v_j}\right) \quad \forall j \ge 0
\end{equation}
with $\theta$ as in \eqref{6}. We have
\begin{align*}
\int_{A_j} |u|^s\, dx & = \omega_{N-1}\, 2^{Nj} \int_1^2 r^{N-1}\, |v_j(r)|^s\, dr,\\[7.5pt]
D_j := \int_{A_j} |\nabla u|^2\, dx & = \omega_{N-1}\, 2^{(N - 2)\, j} \int_1^2 r^{N-1}\, |v_j'(r)|^2\, dr,\\[7.5pt]
m_j := \int_{A_j} |u|^p\, dx & = \omega_{N-1}\, 2^{Nj} \int_1^2 r^{N-1}\, |v_j(r)|^p\, dr,
\end{align*}
so
\begin{gather*}
\norm[L^s(1,2)]{v_j} \asymp 2^{- Nj/s} \norm[L^s(A_j)]{u},\\[10pt]
\norm[L^2(1,2)]{v_j'} \asymp 2^{- (N - 2)\, j/2}\halfthin D_j^{1/2},\\[10pt]
\norm[L^p(1,2)]{v_j} \asymp 2^{- Nj/p}\, m_j^{1/p}.
\end{gather*}
Combining \eqref{8} with these estimates and \eqref{6} gives
\begin{equation} \label{9}
\norm[L^s(A_j)]{u} \le C \left(2^{- \theta\, (N - 1)\, j}\halfthin D_j^{\theta/2}\, m_j^{(1 - \theta)/p} + 2^{- (1/p - 1/s)\, Nj}\, m_j^{1/p}\right) \quad \forall j \ge 0.
\end{equation}

We have
\[
\norm[L^2(\R^N)]{\nabla u}^2 \ge \sum_{j \ge 0} \norm[L^2(A_j)]{\nabla u}^2 = \sum_{j \ge 0} D_j
\]
and
\[
\norm[L^{p,q}_b(\R^N)]{u}^{pq} \ge \sum_{j \ge 0} \int_{2^{j+1}}^{2^{j+2}} \rho^{-b}\, h_u^q(\rho)\, d\rho \ge \sum_{j \ge 0} \left(\int_{A_j} |u|^p\, dx\right)^q \int_{2^{j+1}}^{2^{j+2}} \rho^{-b}\, d\rho = c_b \sum_{j \ge 0} \mu_j, 
\]
where $\mu_j = 2^{- (b - 1)\, j}\, m_j^q$ and $c_b > 0$ is a constant. So
\begin{multline}
\sum_{j \ge 0} 2^{- \theta\, (N - 1)\, j}\halfthin D_j^{\theta/2}\, m_j^{(1 - \theta)/p} = \sum_{j \ge 0} 2^{- \gamma_\ast j}\halfthin D_j^{\theta/2}\, \mu_j^{(1 - \theta)/pq}\\[7.5pt]
\le \left(\sum_{j \ge 0} 2^{- \gamma_\ast \sigma j}\right)^{1/\sigma}\! \left(\sum_{j \ge 0} D_j\right)^{\theta/2}\! \left(\sum_{j \ge 0} \mu_j\right)^{(1 - \theta)/pq} \le C \norm[L^2(\R^N)]{\nabla u}^\theta \norm[L^{p,q}_b(\R^N)]{u}^{1 - \theta},
\end{multline}
where $\sigma \in (1,2)$ is given by $1/\sigma + \theta/2 + (1 - \theta)/pq = 1$, and
\begin{equation} \label{10}
\sum_{j \ge 0} 2^{- (1/p - 1/s)\, Nj}\, m_j^{1/p} = \sum_{j \ge 0} 2^{- \gamma_\# j}\halfthin \mu_j^{1/pq} \le \left(\sum_{j \ge 0} 2^{- \gamma_\# \tau j}\right)^{1/\tau}\! \left(\sum_{j \ge 0} \mu_j\right)^{1/pq} \le C \norm[L^{p,q}_b(\R^N)]{u},
\end{equation}
where $\tau \in (1,2)$ is given by $1/\tau + 1/pq = 1$.

Combining \eqref{7} with \eqref{9}--\eqref{10} gives the desired inequality.
\end{proof}

We are now ready to prove Theorem \ref{Theorem 1}.

\begin{proof}[Proof of Theorem \ref{Theorem 1}]
Continuity of the embedding for all $s \in (2^{p,q}_{b,\ast},2^\ast]$ when $N \ge 3$ is immediate from Theorem \ref{Theorem 5} and Lemma \ref{Lemma 1}. Let
\[
2^{p,q}_{b,\ast} < s < \begin{cases}
\infty & \text{if } N = 2\\[5pt]
2^\ast & \text{if } N \ge 3
\end{cases}
\]
and let $\seq{u_j}$ be a bounded sequence in $E^{p,q}_{b,\rad}(\R^N)$. Since $E^{p,q}_{b,\rad}(\R^N)$ is reflexive by Theorem \ref{Theorem 4}, we may assume that $u_j \wto u$ in $E^{p,q}_{b,\rad}(\R^N)$ for a renamed subsequence. Set $v_j = u_j - u$. Then $v_j \wto 0$ and $\seq{v_j}$ is bounded in $E^{p,q}_{b,\rad}(\R^N)$. For any $R > 0$,
\[
\norm[L^s(\R^N)]{v_j}^s = \norm[L^s(B_R)]{v_j}^s + \norm[L^s(\R^N \setminus B_R)]{v_j}^s.
\]
In view of Theorem \ref{Theorem 5}, $v_j \wto 0$ in $H^1(B_R)$ and hence $v_j \to 0$ in $L^s(B_R)$ by the Rellich-Kondrachov theorem. On the other hand, $\norm[L^s(\R^N \setminus B_R)]{v_j} \to 0$ as $R \to \infty$, uniformly in $j$, by Lemma \ref{Lemma 1}. It follows that $v_j \to 0$ in $L^s(\R^N)$.
\end{proof}

Our next theorem shows that the exponent $2^{p,q}_{b,\ast}$ in Theorem \ref{Theorem 1} is sharp.

\begin{theorem} \label{thm:emb-sharp}
Assume \eqref{33}--\eqref{2}. Then $E^{p,q}_{b,\rad}(\R^N)$ is not continuously embedded in $L^s(\R^N)$ for any $s < 2^{p,q}_{b,\ast}$.
\end{theorem}

\begin{proof}
We will show that if
\begin{equation} \label{45}
\norm[L^s(\R^N)]{u} \le C \norm[E^{p,q}_b(\R^N)]{u} \quad \forall u \in E^{p,q}_{b,\rad}(\R^N)
\end{equation}
for some constant $C > 0$, then $s \ge 2^{p,q}_{b,\ast}$. Let $\psi : \R \to [0,1]$ be a smooth function such that $\supp \psi \subset [0,1]$, $\psi \equiv 1$ on $[1/3,2/3]$, and $|\psi'| \le 4$. For $R \ge 1$ and $0 < \delta \le R$, set
\[
u_{R,\delta}(x) = \psi\bigg(\frac{|x| - R}{\delta}\bigg), \quad x \in \R^N.
\]
Then $u_{R,\delta} \in E^{p,q}_{b,\rad}(\R^N)$ is a radial function supported in the annulus $R \le |x| \le R + \delta$. Since $u_{R,\delta} \equiv 1$ in $R + \delta/3 \le |x| \le R + 2 \delta/3$,
\begin{equation} \label{46}
\norm[L^s(\R^N)]{u_{R,\delta}}^s \ge \omega_{N-1} \int_{R + \delta/3}^{R + 2 \delta/3} r^{N-1}\, dr \ge \frac{1}{3}\, \omega_{N-1} R^{N-1} \delta.
\end{equation}
Since $|\nabla u_{R,\delta}| = |\psi'/\delta| \le 4/\delta$ and $R + \delta \le 2R$,
\begin{equation}
\norm[L^2(\R^N)]{\nabla u_{R,\delta}}^2 \le 16\, \omega_{N-1}\halfthin \delta^{-2} \int_R^{R + \delta} r^{N-1}\, dr \le CR^{N-1} \delta^{-1}.
\end{equation}
Since $h_{u_{R,\delta}}(\rho) = 0$ for $\rho \le R$ and
\[
h_{u_{R,\delta}}(\rho) \le \norm[L^p(\R^N)]{u_{R,\delta}}^p \le \omega_{N-1} \int_R^{R + \delta} r^{N-1}\, dr \le CR^{N-1} \delta
\]
for all $\rho > 0$,
\begin{equation} \label{47}
\norm[L^{p,q}_b(\R^N)]{u_{R,\delta}}^{pq} \le CR^{(N - 1)\, q}\, \delta^q \int_R^\infty \rho^{-b}\, d\rho = CR^{(N - 1)\, q + 1 - b}\, \delta^q.
\end{equation}

Now take $\delta = R^{1 - \kappa}$ with $\kappa > 0$ as in \eqref{48} and note that $\delta \le R$ since $R \ge 1$. Inequality \eqref{45} together with the estimates \eqref{46}--\eqref{47} and this choice of $\delta$ leads to
\[
R^{(N + \kappa - 2)(2^{p,q}_{b,\ast} - s)/2s} \le C \quad \forall R \ge 1,
\]
which forces $s \ge 2^{p,q}_{b,\ast}$.
\end{proof}

\subsection{Weak continuity of the ball-mass operator} \label{ssec:weak-cont}

Our next theorem shows that the ball-mass operator is weakly continuous.

\begin{theorem} \label{thm:weak-cont}
Assume \eqref{33}--\eqref{2}. If $u_j \wto u$ in $E^{p,q}_{b,\rad}(\R^N)$, then
\[
\int_{\R^N} \nabla u_j \cdot \nabla v\, dx + \int_{\R^N} \Phi_{u_j}(|x|)\, |u_j|^{p-2}\, u_j\halfthin v\, dx \to \int_{\R^N} \nabla u \cdot \nabla v\, dx + \int_{\R^N} \Phi_u(|x|)\, |u|^{p-2}\, uv\, dx
\]
for all $v \in E^{p,q}_{b,\rad}(\R^N)$.
\end{theorem}

\begin{proof}
The mapping
\[
E^{p,q}_{b,\rad}(\R^N) \to \R, \quad u \mapsto \int_{\R^N} \nabla u \cdot \nabla v\, dx
\]
is a bounded linear functional on $E^{p,q}_{b,\rad}(\R^N)$, so weak convergence implies
\[
\int_{\R^N} \nabla u_j \cdot \nabla v\, dx \to \int_{\R^N} \nabla u \cdot \nabla v\, dx.
\]

By \eqref{107},
\[
\int_{\R^N} \Phi_{u_j}(|x|)\, |u_j|^{p-2}\, u_j\halfthin v\, dx = a \int_0^\infty \rho^{-b}\, h_{u_j}^{q-1}(\rho)\, P_\rho(u_j,v)\, d\rho
\]
and
\[
\int_{\R^N} \Phi_u(|x|)\, |u|^{p-2}\, uv\, dx = a \int_0^\infty \rho^{-b}\, h_u^{q-1}(\rho)\, P_\rho(u,v)\, d\rho,
\]
where
\[
P_\rho(u,v) = \int_{B_\rho} |u|^{p-2}\, uv\, dx.
\]
By the H\"{o}lder inequality,
\begin{equation} \label{108}
|P_\rho(u,v)| \le h_u^{1-1/p}(\rho)\, h_v^{1/p}(\rho).
\end{equation}
By Theorem \ref{Theorem 1}, \eqref{106}, and the H\"{o}lder inequality, $u_j \to u$ in $L^p_\loc(\R^N)$ and hence
\[
\rho^{-b}\, h_{u_j}^{q-1}(\rho)\, P_\rho(u_j,v) \to \rho^{-b}\, h_u^{q-1}(\rho)\, P_\rho(u,v)
\]
for every $\rho \in (0,\infty)$. By \eqref{108} and Theorem \ref{Theorem 3},
\[
\abs{\rho^{-b}\, h_{u_j}^{q-1}(\rho)\, P_\rho(u_j,v)} \le \rho^{-b}\, h_{u_j}^{q-1/p}(\rho)\, h_v^{1/p}(\rho) \le C \rho^{-1} \norm[L^{p,q}_b(\R^N)]{u_j}^{pq-1} \norm[L^{p,q}_b(\R^N)]{v} \le C \rho^{-1}.
\]
So
\[
\int_\delta^R \rho^{-b}\, h_{u_j}^{q-1}(\rho)\, P_\rho(u_j,v)\, d\rho \to \int_\delta^R \rho^{-b}\, h_u^{q-1}(\rho)\, P_\rho(u,v)\, d\rho
\]
for any $R > \delta > 0$ by the dominated convergence theorem. On the other hand, setting $I = (0,\delta) \cup (R,\infty)$, using \eqref{108}, and applying the H\"{o}lder inequality gives
\begin{multline*}
\abs{\int_I \rho^{-b}\, h_{u_j}^{q-1}(\rho)\, P_\rho(u_j,v)\, d\rho} \le \int_I \big(\rho^{-b}\, h_{u_j}^q(\rho)\big)^{1-1/pq} \left(\rho^{-b}\, h_v^q(\rho)\right)^{1/pq} d\rho\\[7.5pt]
\le \norm[L^{p,q}_b(\R^N)]{u_j}^{pq-1} \left(\int_I \rho^{-b}\, h_v^q(\rho)\, d\rho\right)^{1/pq} \le C \left(\int_I \rho^{-b}\, h_v^q(\rho)\, d\rho\right)^{1/pq} \to 0
\end{multline*}
uniformly in $j$ as $\delta \to 0$ and $R \to \infty$ since
\[
\int_0^\infty \rho^{-b}\, h_v^q(\rho)\, d\rho = \norm[L^{p,q}_b(\R^N)]{v}^{pq} < \infty.
\]
Similarly,
\[
\int_I \rho^{-b}\, h_u^{q-1}(\rho)\, P_\rho(u,v)\, d\rho \to 0
\]
as $\delta \to 0$ and $R \to \infty$. Hence
\[
\int_{\R^N} \Phi_{u_j}(|x|)\, |u_j|^{p-2}\, u_j\halfthin v\, dx \to \int_{\R^N} \Phi_u(|x|)\, |u|^{p-2}\, uv\, dx. \QED
\]
\end{proof}

\subsection{Radial decay estimate}

Set
\[
\zeta = \frac{b^\ast - b}{q}
\]
and note that $\zeta > 0$ by \eqref{2} and \eqref{36}. We have the following analog of the radial lemma of Strauss \cite{MR0454365} (see also \cite{MR695535}).

\begin{theorem} \label{Theorem 6}
If \eqref{33}--\eqref{2} hold, then there exists a constant $C > 0$ such that
\[
|u(r)|^{p+2} \le Cr^{- \zeta} \norm[L^2(\R^N)]{\nabla u}^2 \norm[L^{p,q}_b(\R^N)]{u}^p \quad \forall u \in E^{p,q}_{b,\rad}(\R^N)
\]
for all $r > 0$. In particular, for any $R > 0$,
\[
\norm[L^\infty(\R^N \setminus B_R)]{u} \le C_R \norm[E^{p,q}_b(\R^N)]{u} \quad \forall u \in E^{p,q}_{b,\rad}(\R^N)
\]
for some constant $C_R > 0$.
\end{theorem}

\begin{proof}
Theorem \ref{Theorem 1} gives a sequence $r_j \to \infty$ such that $u(r_j) \to 0$. Integrating from $r$ to $r_j$ and passing to the limit gives
\begin{multline*}
|u(r)|^{p/2 + 1} = - \int_r^\infty \frac{d}{d\rho}\, |u(\rho)|^{p/2 + 1}\, d\rho = - \left(\frac{p}{2} + 1\right) \int_r^\infty |u(\rho)|^{p/2 - 1}\, u(\rho)\, u'(\rho)\, d\rho\\[7.5pt]
\le C \left(\int_r^\infty \omega_{N-1}\, \rho^{N-1}\, |u'(\rho)|^2\, d\rho\right)^{1/2}\! \left(\int_r^\infty \frac{|u(\rho)|^p}{\rho^{N-1}}\, d\rho\right)^{1/2}
\end{multline*}
by the Cauchy-Schwarz inequality, so
\begin{equation} \label{11}
|u(r)|^{p+2} \le C \norm[L^2(\R^N)]{\nabla u}^2 \int_r^\infty \frac{|u(\rho)|^p}{\rho^{N-1}}\, d\rho.
\end{equation}
Since
\[
h_u(\rho) = \omega_{N-1} \int_0^\rho \tau^{N-1}\, |u(\tau)|^p\, d\tau,
\]
$h_u'(\rho) = \omega_{N-1}\, \rho^{N-1}\, |u(\rho)|^p$ and hence
\begin{equation}
\int_r^\infty \frac{|u(\rho)|^p}{\rho^{N-1}}\, d\rho = \frac{1}{\omega_{N-1}} \int_r^\infty \frac{h_u'(\rho)}{\rho^{2\, (N - 1)}}\, d\rho.
\end{equation}
By Theorem \ref{Theorem 3},
\[
h_u(\rho) \le C \rho^{(b-1)/q} \norm[L^{p,q}_b(\R^N)]{u}^p,
\]
in particular,
\[
\frac{h_u(\rho)}{\rho^{2\, (N - 1)}} \le C \rho^{- \zeta} \norm[L^{p,q}_b(\R^N)]{u}^p \to 0 \quad \text{as } \rho \to \infty.
\]
So integrating by parts gives
\begin{multline} \label{12}
\int_r^\infty \frac{h_u'(\rho)}{\rho^{2\, (N - 1)}}\, d\rho = - \frac{h_u(r)}{r^{2\, (N - 1)}} + 2\, (N - 1) \int_r^\infty \frac{h_u(\rho)}{\rho^{2N-1}}\, d\rho \le C \norm[L^{p,q}_b(\R^N)]{u}^p \int_r^\infty \rho^{- \zeta - 1}\, d\rho\\[7.5pt]
= Cr^{- \zeta} \norm[L^{p,q}_b(\R^N)]{u}^p.
\end{multline}
Combining \eqref{11}--\eqref{12} gives the desired inequality.
\end{proof}

\subsection{Ball-mass Trudinger-Moser inequality}

In this subsection we prove an analog of the Trudinger-Moser inequality \cite{MR0301504,MR0216286} for $E^{p,q}_{b,\rad}(\R^2)$. Since the domain is unbounded, the exponential term must be truncated as in \cite{MR1163431}. When $N = 2$, let $n$ be the smallest integer such that
\begin{equation} \label{200}
2n > 2^{p,q}_{b,\ast} = \frac{2\, (pq + b - 1)}{2q + 1 - b},
\end{equation}
let
\[
\psi(t) = e^t - \sum_{l=0}^{n-1} \frac{t^l}{l!} = \sum_{l \ge n} \frac{t^l}{l!},
\]
and set
\[
\Psi_\nu(u) = \int_{\R^2} \psi(\nu u^2)\, dx, \quad u \in E^{p,q}_{b,\rad}(\R^2)
\]
for $\nu > 0$. Let
\[
U = \bgset{u \in E^{p,q}_{b,\rad}(\R^2) : \norm[E^{p,q}_b(\R^2)]{u} \le 1}.
\]
We have the following Trudinger-Moser type inequality.

\begin{theorem} \label{Theorem 7}
If $N = 2$ and
\[
1 < p < \infty, \qquad 1 \le q < \infty, \qquad pq > 2, \qquad 1 < b < 1 + 2q,
\]
then for each $\nu \in (0,4 \pi)$,
\begin{equation} \label{201}
S_\nu := \sup_{u \in U}\, \Psi_\nu(u) < \infty
\end{equation}
and the supremum is attained at some $u$ with $\norm[E^{p,q}_b(\R^2)]{u} = 1$.
\end{theorem}

\begin{proof}
Let $u \in U$. We have
\begin{equation} \label{13}
\Psi_\nu(u) = \int_{B_1} \psi(\nu u^2)\, dx + \int_{\R^2 \setminus B_1} \psi(\nu u^2)\, dx.
\end{equation}
By Theorem \ref{Theorem 5}, $u \in H^1(B_1)$. The trace of $u$ on $\bdry{B_1}$ is $u(1)$ since $u$ is radial, so $v := u - u(1) \in H^1_0(B_1)$ with
\[
\int_{B_1} |\nabla v|^2\, dx = \int_{B_1} |\nabla u|^2\, dx \le \norm[E^{p,q}_b(\R^2)]{u}^2 \le 1.
\]
Set $\delta = 4 \pi/\nu - 1 > 0$. Since
\[
\nu u^2 \le (1 + \delta)\, \nu v^2 + (1 + \delta^{-1})\, \nu u(1)^2
\]
by Young's inequality and
\[
|u(1)| \le C \norm[L^2(\R^2)]{\nabla u}^{2/(p+2)} \norm[L^{p,q}_b(\R^2)]{u}^{p/(p+2)} \le C
\]
by Theorem \ref{Theorem 6},
\[
\nu u^2 \le 4 \pi v^2 + C.
\]
So
\begin{equation}
\int_{B_1} \psi(\nu u^2)\, dx \le \int_{B_1} e^{\nu u^2}\, dx \le C \int_{B_1} e^{4 \pi v^2}\, dx \le C
\end{equation}
by the Trudinger-Moser inequality on $H^1_0(B_1)$. Since $\psi(t) \le t^n e^t$ for all $t \ge 0$ and
\[
\norm[L^\infty(\R^2 \setminus B_1)]{u} \le C \norm[E^{p,q}_b(\R^2)]{u} \le C
\]
by Theorem \ref{Theorem 6},
\[
\psi(\nu u^2) \le \nu^n u^{2n} e^{\nu u^2} \le C u^{2n} \quad \text{in } \R^2 \setminus B_1.
\]
So
\begin{equation} \label{14}
\int_{\R^2 \setminus B_1} \psi(\nu u^2)\, dx \le C \norm[L^{2n}(\R^2 \setminus B_1)]{u}^{2n} \le C \norm[E^{p,q}_b(\R^2)]{u}^{2n} \le C
\end{equation}
by Theorem \ref{Theorem 1} since $2n > 2^{p,q}_{b,\ast}$. Combining \eqref{13}--\eqref{14} gives \eqref{201}.

Let $\seq{u_j} \subset U$ be a maximizing sequence. Since $U$ is bounded and $E^{p,q}_{b,\rad}(\R^2)$ is reflexive by Theorem \ref{Theorem 4}, a renamed subsequence of $\seq{u_j}$ converges weakly to some $u \in E^{p,q}_{b,\rad}(\R^2)$. By the weak lower semicontinuity of the norm, $u \in U$. We have 
\begin{equation} \label{202}
\Psi_\nu(u_j) = \int_{\R^2} \left(\sum_{l \ge n} \frac{\nu^l u_j^{2l}}{l!}\right) dx = \sum_{l \ge n} \frac{\nu^l}{l!} \norm[L^{2l}(\R^2)]{u_j}^{2l}
\end{equation}
by Fubini's theorem. Fix $\nu' \in (\nu,4 \pi)$. For $l \ge n$, \eqref{202} together with \eqref{201} gives
\[
\frac{(\nu')^l}{l!} \norm[L^{2l}(\R^2)]{u_j}^{2l} \le \Psi_{\nu'}(u_j) \le S_{\nu'} < \infty.
\]
Hence
\[
\frac{\nu^l}{l!} \norm[L^{2l}(\R^2)]{u_j}^{2l} \le S_{\nu'} \left(\frac{\nu}{\nu'}\right)^l =: M_l,
\]
and $\sum_{l \ge n} M_l < \infty$ since $\nu < \nu'$, so the last series in \eqref{202} converges uniformly in $j$ by the Weierstrass $M$-test. Since $u_j \to u$ in $L^{2l}(\R^2)$ for $l \ge n$ by \eqref{200} and Theorem \ref{Theorem 1}, passing to the limit in \eqref{202} now gives
\[
S_\nu = \sum_{l \ge n} \frac{\nu^l}{l!} \norm[L^{2l}(\R^2)]{u}^{2l} = \Psi_\nu(u).
\]
Since $S_\nu > 0$, $u \ne 0$. If $\norm[E^{p,q}_b(\R^2)]{u} < 1$, then $v := u/\norm[E^{p,q}_b(\R^2)]{u} \in U$ and $\Psi_\nu(v) > \Psi_\nu(u) = S_\nu$, a contradiction. So $\norm[E^{p,q}_b(\R^2)]{u} = 1$.
\end{proof}

\begin{remark} \label{rmk:tm-sharp}
The constant $4 \pi$ in Theorem \ref{Theorem 7} is sharp. The radial Moser functions concentrating at the origin lie in $E^{p,q}_{b,\rad}(\R^2)$, so the conclusion fails for $\nu > 4 \pi$. Whether it holds for $\nu = 4 \pi$ is an open problem (the attainment of the supremum at the borderline exponent is a delicate matter already for the classical inequality on a bounded domain, where it was settled by Carleson and Chang \cite{MR878016}).
\end{remark}

\subsection{Density of radial test functions}

Denote by $C^\infty_{c,\rad}(\R^N)$ the space of radial functions in $C^\infty_c(\R^N)$.

\begin{theorem} \label{Theorem 9}
If \eqref{33}--\eqref{2} hold, then $C^\infty_{c,\rad}(\R^N)$ is dense in $E^{p,q}_{b,\rad}(\R^N)$.
\end{theorem}

We begin by proving the following truncation lemma.

\begin{lemma} \label{Lemma 2}
Let $\xi : [0,\infty) \to [0,1]$ be a smooth function such that $\xi \equiv 1$ on $[0,1/3]$, $\xi \equiv 0$ on $[2/3,\infty)$, and $|\xi'| \le 4$. For $u \in E^{p,q}_{b,\rad}(\R^N)$ and $R > 0$, set $\xi_R(x) = \xi(|x|/R)$ and $u_R = \xi_R\halfthin u$. Then $u_R \in E^{p,q}_{b,\rad}(\R^N)$ is a radial function with compact support and $u_R \to u$ in $E^{p,q}_{b,\rad}(\R^N)$ as $R \to \infty$.
\end{lemma}

\begin{proof}
Since $\xi_R \equiv 1$ on $B_{R/3}$ and $0 \le 1 - \xi_R \le 1$,
\[
\norm[L^{p,q}_b(\R^N)]{u - u_R}^{pq} = \int_0^\infty \rho^{-b} \left(\int_{B_\rho} |(1 - \xi_R)\, u|^p\, dy\right)^q d\rho \le \int_{R/3}^\infty \rho^{-b}\, h_u^q(\rho)\, d\rho,
\]
which goes to zero as $R \to \infty$ since $u \in L^{p,q}_b(\R^N)$.

Since $\nabla(u - u_R) = (1 - \xi_R)\, \nabla u - u\, \nabla \xi_R$,
\[
\norm[L^2(\R^N)]{\nabla(u - u_R)} \le \norm[L^2(\R^N)]{(1 - \xi_R)\, \nabla u} + \norm[L^2(\R^N)]{u\, \nabla \xi_R}.
\]
Since $(1 - \xi_R)\, |\nabla u| \to 0$ a.e.\! in $\R^N$ and $(1 - \xi_R)^2\, |\nabla u|^2 \le |\nabla u|^2 \in L^1(\R^N)$,
\[
\norm[L^2(\R^N)]{(1 - \xi_R)\, \nabla u}^2 = \int_{\R^N} (1 - \xi_R)^2\, |\nabla u|^2\, dx \to 0 \quad \text{as } R \to \infty
\]
by the dominated convergence theorem. Since $|\nabla \xi_R| = |(1/R)\, \xi'(|x|/R)| \le (4/R)\, \mathbf{1}_{B_{2R/3} \setminus B_{R/3}}$ and $|u(r)| \le Cr^{- \zeta/(p + 2)} \norm[E^{p,q}_b(\R^N)]{u}$ by Theorem \ref{Theorem 6},
\begin{multline*}
\hspace{-12.2pt} \norm[L^2(\R^N)]{u\, \nabla \xi_R}^2 \le \frac{16}{R^2} \int_{B_{2R/3} \setminus B_{R/3}} u^2\, dx = \frac{C}{R^2} \int_{R/3}^{2R/3} r^{N-1}\, u^2(r)\, dr\\[7.5pt]
\hspace{-12.2pt} \le \frac{C}{R^2} \norm[E^{p,q}_b(\R^N)]{u}^2 \int_{R/3}^{2R/3} r^{N - 1 - 2 \zeta/(p + 2)}\, dr = CR^{N - 2 - 2 \zeta/(p + 2)} \norm[E^{p,q}_b(\R^N)]{u}^2 \to 0 \quad \text{as } R \to \infty
\end{multline*}
since
\[
N - 2 - \frac{2 \zeta}{p + 2} = - \frac{2\, (b_\ast - b)}{(p + 2)\, q} < 0
\]
by \eqref{2}. So $\norm[L^2(\R^N)]{\nabla(u - u_R)} \to 0$.
\end{proof}

Next we prove the following local estimate.

\begin{lemma} \label{Lemma 3}
If
\begin{equation} \label{15}
2^{p,q}_{b,\#} < s \begin{cases}
< \infty & \text{if } N = 2\\[5pt]
\le 2^\ast & \text{if } N \ge 3,
\end{cases}
\end{equation}
then for any $R > 0$,
\[
\norm[L^{p,q}_b(\R^N)]{u} \le C_R \norm[L^s(B_R)]{u} \quad \forall u \in L^{p,q}_b(\R^N) \text{ with } \supp u \subset B_R
\]
for some constant $C_R > 0$.
\end{lemma}

\begin{proof}
Since $p < s$, the H\"{o}lder inequality gives
\[
h_u(\rho) \le C \rho^{N\, (1 - p/s)} \norm[L^s(B_\rho)]{u}^p \le C \rho^{N\, (1 - p/s)} \norm[L^s(B_R)]{u}^p
\]
for $0 < \rho < R$ and
\[
h_u(\rho) = \norm[L^p(B_R)]{u}^p \le CR^{N\, (1 - p/s)} \norm[L^s(B_R)]{u}^p
\]
for $\rho \ge R$. So
\[
\norm[L^{p,q}_b(\R^N)]{u} \le C \norm[L^s(B_R)]{u} \left(\int_0^R \rho^{Nq\, (1 - p/s) - b}\, d\rho + R^{Nq\, (1 - p/s)} \int_R^\infty \rho^{-b}\, d\rho\right)^{1/pq}
\]
and both integrals converge since $s > 2^{p,q}_{b,\#}$ and $b > 1$.
\end{proof}

Next we prove the following mollification lemma.

\begin{lemma} \label{Lemma 4}
Let $\varphi \in C^\infty_c(B_1)$ be a radial function with $\varphi \ge 0$ and $\int \varphi = 1$. For $u \in E^{p,q}_{b,\rad}(\R^N)$ with compact support and $\eps > 0$, set $\varphi_\eps(x) = \eps^{-N} \varphi(x/\eps)$ and $u^\eps = u \star \varphi_\eps$. Then $u^\eps \in C^\infty_{c,\rad}(\R^N)$ and $u^\eps \to u$ in $E^{p,q}_{b,\rad}(\R^N)$ as $\eps \to 0$.
\end{lemma}

\begin{proof}
Let $s$ satisfy \eqref{15} and let $R > 0$ be so large that $\supp u \subset B_{R/2}$. Then $\supp u^\eps \subset B_{R/2 + \eps} \subset B_R$ for $\eps < R/2$. Since $u \in H^1(B_R)$ by Theorem \ref{Theorem 5} and $u$ has compact support in $B_R$, $u^\eps \to u$ in $H^1(B_R)$ as $\eps \to 0$, in particular, $\norm[L^2(\R^N)]{\nabla(u^\eps - u)} \to 0$. Since $H^1(B_R)$ is continuously embedded in $L^s(B_R)$, then $u^\eps \to u$ in $L^s(B_R)$ and hence also in $L^{p,q}_b(\R^N)$ by Lemma \ref{Lemma 3}.
\end{proof}

We are now ready to prove Theorem \ref{Theorem 9}.

\begin{proof}[Proof of Theorem \ref{Theorem 9}]
Given $u \in E^{p,q}_{b,\rad}(\R^N)$ and $\delta > 0$, $\exists R > 0$ with $\norm[E^{p,q}_b(\R^N)]{u - u_R} < \delta/2$ by Lemma \ref{Lemma 2} and $\exists \eps > 0$ with $\norm[E^{p,q}_b(\R^N)]{u_R^\eps - u_R} < \delta/2$ by Lemma \ref{Lemma 4}. Then $u_R^\eps \in C^\infty_{c,\rad}(\R^N)$ and $\norm[E^{p,q}_b(\R^N)]{u_R^\eps - u} < \delta$.
\end{proof}

\section{Scaling-based variational methods} \label{section:scaling}

In this section we recall the scaling-based variational methods from \cite{MR5043800} that we will need in the sequel. We also prove a new scaling-based multiplicity result for constrained even functionals and a variant scaled linking theorem that are of independent interest (see Theorem \ref{Theorem 10} and Theorem \ref{Theorem 8}).

\subsection{Scalings}

Let $W$ be a reflexive Banach space. The following notion of a scaling in $W$ was introduced in \cite{MR5043800}.

\begin{definition}[{\cite[Definition 1.1]{MR5043800}}]
A scaling in $W$ is a continuous mapping
\[
W \times [0,\infty) \to W, \quad (u,t) \mapsto u_t
\]
satisfying
\begin{enumerate}
\item[$(A_1)$] $(u_{t_1})_{t_2} = u_{t_1 t_2}$ for all $u \in W$ and $t_1, t_2 \ge 0$;
\item[$(A_2)$] $(\tau u)_t = \tau u_t$ for all $u \in W$, $\tau \in \R$, and $t \ge 0$;
\item[$(A_3)$] $u_0 = 0$ and $u_1 = u$ for all $u \in W$;
\item[$(A_4)$] $u_t$ is bounded on bounded sets in $W \times [0,\infty)$;
\item[$(A_5)$] there exists $\sigma > 0$ such that $\norm{u_t} = \O(t^\sigma)$ as $t \to \infty$, uniformly in $u$ on bounded sets.
\end{enumerate}
\end{definition}

\subsection{Scaled operators}

Denote by $W^\ast$ the dual of $W$. Recall that $q \in C(W,W^\ast)$ is a potential operator if there is a functional $Q \in C^1(W,\R)$, called a potential for $q$, with Fr\'{e}chet derivative $Q' = q$. By replacing $Q$ with $Q - Q(0)$ if necessary, we may assume that $Q(0) = 0$.

\begin{definition}[{\cite[Definition 2.1]{MR5043800}}]
A scaled operator is an odd potential operator $\As \in C(W,W^\ast)$ that maps bounded sets into bounded sets and satisfies
\begin{equation} \label{105}
\As(u_t) v_t = t^\sigma \As(u) v \quad \forall u, v \in W,\, t \ge 0.
\end{equation}
\end{definition}

\subsection{Cohomological index}

In this subsection we recall the definition and basic properties of the $\Z_2$-cohomological index of Fadell and Rabinowitz (see \cite{MR0478189}). We will use this index to construct minimax eigenvalues of scaled eigenvalue problems in the next section. Krasnoselskii's genus is not suitable for our purposes here since it lacks the piercing property of the cohomological index.

\begin{definition}[\cite{MR0478189}] \label{Definition 1}
Let $\A$ denote the class of symmetric subsets of $W \setminus \set{0}$. For $A \in \A$, let $\overline{A} = A/\Z_2$ be the quotient space of $A$ with each $u$ and $-u$ identified, let $f : \overline{A} \to \RP^\infty$ be the classifying map of $\overline{A}$, and let $f^\ast : H^\ast(\RP^\infty) \to H^\ast(\overline{A})$ be the induced homomorphism of the Alexander-Spanier cohomology rings. The cohomological index of $A$ is defined by
\[
i(A) = \begin{cases}
0 & \text{if } A = \emptyset\\[5pt]
\sup \set{m \ge 1 : f^\ast(\omega^{m-1}) \ne 0} & \text{if } A \ne \emptyset,
\end{cases}
\]
where $\omega \in H^1(\RP^\infty)$ is the generator of the polynomial ring $H^\ast(\RP^\infty) = \Z_2[\omega]$.
\end{definition}

\begin{example}
The classifying map of the unit sphere $S^N$ in $\R^{N+1},\, N \ge 0$ is the inclusion $\RP^N \incl \RP^\infty$, which induces isomorphisms on the cohomology groups $H^l$ for $l \le N$, so $i(S^N) = N + 1$.
\end{example}

The following proposition summarizes the basic properties of the cohomological index.

\begin{proposition}[\cite{MR0478189}] \label{Proposition 7}
The index $i : \A \to \N \cup \set{0,\infty}$ has the following properties:
\begin{enumerate}
\item[$(i_1)$] Definiteness: $i(A) = 0$ if and only if $A = \emptyset$.
\item[$(i_2)$] Monotonicity: If there is an odd continuous map from $A$ to $B$ (in particular, if $A \subset B$), then $i(A) \le i(B)$. Thus, equality holds when the map is an odd homeomorphism.
\item[$(i_3)$] Dimension: $i(A) \le \dim W$.
\item[$(i_4)$] Continuity: If $A$ is closed, then there is a closed neighborhood $U \in \A$ of $A$ such that $i(U) = i(A)$. When $A$ is compact, $U$ may be chosen to be a $\delta$-neighborhood $U_\delta(A) = \set{u \in W : \dist{u}{A} \le \delta}$.
\item[$(i_5)$] Subadditivity: If $A$ and $B$ are closed, then $i(A \cup B) \le i(A) + i(B)$.
\item[$(i_6)$] Stability: If $\Sigma A$ is the suspension of $A \ne \emptyset$, obtained as the quotient space of $A \times [-1,1]$ with $A \times \set{1}$ and $A \times \set{-1}$ collapsed to different points, then $i(\Sigma A) = i(A) + 1$.
\item[$(i_7)$] Piercing property: If $C$, $C_0$, and $C_1$ are closed and $\varphi : C \times [0,1] \to C_0 \cup C_1$ is a continuous map such that $\varphi(-u,t) = - \varphi(u,t)$ for all $(u,t) \in C \times [0,1]$, $\varphi(C \times [0,1])$ is closed, $\varphi(C \times \set{0}) \subset C_0$, and $\varphi(C \times \set{1}) \subset C_1$, then $i(\varphi(C \times [0,1]) \cap C_0 \cap C_1) \ge i(C)$.
\item[$(i_8)$] Neighborhood of zero: If $U$ is a bounded closed symmetric neighborhood of $0$, then $i(\bdry{U}) = \dim W$.
\end{enumerate}
\end{proposition}

\subsection{Scaled eigenvalue problems}

Let $\As$ and $\Bs$ be scaled operators satisfying
\begin{enumerate}
\item[$(A_6)$] $\As(u) u > 0$ for all $u \in W \setminus \set{0}$;
\item[$(A_7)$] every sequence $\seq{u_j}$ in $W$ such that $u_j \wto u$ and $\As(u_j)(u_j - u) \to 0$ has a subsequence that converges strongly to $u$;
\item[$(A_8)$] $\Bs(u) u > 0$ for all $u \in W \setminus \set{0}$;
\item[$(A_9)$] if $u_j \wto u$ in $W$, then $\Bs(u_j) \to \Bs(u)$ in $W^\ast$.
\end{enumerate}
The scaled eigenvalue problem
\begin{equation} \label{25}
\As(u) = \lambda \Bs(u) \quad \text{in } W^\ast
\end{equation}
was studied in \cite{MR5043800}. We say that $\lambda \in \R$ is an eigenvalue of this problem if there is a $u \in W \setminus \set{0}$, called an eigenfunction associated with $\lambda$, satisfying \eqref{25}. Then $u_t$ is also an eigenfunction associated with $\lambda$ for any $t > 0$ since
\[
\As(u_t) v = \As(u_t) (v_{t^{-1}})_t = t^\sigma \As(u) v_{t^{-1}} = \lambda t^\sigma \Bs(u) v_{t^{-1}} = \lambda \Bs(u_t) (v_{t^{-1}})_t = \lambda \Bs(u_t) v
\]
for all $v \in W$.

The potentials
\begin{equation} \label{21}
I_\sigma(u) = \int_0^1 \As(\tau u) u\, d\tau, \quad J_\sigma(u) = \int_0^1 \Bs(\tau u) u\, d\tau, \quad u \in W
\end{equation}
of $\As$ and $\Bs$, respectively, are even, bounded on bounded sets, and have the scaling \linebreak property
\begin{equation} \label{104}
I_\sigma(u_t) = t^\sigma I_\sigma(u), \quad J_\sigma(u_t) = t^\sigma J_\sigma(u) \quad \forall u \in W,\, t \ge 0
\end{equation}
(see \cite[Proposition 2.2]{MR5043800}). By \eqref{21}, $(A_6)$, and $(A_8)$,
\begin{equation} \label{22}
I_\sigma(u) > 0, \quad J_\sigma(u) > 0 \quad \forall u \in W \setminus \set{0}.
\end{equation}
It follows from $(A_9)$ that if $u_j \wto u$ in $W$, then $J_\sigma(u_j) \to J_\sigma(u)$ (see {\cite[Proposition 2.4]{MR5043800}}). We assume that $I_\sigma$ and $J_\sigma$ satisfy
\begin{enumerate}
\item[$(A_{10})$] $I_\sigma$ is coercive, i.e., $I_\sigma(u) \to \infty$ as $\norm{u} \to \infty$;
\item[$(A_{11})$] every solution of problem \eqref{25} satisfies the scaling identity $I_\sigma(u) = \lambda J_\sigma(u)$.
\end{enumerate}

We have $I_\sigma'(0) = \As(0) = 0$ since $\As$ is odd, so the origin is a critical point of $I_\sigma$. It is the only critical point of $I_\sigma$ since
\[
I_\sigma'(u) u = \As(u) u > 0 \quad \forall u \in W \setminus \set{0}
\]
by $(A_6)$. So $I_\sigma(0) = 0$ is the only critical value of $I_\sigma$, and hence it follows from the implicit function theorem that
\[
\M = \bgset{u \in W : I_\sigma(u) = 1}
\]
is a $C^1$-Finsler manifold. Since $I_\sigma$ is continuous, even, and coercive, we see that $\M$ is complete, symmetric, and bounded. Let
\[
\Psi(u) = \frac{1}{J_\sigma(u)}, \quad u \in W \setminus \set{0}
\]
and let $\widetilde{\Psi} = \restr{\Psi}{\M}$. Then eigenvalues of problem \eqref{25} coincide with critical values of $\widetilde{\Psi}$ (see \cite[Proposition 2.5]{MR5043800}).

Denote by $\F$ the class of symmetric subsets of $\M$ and by $i(M)$ the cohomological index of $M \in \F$. For $k \ge 1$, let $\F_k = \bgset{M \in \F : i(M) \ge k}$ and set
\[
\lambda_k := \inf_{M \in \F_k}\, \sup_{u \in M}\, \widetilde{\Psi}(u).
\]
The following theorem was proved in \cite{MR5043800}.

\begin{theorem}[{\cite[Theorem 2.10]{MR5043800}}] \label{Theorem 13}
Assume $(A_1)$--$(A_{11})$. Then $\lambda_k \nearrow \infty$ is a sequence of eigenvalues of problem \eqref{25}.
\begin{enumroman}
\item \label{Theorem 13.i} The first eigenvalue is given by
    \[
    \lambda_1 = \min_{u \in \M}\, \widetilde{\Psi}(u) > 0.
    \]
\item If $\lambda_k = \dotsb = \lambda_{k+m-1} = \lambda$ and $E_\lambda$ is the set of eigenfunctions associated with $\lambda$ that lie on $\M$, then $i(E_\lambda) \ge m$.
\item \label{Theorem 13.iii} If $\lambda_k < \lambda < \lambda_{k+1}$, then
    \[
    i(\widetilde{\Psi}^{\lambda_k}) = i(\M \setminus \widetilde{\Psi}_\lambda) = i(\widetilde{\Psi}^\lambda) = i(\M \setminus \widetilde{\Psi}_{\lambda_{k+1}}) = k,
    \]
    where $\widetilde{\Psi}^a = \bgset{u \in \M : \widetilde{\Psi}(u) \le a}$ and $\widetilde{\Psi}_a = \bgset{u \in \M : \widetilde{\Psi}(u) \ge a}$ for $a \in \R$.
\end{enumroman}
\end{theorem}

\subsection{Subscaled and superscaled problems} \label{ssec:sub-sup}

Let $\As$ be as in the last section. In this subsection we consider the nonlinear operator equation
\begin{equation} \label{18}
\As(u) = f(u) \quad \text{in } W^\ast,
\end{equation}
where $f \in C(W,W^\ast)$ is an odd potential operator that maps bounded sets into bounded sets and satisfies
\begin{enumerate}
\item[$(f_1)$] there exists $\gamma \ne \sigma$ such that $f(u_t) v_t = t^\gamma f(u) v$ for all $u, v \in W$ and $t \ge 0$;
\item[$(f_2)$] $f(u) u > 0$ for all $u \in W \setminus \set{0}$;
\item[$(f_3)$] if $u_j \wto u$ in $W$, then $f(u_j) \to f(u)$ in $W^\ast$;
\item[$(f_4)$] for any $\mu \in \R$, every solution of the equation $\As(u) = \mu f(u)$ satisfies the scaling identity $\sigma I_\sigma(u) = \gamma \mu F(u)$, where $F(u) = \int_0^1 f(\tau u) u\, d\tau$ is the potential of $f$.
\end{enumerate}
Solutions of equation \eqref{18} are the critical points of the $C^1$-functional
\[
E(u) = I_\sigma(u) - F(u), \quad u \in W,
\]
where $I_\sigma$ is as in the last section. However, it seems difficult to verify the boundedness of \PS{} sequences of $E$, so we will take a constrained approach.

Let $\M$ be as in the last section. As noted in \cite[Section 2.2.1]{MR5043800}, the scaling induces a continuous projection
\begin{equation} \label{102}
\pi : W \setminus \set{0} \to \M, \quad u \mapsto \widetilde{u} := u_{t_u},
\end{equation}
where $t_u = I_\sigma(u)^{- 1/\sigma}$ (see \eqref{22}). We use this projection to reduce equation \eqref{18} to the nonlinear eigenvalue problem
\begin{equation} \label{17}
\As(u) = \mu f(u)
\end{equation}
on $\M$.

\begin{lemma} \label{Lemma 5}
If $u$ is a nontrivial solution of equation \eqref{18}, then $\widetilde{u}$ is an eigenfunction associated with the eigenvalue
\[
\mu = t_u^{\sigma - \gamma}.
\]
Conversely, if $u \in \M$ is an eigenfunction associated with an eigenvalue $\mu \in \R$, then $\mu \ne 0$,
\[
\overline{u} = u_{\mu^{- 1/(\sigma - \gamma)}}
\]
is a nontrivial solution of equation \eqref{18}, and
\begin{equation} \label{24}
E(\overline{u}) = (1 - \mu F(u))\, \mu^{- \sigma/(\sigma - \gamma)}.
\end{equation}
Moreover, solutions $\overline{u}$ that correspond to distinct basepoints $u \in \M$ are distinct.
\end{lemma}

\begin{proof}
If $u$ is a nontrivial solution of equation \eqref{18}, then
\begin{multline*}
\As(\widetilde{u}) w = \As(u_{t_u})(w_{t_u^{-1}})_{t_u} = t_u^\sigma \As(u) w_{t_u^{-1}} = t_u^{\sigma - \gamma} t_u^\gamma f(u) w_{t_u^{-1}}\\[7.5pt]
= t_u^{\sigma - \gamma} f(u_{t_u})(w_{t_u^{-1}})_{t_u} = t_u^{\sigma - \gamma} f(\widetilde{u}) w
\end{multline*}
for all $w \in W$. Conversely, if $u \in \M$ is an eigenfunction associated with an eigenvalue $\mu \in \R$, then $\mu \ne 0$ by $(A_6)$, $\overline{u} \ne 0$ by (2.3) in \cite{MR5043800}, and
\begin{multline*}
\As(\overline{u}) w = \As(u_{\mu^{- 1/(\sigma - \gamma)}})(w_{\mu^{1/(\sigma - \gamma)}})_{\mu^{- 1/(\sigma - \gamma)}} = \mu^{- \sigma/(\sigma - \gamma)} \As(u) w_{\mu^{1/(\sigma - \gamma)}}\\[7.5pt]
= \mu^{- \gamma/(\sigma - \gamma)} f(u) w_{\mu^{1/(\sigma - \gamma)}} = f(u_{\mu^{- 1/(\sigma - \gamma)}})(w_{\mu^{1/(\sigma - \gamma)}})_{\mu^{- 1/(\sigma - \gamma)}} = f(\overline{u}) w
\end{multline*}
for all $w \in W$. Since $u \in \M$ and $F$ has the scaling property
\[
F(u_t) = t^\gamma F(u) \quad \forall u \in W,\, t \ge 0
\]
(see \cite[Proposition 2.2]{MR5043800}), \eqref{24} follows. Finally, solutions $\overline{u}$ that correspond to distinct basepoints $u \in \M$ are distinct since fibers $\set{u_t : t > 0}$ with distinct basepoints $u$ do not intersect.
\end{proof}

The main result of this section is the following theorem.

\begin{theorem} \label{Theorem 10}
Assume $(A_1)$--$(A_7)$, $(A_{10})$, and $(f_1)$--$(f_4)$. Then equation \eqref{18} has an infinite sequence of nontrivial solutions
\begin{equation} \label{23}
\overline{u}_k = u^k_{\mu_k^{- 1/(\sigma - \gamma)}},
\end{equation}
where $u^k \in \M$, $\mu_k > 0$, and $\mu_k \nearrow \infty$. Moreover,
\[
E(\overline{u}_k) = \left(1 - \frac{\sigma}{\gamma}\right) \mu_k^{- \sigma/(\sigma - \gamma)},
\]
in particular, $E(\overline{u}_k) < 0$ and $E(\overline{u}_k) \nearrow 0$ if $\gamma < \sigma$, while $E(\overline{u}_k) > 0$ and $E(\overline{u}_k) \nearrow \infty$ if $\gamma > \sigma$.
\end{theorem}

The potential $F$ of $f$ is even and bounded on bounded sets (see \cite[Proposition 2.2]{MR5043800}). By $(f_2)$, $F(u) > 0$ for all $u \in W \setminus \set{0}$, so the functional
\[
G(u) = \frac{1}{F(u)}, \quad u \in W \setminus \set{0}
\]
is positive and $\widetilde{G} = \restr{G}{\M}$ is $C^1$. Since
\[
I_\sigma'(u) = \As(u), \qquad G'(u) = - \frac{F'(u)}{F(u)^2} = - G(u)^2\, f(u),
\]
the norm of $\widetilde{G}'(u)$ as an element of the cotangent space $T_u^\ast \M$ at $u \in \M$ is given by
\begin{equation} \label{41}
\bgdnorm[u]{\widetilde{G}'(u)} = \min_{\nu \in \R}\, \bgdnorm{\nu I_\sigma'(u) + G'(u)} = \min_{\nu \in \R}\, \bgdnorm{\nu \As(u) - \widetilde{G}(u)^2\, f(u)},
\end{equation}
where $\dnorm{\cdot}$ is the norm in $W^\ast$ (see, e.g., \cite[Proposition 3.54]{MR2640827}).

\begin{lemma} \label{Lemma 6}
If $u \in \M$ is a critical point of $\widetilde{G}$, then $u$ is an eigenfunction of problem \eqref{17} associated with the eigenvalue
\begin{equation} \label{20}
\mu = \frac{\sigma}{\gamma}\, \widetilde{G}(u).
\end{equation}
\end{lemma}

\begin{proof}
By \eqref{41},
\begin{equation} \label{19}
\nu \As(u) = \widetilde{G}(u)^2\, f(u)
\end{equation}
for some $\nu \in \R$. Since $\widetilde{G}(u) > 0$ and $f(u) \ne 0$ by $(f_2)$, we have $\nu \ne 0$, so $(f_4)$ gives
\[
\sigma \nu I_\sigma(u) = \gamma\, \widetilde{G}(u)^2\, F(u).
\]
Since $u \in \M$, this gives
\[
\nu = \frac{\gamma}{\sigma}\, \widetilde{G}(u),
\]
so \eqref{19} reduces to \eqref{17} with $\mu$ as in \eqref{20}.
\end{proof}

\begin{lemma}
$\widetilde{G}$ satisfies the {\em \PS{}} condition.
\end{lemma}

\begin{proof}
Let $c \in \R$ and let $\seq{u_j} \subset \M$ be a \PS{c} sequence of $\widetilde{G}$, i.e.,
\[
\widetilde{G}(u_j) \to c, \qquad \bgdnorm[u_j]{\widetilde{G}'(u_j)} \to 0.
\]
Since $\M$ is a bounded manifold, $\seq{u_j}$ is bounded and hence converges weakly to some $u \in W$ for a renamed subsequence. Then it follows from $(f_3)$ that $F(u_j) \to F(u)$ (see the proof of Proposition 2.4 in {\cite{MR5043800}}). Since
\[
\widetilde{G}(u_j) = \frac{1}{F(u_j)} \to c,
\]
it follows that $c > 0$ and $F(u) > 0$, in particular, $u \ne 0$.

By \eqref{41}, $\bgdnorm[u_j]{\widetilde{G}'(u_j)} \to 0$ implies that
\begin{equation} \label{42}
\nu_j\, \As(u_j) - \widetilde{G}(u_j)^2\, f(u_j) \to 0
\end{equation}
for some sequence $\seq{\nu_j} \subset \R$. Applying \eqref{42} to $u$ and noting that $\As(u_j) u$ is bounded since $\As$ maps bounded sets into bounded sets and
\[
\widetilde{G}(u_j)^2\, f(u_j) u \to c^2\, f(u) u > 0
\]
by $(f_3)$ and $(f_2)$ shows that $\nu_j$ is bounded away from zero. Now applying \eqref{42} to $u_j - u$ shows that $\As(u_j)(u_j - u) \to 0$, so $u_j \to u$ for a further subsequence by $(A_7)$.
\end{proof}

We are now ready to prove Theorem \ref{Theorem 10}.

\begin{proof}[Proof of Theorem \ref{Theorem 10}]
Denote by $\F$ the class of symmetric subsets of $\M$ and by $i(M)$ the cohomological index of $M \in \F$. For $k \ge 1$, let $\F_k = \bgset{M \in \F : i(M) \ge k}$ and set
\[
c_k := \inf_{M \in \F_k}\, \sup_{u \in M}\, \widetilde{G}(u).
\]
Then $c_k > 0$ and $c_k \nearrow \infty$ is a sequence of critical values of $\widetilde{G}$ (see \cite[Proposition 3.52]{MR2640827}). Let $u^k \in \M$ be a critical point of $\widetilde{G}$ associated with the critical value $c_k$. Then $u^k$ is an eigenfunction of problem \eqref{17} associated with the eigenvalue
\[
\mu_k = \frac{\sigma}{\gamma}\, \widetilde{G}(u^k) = \frac{\sigma}{\gamma}\, c_k
\]
by Lemma \ref{Lemma 6}. Hence $\overline{u}_k$ defined in \eqref{23} is a nontrivial solution of equation \eqref{18} and
\[
E(\overline{u}_k) = (1 - \mu_k F(u^k))\, \mu_k^{- \sigma/(\sigma - \gamma)} = \left(1 - \frac{\mu_k}{\widetilde{G}(u^k)}\right) \mu_k^{- \sigma/(\sigma - \gamma)} = \left(1 - \frac{\sigma}{\gamma}\right) \mu_k^{- \sigma/(\sigma - \gamma)}
\]
by Lemma \ref{Lemma 5}.
\end{proof}

\subsection{Scaling-based multiplicity results for even functionals}

Let $I \in C(W,\R)$ be an even functional satisfying
\begin{equation} \label{121}
I(u_t) = t^\sigma I(u) \quad \forall u \in W,\, t \ge 0
\end{equation}
and
\[
I(u) > 0 \quad \forall u \in W \setminus \set{0}.
\]
Let
\[
\M = \bgset{u \in W : I(u) = 1}
\]
and note that $\M$ is closed and symmetric since $I$ is continuous and even. We assume that $\M$ is a bounded set and that each ray starting from the origin intersects $\M$ at exactly one point. Making essential use of the piercing property of the cohomological index, the following theorem was proved in \cite{MR5043800}.

\begin{theorem}[{\cite[Theorem 2.33]{MR5043800}}] \label{thm:even-non-min}
Let $E \in C^1(W,\R)$ be an even functional and assume that there exists $c^\ast > 0$ such that $E$ satisfies the {\em \PS{c}} condition for all $c \in (0,c^\ast)$. Let $A_0$ and $B_0$ be symmetric subsets of $\M$ such that $A_0$ is compact, $B_0$ is closed, and
\[
i(A_0) \ge k + m - 1, \qquad i(\M \setminus B_0) \le k - 1
\]
for some $k, m \ge 1$. Assume that there exist $R > \rho > 0$ such that
\[
\sup_{u \in A}\, E(u) \le 0 < \inf_{u \in B}\, E(u), \qquad \sup_{u \in X}\, E(u) < c^\ast,
\]
where $A = \set{u_R : u \in A_0}$, $B = \set{u_\rho : u \in B_0}$, and $X = \set{u_t : u \in A_0,\, 0 \le t \le R}$. Then $E$ has $m$ distinct pairs of critical points at levels in $(0,c^\ast)$.
\end{theorem}

This theorem reduces to the following corollary when $B_0 = \M$ and $k = 1$. The special case $u_t = tu$ of this corollary was proved in \cite{MR4999814}.

\begin{corollary}[{\cite[Corollary 2.34]{MR5043800}}] \label{Corollary 1}
Let $E \in C^1(W,\R)$ be an even functional and assume that there exists $c^\ast > 0$ such that $E$ satisfies the {\em \PS{c}} condition for all $c \in (0,c^\ast)$. Let $A_0$ be a compact symmetric subset of $\M$ with $i(A_0) = m \ge 1$. Assume that there exist $R > \rho > 0$ such that
\[
\sup_{u \in A}\, E(u) \le 0 < \inf_{u \in \M_\rho}\, E(u), \qquad \sup_{u \in X}\, E(u) < c^\ast,
\]
where $A = \set{u_R : u \in A_0}$, $\M_\rho = \set{u_\rho : u \in \M}$, and $X = \set{u_t : u \in A_0,\, 0 \le t \le R}$. Then $E$ has $m$ distinct pairs of critical points at levels in $(0,c^\ast)$.
\end{corollary}

In particular, we have the following corollary when $c^\ast = \infty$.

\begin{corollary} \label{Corollary 2}
Let $E \in C^1(W,\R)$ be an even functional that satisfies the {\em \PS{c}} condition for all $c \in (0,\infty)$. Assume that for any $m \ge 1$, there exist a compact symmetric subset $A_0$ of $\M$ with $i(A_0) = m$ and $R > \rho > 0$ such that
\[
\sup_{u \in A}\, E(u) \le 0 < \inf_{u \in \M_\rho}\, E(u),
\]
where $A = \set{u_R : u \in A_0}$ and $\M_\rho = \set{u_\rho : u \in \M}$. Then $E$ has infinitely many critical points at positive levels.
\end{corollary}

\subsection{Scaled linking theorem}

Let $X$ be a closed subset of $W$, let $A$ be a proper subset of $X$ that is closed and bounded, and let $B$ be a nonempty closed subset of $W$ such that $A \cap B = \emptyset$. Set
\[
\Gamma = \bgset{\gamma \in C(X,W) : \gamma(X) \text{ is closed and } \restr{\gamma}{A} = \id}.
\]

\begin{definition}
We say that $A$ links $B$ with respect to $X$ if
\[
\gamma(X) \cap B \ne \emptyset \quad \forall \gamma \in \Gamma.
\]
\end{definition}

Let $\M$ be as in the last section. The following proposition is a variant of \cite[Proposition 2.22]{MR5043800}.

\begin{proposition} \label{Proposition 9}
Let $A_0$ and $B_0$ be disjoint nonempty symmetric subsets of $\M$ such that $A_0$ is compact, $B_0$ is closed, and
\begin{equation} \label{223}
i(A_0) = i(\M \setminus B_0) < \infty.
\end{equation}
Let $e \in \M$ and let $H \in C(A_0 \times [0,1],\M)$ satisfy
\[
H(u,0) = u, \quad H(u,1) = e \quad \forall u \in A_0.
\]
Let $R > \rho > 0$ and set
\begin{gather*}
X = \set{H(u,\tau)_t : u \in A_0,\, \tau \in [0,1],\, 0 \le t \le R},\\
A = \set{u_t : u \in A_0,\, 0 \le t \le R} \cup \set{H(u,\tau)_R : u \in A_0,\, \tau \in [0,1]},\\
B = \set{u_\rho : u \in B_0}.
\end{gather*}
Then $A$ links $B$ with respect to $X$.
\end{proposition}

\begin{proof}
Let $\Sigma A_0$ be the suspension of $A_0$, obtained as the quotient space of $A_0 \times [-1,1]$ with $A_0 \times \set{1}$ and $A_0 \times \set{-1}$ collapsed to different points, equipped with the free involution $[u,\tau] \mapsto [-u,-\tau]$. The cohomological index is defined for all paracompact free $\Z_2$-spaces, and properties $(i_2)$, $(i_6)$, and $(i_7)$ of Proposition \ref{Proposition 7}, which are the ones used below, hold in this generality (see \cite{MR0478189}). Recall that
\begin{equation} \label{eq:susp}
i(\Sigma A_0) = i(A_0) + 1
\end{equation}
by Proposition \ref{Proposition 7} $(i_6)$.

If $A$ does not link $B$ with respect to $X$, then there is a continuous map $\gamma : X \to W \setminus B$ with $\restr{\gamma}{A} = \id$. Consider the map $\varphi : \Sigma A_0 \times [0,1] \to W$ defined by
\[
\varphi([u,\tau],t) = \begin{cases}
\gamma(H(u,\tau)_{Rt}) & \text{if } \tau \in [0,1]\\[5pt]
- \gamma(H(-u,-\tau)_{Rt}) & \text{if } \tau \in [-1,0].
\end{cases}
\]
Note that $H(u,\tau)_{Rt} = (H(u,\tau)_R)_t \in X$ for $t \in [0,1]$ by $(A_1)$. The map $\varphi$ is well-defined on $\Sigma A_0 \times [0,1]$ since $H(u,1) = e$ is independent of $u$, and it is continuous since $\gamma$ is the identity on the set $\set{u_t : u \in A_0,\, 0 \le t \le R} \subset A$ and hence
\[
\gamma(H(u,0)_{Rt}) = \gamma(u_{Rt}) = u_{Rt} = - (-u)_{Rt} = - \gamma((-u)_{Rt}) = - \gamma(H(-u,0)_{Rt})
\]
by $(A_2)$. It is immediate from the definition that $\varphi([-u,-\tau],t) = - \varphi([u,\tau],t)$ for all $([u,\tau],t) \in \Sigma A_0 \times [0,1]$. Moreover, $\varphi(\Sigma A_0 \times [0,1]) = \gamma(X) \cup - \gamma(X)$ is compact since $X$ is compact, $\varphi(\Sigma A_0 \times \set{0}) = \set{0}$ by $(A_3)$ since $0 \in A$ and hence $\gamma(0) = 0$, and
\[
\varphi(\Sigma A_0 \times \set{1}) = \set{H(u,\tau)_R : u \in A_0,\, \tau \in [0,1]} \cup \set{- H(u,\tau)_R : u \in A_0,\, \tau \in [0,1]}
\]
since $\restr{\gamma}{A} = \id$ and $A_0$ is symmetric. Noting that $I = R^\sigma > \rho^\sigma$ on $\varphi(\Sigma A_0 \times \set{1})$ by \eqref{121} since $\pm H(u,\tau) \in \M$ and applying Proposition \ref{Proposition 7} $(i_7)$ with $C = \Sigma A_0$, $C_0 = \set{u \in W : I(u) \le \rho^\sigma}$, and $C_1 = \set{u \in W : I(u) \ge \rho^\sigma}$ gives
\begin{equation}
i(\Sigma A_0) \le i((\gamma(X) \cup - \gamma(X)) \cap \M_\rho),
\end{equation}
where
\[
\M_\rho = C_0 \cap C_1 = \set{u \in W : I(u) = \rho^\sigma} = \set{u_\rho : u \in \M}.
\]
Since $\gamma(X) \cap B = \emptyset$ and $B$ is symmetric,
\[
(\gamma(X) \cup - \gamma(X)) \cap \M_\rho \subset \M_\rho \setminus B,
\]
and the map $\M_\rho \setminus B \to \M \setminus B_0,\, u \mapsto u_{\rho^{-1}}$ is an odd homeomorphism, so
\begin{equation} \label{224}
i((\gamma(X) \cup - \gamma(X)) \cap \M_\rho) \le i(\M_\rho \setminus B) = i(\M \setminus B_0)
\end{equation}
by Proposition \ref{Proposition 7} $(i_2)$. Combining \eqref{eq:susp}--\eqref{224} gives $i(A_0) + 1 \le i(\M \setminus B_0)$, contradicting \eqref{223}.
\end{proof}

The following scaled linking theorem follows from Proposition \ref{Proposition 9} (see, e.g., \cite[Proposition 3.21]{MR2640827}). Note that the set $X$ in Proposition \ref{Proposition 9} is compact since $A_0$ is compact and the scaling is continuous, so $\gamma(X)$ is automatically closed for every $\gamma \in C(X,W)$ and the closedness requirement in the definition of $\Gamma$ may be dropped.

\begin{theorem} \label{Theorem 8}
Let $E \in C^1(W,\R)$. Let $A_0$ and $B_0$ be disjoint nonempty symmetric subsets of $\M$ such that $A_0$ is compact, $B_0$ is closed, and
\[
i(A_0) = i(\M \setminus B_0) < \infty.
\]
Let $e \in \M$ and let $H \in C(A_0 \times [0,1],\M)$ satisfy
\[
H(u,0) = u, \quad H(u,1) = e \quad \forall u \in A_0.
\]
Let $R > \rho > 0$ and set
\begin{gather*}
X = \set{H(u,\tau)_t : u \in A_0,\, \tau \in [0,1],\, 0 \le t \le R},\\
A = \set{u_t : u \in A_0,\, 0 \le t \le R} \cup \set{H(u,\tau)_R : u \in A_0,\, \tau \in [0,1]},\\
B = \set{u_\rho : u \in B_0}.
\end{gather*}
Assume that
\[
\sup_{u \in A}\, E(u) \le \inf_{u \in B}\, E(u), \qquad \sup_{u \in X}\, E(u) < \infty.
\]
Let $\Gamma = \bgset{\gamma \in C(X,W) : \restr{\gamma}{A} = \id}$ and set
\[
c := \inf_{\gamma \in \Gamma}\, \sup_{u \in \gamma(X)}\, E(u).
\]
Then $\inf_{u \in B}\, E(u) \le c \le \sup_{u \in X}\, E(u)$. If, in addition, $E$ satisfies the {\em \PS{c}} condition, then $c$ is a critical value of $E$.
\end{theorem}

\section{Existence and multiplicity results} \label{section:main}

In this section we state and prove our main existence and multiplicity results for equation \eqref{1}. Throughout this section we write $\norm{\cdot}$ for the norm in $E^{p,q}_b(\R^N)$.

\subsection{Scaling in the ball-mass energy space}

Set
\begin{equation} \label{103}
u_t(x) = t^\alpha\, u(tx), \quad x \in \R^N,\, t \ge 0,
\end{equation}
where $\alpha$ is given in \eqref{37}. Then
\begin{gather}
\label{520} \int_{\R^N} |\nabla u_t|^2\, dx = t^{2 \alpha + 2 - N} \int_{\R^N} |\nabla u|^2\, dx,\\[10pt]
\label{521} \int_0^\infty \rho^{-b}\, h_{u_t}^q(\rho)\, d\rho = t^{(\alpha p - N)\, q} \int_0^\infty \rho^{-b}\, h_u^q(t \rho)\, d\rho = t^{(\alpha p - N)\, q + b - 1} \int_0^\infty \rho^{-b}\, h_u^q(\rho)\, d\rho,\\[10pt]
\notag \int_{\R^N} |u_t|^\beta\, dx = t^{\alpha \beta - N} \int_{\R^N} |u|^\beta\, dx,
\end{gather}
where $\beta$ is given in \eqref{31}. We have
\begin{equation} \label{55}
\sigma := 2 \alpha + 2 - N = (\alpha p - N)\, q + b - 1 = \alpha \beta - N = \frac{2\, (b_\ast - b)}{pq - 2} > 0
\end{equation}
by \eqref{29}, \eqref{34}, and \eqref{2}.

The mapping
\begin{equation} \label{32}
E^{p,q}_{b,\rad}(\R^N) \times [0,\infty) \to E^{p,q}_{b,\rad}(\R^N), \quad (u,t) \mapsto u_t
\end{equation}
is a scaling in $E^{p,q}_{b,\rad}(\R^N)$. Assumptions $(A_1)$--$(A_3)$ are clearly satisfied. By \eqref{520}, \eqref{521}, and \eqref{55},
\begin{equation} \label{38}
\norm{u_t}^{2pq} = t^{pq \sigma} \norm[L^2(\R^N)]{\nabla u}^{2pq} + t^{2 \sigma} \norm[L^{p,q}_b(\R^N)]{u}^{2pq} \le \max \set{t^{pq \sigma},t^{2 \sigma}} \norm{u}^{2pq},
\end{equation}
and $(A_4)$ and $(A_5)$ follow from this estimate. It only remains to verify continuity, which we do in the following proposition.

\begin{proposition} \label{prop:continuous}
The mapping in \eqref{32} is continuous.
\end{proposition}

\begin{proof}
We will show that if $u_j \to u$ in $E^{p,q}_{b,\rad}(\R^N)$ and $t_j \to t$ in $[0,\infty)$, then $(u_j)_{t_j} \to u_t$ in $E^{p,q}_{b,\rad}(\R^N)$. If $t = 0$, then $u_t = 0$ and $(u_j)_{t_j} \to 0$ by \eqref{38}, so suppose $t > 0$. Then
\[
\|(u_j)_{t_j} - u_t\| \le \|(u_j)_{t_j} - u_{t_j}\| + \|u_{t_j} - u_t\|,
\]
and \eqref{38} gives
\[
\|(u_j)_{t_j} - u_{t_j}\| = \|(u_j - u)_{t_j}\| \le \max \bgset{t_j^{\sigma/2},t_j^{\sigma/pq}} \norm{u_j - u} \to 0
\]
and
\[
\|u_{t_j} - u_t\| = \|(u_{t_j/t})_t - u_t\| = \|(u_{t_j/t} - u)_t\| \le \max \set{t^{\sigma/2},t^{\sigma/pq}} \norm{u_{t_j/t} - u}.
\]
So it suffices to show that $u_{t_j} \to u$ if $t_j \to 1$. We may assume that $u \in C^\infty_{c,\rad}(\R^N)$ by density (see Theorem \ref{Theorem 9}).

Since $(u_{t_j})$ is bounded in $E^{p,q}_{b,\rad}(\R^N)$ by \eqref{38} and $E^{p,q}_{b,\rad}(\R^N)$ is reflexive by Theorem \ref{Theorem 4}, a renamed subsequence of $(u_{t_j})$ converges weakly in $E^{p,q}_{b,\rad}(\R^N)$. The weak limit is $u$ since $u_{t_j} \to u$ pointwise as $t_j \to 1$ by the continuity of $u$. Since
\[
\|u_{t_j}\|^{2pq} = t_j^{pq \sigma} \norm[L^2(\R^N)]{\nabla u}^{2pq} + t_j^{2 \sigma} \norm[L^{p,q}_b(\R^N)]{u}^{2pq} \to \norm[L^2(\R^N)]{\nabla u}^{2pq} + \norm[L^{p,q}_b(\R^N)]{u}^{2pq} = \norm{u}^{2pq}
\]
by \eqref{38} and hence $\|u_{t_j}\| \to \norm{u}$, the desired conclusion follows from uniform convexity.
\end{proof}

\subsection{Ball-mass operator and its potential}

The scaled operator that corresponds to the left-hand side of equation \eqref{1} and its potential are given by
\begin{equation} \label{100}
\As(u) v = \int_{\R^N} \nabla u \cdot \nabla v\, dx + \int_{\R^N} \Phi_u(|x|)\, |u|^{p-2}\, uv\, dx, \quad u, v \in E^{p,q}_{b,\rad}(\R^N)
\end{equation}
and
\begin{equation} \label{101}
I_\sigma(u) = \frac{1}{2} \int_{\R^N} |\nabla u|^2\, dx + \frac{a}{pq} \int_0^\infty \rho^{-b}\, h_u^q(\rho)\, d\rho, \quad u \in E^{p,q}_{b,\rad}(\R^N),
\end{equation}
respectively (see Proposition \ref{prop:c1}). By \eqref{73},
\begin{equation} \label{134}
\As(u) u = \int_{\R^N} |\nabla u|^2\, dx + a \int_0^\infty \rho^{-b}\, h_u^q(\rho)\, d\rho
\end{equation}
and \eqref{101} can also be written as
\begin{equation} \label{140}
I_\sigma(u) = \frac{1}{2} \int_{\R^N} |\nabla u|^2\, dx + \frac{1}{pq} \int_{\R^N} \Phi_u(|x|)\, |u|^p\, dx,
\end{equation}
which is often the more convenient form in the applications. Assumptions $(A_6)$ and $(A_{10})$ are clearly satisfied. We verify assumption $(A_7)$ in the following proposition.

\begin{proposition} \label{prop:a7}
Every sequence $\seq{u_j}$ in $E^{p,q}_{b,\rad}(\R^N)$ such that $u_j \wto u$ and
\begin{equation} \label{40}
\limsup_{j \to \infty}\halfthin \left(\int_{\R^N} \nabla u_j \cdot \nabla (u_j - u)\, dx + \int_{\R^N} \Phi_{u_j}(|x|)\, |u_j|^{p-2}\, u_j\halfthin (u_j - u)\, dx\right) \le 0
\end{equation}
converges strongly to $u$.
\end{proposition}

\begin{proof}
Corollary \ref{cor:bm-pairing} shows that the mapping
\[
E^{p,q}_{b,\rad}(\R^N) \to \R, \quad v \mapsto \int_{\R^N} \nabla u \cdot \nabla v\, dx + \int_{\R^N} \Phi_u(|x|)\, |u|^{p-2}\, uv\, dx
\]
is a bounded linear functional on $E^{p,q}_{b,\rad}(\R^N)$, so weak convergence implies
\[
\int_{\R^N} \nabla u \cdot \nabla (u_j - u)\, dx + \int_{\R^N} \Phi_u(|x|)\, |u|^{p-2}\, u\, (u_j - u)\, dx \to 0.
\]
Subtracting this from \eqref{40} gives
\[
\limsup_{j \to \infty}\halfthin \left(\norm[L^2(\R^N)]{\nabla (u_j - u)}^2 + \int_{\R^N} \left[\Phi_{u_j}(|x|)\, |u_j|^{p-2}\, u_j - \Phi_u(|x|)\, |u|^{p-2}\, u\right] (u_j - u)\, dx\right) \le 0.
\]
By Corollary \ref{cor:bm-pairing}, the last integral is greater than or equal to
\begin{multline*}
a \left(\norm[L^{p,q}_b(\R^N)]{u_j}^{pq} - \norm[L^{p,q}_b(\R^N)]{u_j}^{pq-1} \norm[L^{p,q}_b(\R^N)]{u} - \norm[L^{p,q}_b(\R^N)]{u}^{pq-1} \norm[L^{p,q}_b(\R^N)]{u_j} + \norm[L^{p,q}_b(\R^N)]{u}^{pq}\right)\\[7.5pt]
= a \left(\norm[L^{p,q}_b(\R^N)]{u_j}^{pq-1} - \norm[L^{p,q}_b(\R^N)]{u}^{pq-1}\right)\! \left(\norm[L^{p,q}_b(\R^N)]{u_j} - \norm[L^{p,q}_b(\R^N)]{u}\right) \ge 0,
\end{multline*}
so it follows that $\norm[L^2(\R^N)]{\nabla (u_j - u)} \to 0$ and $\norm[L^{p,q}_b(\R^N)]{u_j} \to \norm[L^{p,q}_b(\R^N)]{u}$. Then $\norm{u_j} \to \norm{u}$ and hence the desired conclusion follows since $u_j \wto u$ and $E^{p,q}_{b,\rad}(\R^N)$ is uniformly convex (see Theorem \ref{Theorem 4}).
\end{proof}

\subsection{Regularity of weak radial solutions}

Since $b < b_\#$ by \eqref{2} and \eqref{36}, $N\halfthin (q - 1) - b > -3$. Fix $\theta \in (0,2)$ with $N\halfthin (q - 1) - b > - \theta - 1$. We have the following regularity result and pointwise bounds near zero for weak radial solutions of equation \eqref{1}.

\begin{theorem} \label{thm:bm-regularity}
Assume \eqref{33}--\eqref{2} and \eqref{26}. If $u \in E^{p,q}_{b,\rad}(\R^N)$ is a weak radial solution of equation \eqref{1}, then $u \in L^\infty_\loc(\R^N) \cap W^{2,m}_\loc(\R^N \setminus \set{0}) \cap C^{1,\vartheta}_\loc(\R^N \setminus \set{0})$ for all $m < \infty$ and some $\vartheta \in (0,1)$. Moreover, there exists a constant $C > 0$ such that
\begin{multline} \label{58}
|u(r)| \le C, \quad |f(u)| \le C, \quad h_u(r) \le Cr^N, \quad \Phi_u(r) \le Cr^{- \theta}, \quad |u'(r)| \le Cr^{1 - \theta},\\[7.5pt]
|u''(r)| \le Cr^{- \theta}
\end{multline}
for $0 < r \le 1$.
\end{theorem}

\begin{proof}
Set $g = f(u) - \Phi_u\, |u|^{p-2}\, u$ so that $- \Delta u = g$ weakly. By \eqref{26} and Theorem \ref{Theorem 1}, $u \in L^1_\loc(\R^N)$ and $f(u) \in L^1_\loc(\R^N)$. Taking $v = \sign(u)\, \mathbf{1}_{B_R}$ in Corollary \ref{cor:bm-pairing} shows that $\Phi_u\, |u|^{p-2}\, u \in L^1_\loc(\R^N)$. Hence $g \in L^1_\loc(\R^N)$ and Kato's inequality gives
\[
- \Delta\, |u| \le \sign(u)\, g = \sign(u)\, f(u) - \Phi_u\, |u|^{p-1} \le C \left(|u|^{s_1 - 1} + |u|^{s_2 - 1}\right)
\]
weakly. First we show that $u \in L^m_\loc(\R^N)$ for all $m < \infty$. This follows from Theorem \ref{Theorem 1} if $N = 2$, so suppose $N \ge 3$. Then
\[
C \left(|u|^{s_1 - 1} + |u|^{s_2 - 1}\right) \le C \left(1 + |u|^{2^\ast - 1}\right) \le C \left(1 + |u|^{2^\ast - 2}\right) (1 + |u|),
\]
so $w = |u|$ is a nonnegative subsolution of $- \Delta w = V\, (1 + w)$, where $V = C \left(1 + |u|^{2^\ast - 2}\right) \in L^{N/2}_\loc(\R^N)$ since $\left(2^\ast - 2\right) N/2 = 2^\ast$ and $u \in L^{2^\ast}(\R^N)$. The Br\'{e}zis-Kato iteration (see \cite{MR539217}) now gives $u \in L^m_\loc(\R^N)$ for all $m < \infty$. So $C \left(|u|^{s_1 - 1} + |u|^{s_2 - 1}\right) \in L^m_\loc(\R^N)$ for all $m < \infty$ and hence the local boundedness estimate for subsolutions gives $u \in L^\infty_\loc(\R^N)$ (see, e.g., \cite{MR1814364}). Then $g \in L^\infty_\loc(\R^N \setminus \set{0})$ since $\Phi_u$ is continuous on $(0,\infty)$, so elliptic regularity gives $u \in W^{2,m}_\loc(\R^N \setminus \set{0})$ for all $m < \infty$ and hence $u \in C^{1,\vartheta}_\loc(\R^N \setminus \set{0})$ for some $\vartheta \in (0,1)$.

Let $0 < r \le 1$. Since $u \in L^\infty(B_1)$, $|u(r)| \le C$ and hence $|f(u)| \le C$ by \eqref{26} and $h_u(r) \le Cr^N$. Then
\[
\Phi_u(r) \le C \int_r^1 \rho^{N\, (q - 1) - b}\, d\rho + \Phi_u(1) \le Cr^{- \theta}
\]
and
\[
|g(r)| \le C\, (1 + \Phi_u(r)) \le Cr^{- \theta}.
\]
The radial form $- \left(r^{N-1}\, u'(r)\right)' = r^{N-1}\halfthin g$ of $- \Delta u = g$ now gives
\[
|u'(r)| = \abs{-r^{-(N-1)} \int_0^r \tau^{N-1}\, g(\tau)\, d\tau} \le Cr^{-(N-1)} \int_0^r \tau^{N - \theta - 1}\, d\tau = Cr^{1 - \theta}
\]
since $\theta < 2 \le N$ and
\[
|u''(r)| = \abs{-g(r) - \frac{N - 1}{r}\, u'(r)} \le |g(r)| + \frac{N - 1}{r}\, |u'(r)| \le Cr^{- \theta}. \QED
\]
\end{proof}

\begin{remark} \label{rmk:distributional}
A weak radial solution $u$ of equation \eqref{1} is a solution in the sense of distributions, i.e., the defining identity holds for all test functions $v \in C^\infty_c(\R^N)$, radial or not. Indeed, $\Phi_u(|x|)\, |u|^{p-1}$ and $f(u)$ belong to $L^1_\loc(\R^N)$ (see the proof of Theorem \ref{thm:bm-regularity}), and for $v \in C^\infty_c(\R^N)$ the spherical average
\[
\bar{v}(x) = \frac{1}{\omega_{N-1}} \int_{\bdry{B_1}} v(|x|\, \xi)\, dS
\]
belongs to $C^\infty_{c,\rad}(\R^N) \subset E^{p,q}_{b,\rad}(\R^N)$. Since $u$ and $\Phi_u$ are radial and spherical averaging commutes with radial differentiation,
\begin{gather*}
\int_{\R^N} \nabla u \cdot \nabla v\, dx = \int_{\R^N} \nabla u \cdot \nabla \bar{v}\, dx, \qquad \int_{\R^N} f(u)\, v\, dx = \int_{\R^N} f(u)\, \bar{v}\, dx,\\[7.5pt]
\int_{\R^N} \Phi_u(|x|)\, |u|^{p-2}\, uv\, dx = \int_{\R^N} \Phi_u(|x|)\, |u|^{p-2}\, u \bar{v}\, dx,
\end{gather*}
and the claim follows. We note that the principle of symmetric criticality cannot be invoked directly here since the term $\int_{\R^N} F(u)\, dx$ is in general not defined on the whole space $E^{p,q}_b(\R^N)$ by Theorem \ref{thm:no-embedding}.
\end{remark}

\subsection{Ball-mass Poho\v{z}aev identity}

We have the following analog of the Poho\v{z}aev identity \cite{MR695535,MR2431434} for equation \eqref{1}. Assumptions $(A_{11})$ and $(f_4)$ are special cases of this.

\begin{theorem} \label{thm:bm-pohozaev}
\hspace{-2.37pt}Assume \eqref{33}--\eqref{2} and \eqref{26}. \hspace{-2.37pt}Then every weak radial solution $u \in E^{p,q}_{b,\rad}(\R^N)$ of the equation
\begin{equation} \label{59}
- \Delta u + \Phi_u(|x|)\, |u|^{p-2}\, u = f(u) \quad \text{in } \R^N
\end{equation}
satisfies the Poho\v{z}aev identity
\begin{equation} \label{66}
\frac{N - 2}{2} \int_{\R^N} |\nabla u|^2\, dx + \frac{a\, (b^\# - b)}{pq} \int_0^\infty \rho^{-b}\, h_u^q(\rho)\, d\rho = N \int_{\R^N} F(u)\, dx,
\end{equation}
the Nehari identity
\begin{equation} \label{67}
\int_{\R^N} |\nabla u|^2\, dx + a \int_0^\infty \rho^{-b}\, h_u^q(\rho)\, d\rho = \int_{\R^N} f(u)\, u\, dx,
\end{equation}
and the scaling identity
\begin{equation} \label{60}
\sigma \left(\frac{1}{2} \int_{\R^N} |\nabla u|^2\, dx + \frac{a}{pq} \int_0^\infty \rho^{-b}\, h_u^q(\rho)\, d\rho\right) = \int_{\R^N} \big(\alpha\halfthin f(u)\, u - N F(u)\big)\, dx,
\end{equation}
where $F(t) = \int_0^t f(\tau)\, d\tau$ is the primitive of $f$.
\end{theorem}

\begin{proof}
By Theorem \ref{thm:bm-regularity}, $u \in L^\infty_\loc(\R^N) \cap W^{2,m}_\loc(\R^N \setminus \set{0}) \cap C^{1,\vartheta}_\loc(\R^N \setminus \set{0})$ for all $m < \infty$ and some $\vartheta \in (0,1)$. Let $\xi : [0,\infty) \to [0,1]$ be a smooth function such that $\xi \equiv 1$ on $[0,1]$, $\xi \equiv 0$ on $[2,\infty)$, and $|\xi'| \le 2$. For $R \ge 1$, set $\xi_R(t) = \xi(t/R)$ and
\[
w_R(x) = \xi_R(|x|)\halfthin (x \cdot \nabla u) = r\, \xi_R(r)\, u'(r).
\]
Then
\begin{equation} \label{68}
\nabla w_R = \big[\xi_R(r)\, u'(r) + r\, \xi_R(r)\, u''(r) + r\, \xi_R'(r)\, u'(r)\big]\, \frac{x}{r}
\end{equation}
and hence
\[
|\nabla w_R| \le |u'(r)| + r\, |u''(r)| \le Cr^{1 - \theta} \quad \text{for } 0 < r \le 1
\]
by \eqref{58}, so
\[
\int_{B_1} |\nabla w_R|^2\, dx \le C \int_0^1 r^{N + 1 - 2 \theta}\, dr < \infty
\]
since $\theta < 2 \le (N + 2)/2$. Since $u \in C^{1,\vartheta}_\loc(\R^N \setminus \set{0})$ and $\xi_R$ has compact support, it follows that $w_R \in E^{p,q}_{b,\rad}(\R^N)$ with compact support. Testing \eqref{59} with $w_R$ gives
\begin{equation} \label{62}
\int_{\R^N} \nabla u \cdot \nabla w_R\, dx + \int_{\R^N} \Phi_u(|x|)\, |u|^{p-2}\, u w_R\, dx = \int_{\R^N} f(u)\, w_R\, dx.
\end{equation}
We will pass to the limit as $R \to \infty$.

First we note that for any $g \in L^1([0,\infty))$,
\begin{equation} \label{71}
\int_0^\infty \xi_R(r)\, g(r)\, dr \to \int_0^\infty g(r)\, dr, \qquad \int_0^\infty r\, \xi_R'(r)\, g(r)\, dr \to 0
\end{equation}
as $R \to \infty$. The first limit follows from the dominated convergence theorem. Since $|r\, \xi_R'(r)| = |(r/R)\, \xi'(r/R)| \le 4$ and $\supp \xi_R' \subset [R,\infty)$,
\[
\abs{\int_0^\infty r\, \xi_R'(r)\, g(r)\, dr} \le 4 \int_R^\infty |g(r)|\, dr \to 0.
\]

Since $u \in H^2_\loc(\R^N \setminus \set{0})$, \eqref{68} gives
\begin{multline*}
\int_{\R^N} \nabla u \cdot \nabla w_R\, dx = \omega_{N-1}\, \bigg(\int_0^\infty r^{N-1}\, \xi_R(r)\, u'(r)^2\, dr + \frac{1}{2} \int_0^\infty r^N\halfthin \xi_R(r)\, (u'(r)^2)'\, dr\\[7.5pt]
+ \int_0^\infty r^N\halfthin \xi_R'(r)\, u'(r)^2\, dr\bigg).
\end{multline*}
Since $r^N\halfthin u'(r)^2 \le Cr^{N + 2 - 2 \theta} \to 0$ as $r \to 0$ by \eqref{58} and $\xi_R$ has compact support, integrating the second integral on the right by parts leads to
\[
\int_{\R^N} \nabla u \cdot \nabla w_R\, dx = \omega_{N-1} \left(- \frac{N - 2}{2} \int_0^\infty r^{N-1}\, \xi_R(r)\, u'(r)^2\, dr + \frac{1}{2} \int_0^\infty r^N\halfthin \xi_R'(r)\, u'(r)^2\, dr\right).
\]
Since $\nabla u \in L^2(\R^N)$, \eqref{71} now shows that
\begin{equation} \label{63}
\int_{\R^N} \nabla u \cdot \nabla w_R\, dx \to - \frac{N - 2}{2} \int_{\R^N} |\nabla u|^2\, dx \quad \text{as } R \to \infty.
\end{equation}

We have
\[
\int_{\R^N} \Phi_u(|x|)\, |u|^{p-2}\, u w_R\, dx = \frac{\omega_{N-1}}{p} \int_0^\infty r^N\halfthin \Phi_u(r)\, \xi_R(r)\, (|u(r)|^p)'\, dr.
\]
Since $r^N\halfthin \Phi_u(r)\, |u(r)|^p \le Cr^{N - \theta} \to 0$ as $r \to 0$ by \eqref{58} and $\xi_R$ has compact support, integrating by parts gives
\begin{multline*}
\int_{\R^N} \Phi_u(|x|)\, |u|^{p-2}\, u w_R\, dx = - \frac{\omega_{N-1}}{p}\, \bigg(N \int_0^\infty r^{N-1}\, \Phi_u(r)\, \xi_R(r)\, |u(r)|^p\, dr\\[7.5pt]
+ \int_0^\infty r^N\halfthin \Phi_u'(r)\, \xi_R(r)\, |u(r)|^p\, dr + \int_0^\infty r^N\halfthin \Phi_u(r)\, \xi_R'(r)\, |u(r)|^p\, dr\bigg).
\end{multline*}
Since $\Phi_u'(r) = - ar^{-b}\, h_u^{q-1}(r)$ and $h_u'(r) = \omega_{N-1}\, r^{N-1}\, |u(r)|^p$,
\begin{multline*}
\omega_{N-1} \int_0^\infty r^N\halfthin \Phi_u'(r)\, \xi_R(r)\, |u(r)|^p\, dr = - a \int_0^\infty r^{1-b}\, \xi_R(r)\, h_u^{q-1}(r)\, h_u'(r)\, dr\\[7.5pt]
= - \frac{a}{q} \int_0^\infty r^{1-b}\, \xi_R(r)\, (h_u^q(r))'\, dr.
\end{multline*}
Since $r^{1-b}\, h_u^q(r) \le Cr^{b^\# - b} \to 0$ as $r \to 0$ by \eqref{58}, \eqref{2}, and \eqref{36}, a second integration by parts leads to
\begin{multline*}
\int_{\R^N} \Phi_u(|x|)\, |u|^{p-2}\, u w_R\, dx = - \frac{N \omega_{N-1}}{p} \int_0^\infty r^{N-1}\, \Phi_u(r)\, \xi_R(r)\, |u(r)|^p\, dr\\[7.5pt]
- \frac{a\, (1 - b)}{pq} \int_0^\infty r^{-b}\, \xi_R(r)\, h_u^q(r)\, dr - \frac{a}{pq} \int_0^\infty r^{1-b}\, \xi_R'(r)\, h_u^q(r)\, dr\\[7.5pt]
- \frac{\omega_{N-1}}{p} \int_0^\infty r^N\halfthin \Phi_u(r)\, \xi_R'(r)\, |u(r)|^p\, dr.
\end{multline*}
Corollary \ref{cor:bm-pairing} together with \eqref{73} gives
\begin{equation} \label{74}
\omega_{N-1} \int_0^\infty r^{N-1}\, \Phi_u(r)\, |u(r)|^p\, dr = a \int_0^\infty \rho^{-b}\, h_u^q(\rho)\, d\rho < \infty,
\end{equation}
so \eqref{71} now shows that
\begin{equation} \label{64}
\int_{\R^N} \Phi_u(|x|)\, |u|^{p-2}\, u w_R\, dx \to - \frac{a\, (b^\# - b)}{pq} \int_0^\infty \rho^{-b}\, h_u^q(\rho)\, d\rho \quad \text{as } R \to \infty.
\end{equation}

We have
\[
\int_{\R^N} f(u)\, w_R\, dx = \omega_{N-1} \int_0^\infty r^N\halfthin \xi_R(r)\, (F(u))'\, dr.
\]
Since $r^N\halfthin F(u) \le Cr^N \to 0$ as $r \to 0$ by \eqref{58} and $\xi_R$ has compact support, integrating by parts gives
\[
\int_{\R^N} f(u)\, w_R\, dx = - \omega_{N-1} \left(N \int_0^\infty r^{N-1}\, \xi_R(r)\, F(u)\, dr + \int_0^\infty r^N\halfthin \xi_R'(r)\, F(u)\, dr\right).
\]
Since $F(u) \in L^1(\R^N)$ by \eqref{26} and Theorem \ref{Theorem 1}, \eqref{71} now shows that
\begin{equation} \label{65}
\int_{\R^N} f(u)\, w_R\, dx \to - N \int_{\R^N} F(u)\, dx \quad \text{as } R \to \infty.
\end{equation}

Combining \eqref{62} with \eqref{63}, \eqref{64}, and \eqref{65} gives \eqref{66}. In view of \eqref{74}, testing \eqref{59} with $u$ gives \eqref{67}. Multiplying \eqref{67} by $\alpha$, subtracting \eqref{66}, and noting that
\[
\alpha = \frac{\sigma}{2} + \frac{N - 2}{2}, \qquad \frac{b^\# - b}{pq} = \left(\frac{1}{2} - \frac{1}{pq}\right) \sigma + \frac{N - 2}{2}
\]
gives \eqref{60}.
\end{proof}

\subsection{Scaled eigenvalue problem} \label{ssec:eigen-prob}

Consider the scaled eigenvalue problem \eqref{52}, i.e.,
\begin{equation} \label{51}
- \Delta u + \Phi_u(|x|)\, |u|^{p-2}\, u = \lambda\, |u|^{\beta - 2}\, u \quad \text{in } \R^N,
\end{equation}
where $\beta$ is given in \eqref{31}. By \eqref{35},
\begin{equation} \label{303}
2 < \beta < \begin{cases}
\infty & \text{if } N = 2\\[5pt]
2^\ast & \text{if } N \ge 3.
\end{cases}
\end{equation}
We have
\begin{equation} \label{505}
\beta - 2^{p,q}_{b,\ast} = \begin{cases}
\dfrac{4\, (q - b)}{b_\ast - b} & \text{if } N = 2\\[15pt]
\dfrac{4\, [(N - 1)(q - 1) + 1 - b](b_\ast - b)}{(b_\# - b)(b^\ast - b)} & \text{if } N \ge 3.
\end{cases}
\end{equation}
We assume that
\begin{equation} \label{53}
b < (N - 1)(q - 1) + 1
\end{equation}
so that $\beta > 2^{p,q}_{b,\ast}$ by \eqref{2} and \eqref{36}, and hence $E^{p,q}_{b,\rad}(\R^N)$ is compactly embedded in $L^\beta(\R^N)$ by Theorem \ref{Theorem 1}.

The scaled operator that corresponds to the right-hand side of equation \eqref{51} and its potential are given by
\[
\Bs(u) v = \int_{\R^N} |u|^{\beta - 2}\, uv\, dx, \quad u, v \in E^{p,q}_{b,\rad}(\R^N)
\]
and
\begin{equation} \label{305}
J_\sigma(u) = \frac{1}{\beta} \int_{\R^N} |u|^\beta\, dx, \quad u \in E^{p,q}_{b,\rad}(\R^N),
\end{equation}
respectively. Assumption $(A_8)$ is clearly satisfied, while $(A_9)$ follows from the compactness of the embedding $E^{p,q}_{b,\rad}(\R^N) \hookrightarrow L^\beta(\R^N)$. Assumption $(A_{11})$ is the special case $f(t) = \lambda\, |t|^{\beta - 2}\, t$ of \eqref{60}. The following theorem is now immediate from Theorem \ref{Theorem 13}.

\begin{theorem} \label{thm:bm-eigenvalues}
Assume \eqref{33}--\eqref{2} and \eqref{53}. Then the scaled eigenvalue problem \eqref{51} has a sequence of positive eigenvalues $\lambda_k \nearrow \infty$ with radial eigenfunctions.
\end{theorem}

\subsection{Pure power equation}

Consider the pure power equation
\begin{equation} \label{star}
- \Delta u + \Phi_u(|x|)\, |u|^{p-2}\, u = |u|^{s-2}\, u \quad \text{in } \R^N,
\end{equation}
where $s \ne \beta$ (see \eqref{31}) satisfies
\begin{equation} \label{56}
2^{p,q}_{b,\ast} < s < \begin{cases}
\infty & \text{if } N = 2\\[5pt]
2^\ast & \text{if } N \ge 3.
\end{cases}
\end{equation}
The scaled operator that corresponds to the right-hand side of equation \eqref{star} and its potential are given by
\[
f(u) v = \int_{\R^N} |u|^{s-2}\, uv\, dx, \quad u, v \in E^{p,q}_{b,\rad}(\R^N)
\]
and
\[
F(u) = \frac{1}{s} \int_{\R^N} |u|^s\, dx, \quad u \in E^{p,q}_{b,\rad}(\R^N),
\]
respectively. Assumption $(f_1)$ in Subsection \ref{ssec:sub-sup} holds with
\[
\gamma = \alpha s - N \ne \sigma
\]
since $s \ne \beta$ and $\alpha > 0$ by \eqref{35}. Assumption $(f_2)$ is clearly satisfied and $(f_3)$ follows from the compactness of the embedding $E^{p,q}_{b,\rad}(\R^N) \hookrightarrow L^s(\R^N)$. Assumption $(f_4)$ is the special case $f(t) = \mu\, |t|^{s-2}\, t$ of \eqref{60}. The following theorem is now immediate from Theorem \ref{Theorem 10}, where
\[
\E(u) = \frac{1}{2} \int_{\R^N} |\nabla u|^2\, dx + \frac{a}{pq} \int_0^\infty \rho^{-b}\, h_u^q(\rho)\, d\rho - \frac{1}{s} \int_{\R^N} |u|^s\, dx, \quad u \in E^{p,q}_{b,\rad}(\R^N).
\]

\begin{theorem} \label{thm:pure-power}
Assume \eqref{33}--\eqref{2}. If $s \ne \beta$ satisfies \eqref{56}, then the pure power equation \eqref{star} has an infinite sequence of nontrivial weak radial solutions $u_k$.
\begin{enumroman}
\item If $s < \beta$, then equation \eqref{star} is subscaled, $\E(u_k) < 0$, and $\E(u_k) \nearrow 0$.
\item If $s > \beta$, then equation \eqref{star} is superscaled, $\E(u_k) > 0$, and $\E(u_k) \nearrow \infty$.
\end{enumroman}
\end{theorem}

\subsection{Critical equations} \label{ssec:critical}

Consider the pure critical power equation
\begin{equation} \label{pure-crit}
- \Delta u + \Phi_u(|x|)\, |u|^{p-2}\, u = |u|^{2^\ast - 2}\, u \quad \text{in } \R^N,
\end{equation}
where $N \ge 3$,
\begin{gather}
\label{110} 1 < p < 2^\ast,\\[10pt]
\label{111} 1 \le q < \infty, \quad 2 < pq < 2^\ast,\\[10pt]
\label{112} 1 < b < b_\ast := 1 + Nq\, (1 - p/2^\ast).
\end{gather}
The Poho\v{z}aev and Nehari identities in Theorem \ref{thm:bm-pohozaev} give the following nonexistence result.

\begin{theorem} \label{thm:pure-crit-nonex}
Assume \eqref{110}--\eqref{112}. Then equation \eqref{pure-crit} has no nontrivial weak radial solution.
\end{theorem}

\begin{proof}
Let $u \in E^{p,q}_{b,\rad}(\R^N)$ be a weak radial solution of equation \eqref{pure-crit}. For $f(t) = |t|^{2^\ast - 2}\, t$, equations \eqref{66} and \eqref{67} reduce to
\begin{equation} \label{113}
\int_{\R^N} |\nabla u|^2\, dx + \frac{2a\, (b^\# - b)}{pq\, (N - 2)} \int_0^\infty \rho^{-b}\, h_u^q(\rho)\, d\rho = \int_{\R^N} |u|^{2^\ast}\halfthin dx
\end{equation}
and
\begin{equation} \label{114}
\int_{\R^N} |\nabla u|^2\, dx + a \int_0^\infty \rho^{-b}\, h_u^q(\rho)\, d\rho = \int_{\R^N} |u|^{2^\ast}\halfthin dx,
\end{equation}
respectively. By \eqref{112}, \eqref{29}, and \eqref{109},
\begin{equation} \label{115}
\frac{2\, (b^\# - b)}{pq\, (N - 2)} > \frac{2\, (b^\# - b_\ast)}{pq\, (N - 2)} = 1.
\end{equation}
It follows from \eqref{113}--\eqref{115} that $\norm[L^{p,q}_b(\R^N)]{u} = 0$ and hence $u = 0$.
\end{proof}

In view of this nonexistence result for the pure critical power equation, we consider the perturbed critical equation
\begin{equation} \label{pert-crit}
- \Delta u + \Phi_u(|x|)\, |u|^{p-2}\, u = \mu\, |u|^{s-2}\, u + |u|^{2^\ast - 2}\, u \quad \text{in } \R^N,
\end{equation}
where $\mu > 0$ and $s$ satisfies either
\begin{equation} \label{118}
\max \set{2^{p,q}_{b,\ast},pq} < s < 2^\ast,
\end{equation}
or
\begin{equation} \label{117}
2^{p,q}_{b,\ast} < \beta \le s < 2^\ast
\end{equation}
(see \eqref{31}). The variational functional associated with this equation is
\[
\E(u) = I_\sigma(u) - \frac{\mu}{s} \int_{\R^N} |u|^s\, dx - \frac{1}{2^\ast} \int_{\R^N} |u|^{2^\ast}\halfthin dx, \quad u \in E^{p,q}_{b,\rad}(\R^N)
\]
(see \eqref{101}). We have the following local \PS{} condition, where
\begin{equation} \label{88}
S = \inf_{u \in \D^{1,\,2}(\R^N) \setminus \set{0}}\, \frac{\dint_{\R^N} |\nabla u|^2\, dx}{\left(\dint_{\R^N} |u|^{2^\ast}\halfthin dx\right)^{2/2^\ast}}
\end{equation}
is the best Sobolev constant. Local compactness below the first noncompactness level is the standard route for critical growth problems and goes back to Br\'{e}zis and Nirenberg \cite{MR709644} (see also \cite{MR778970,MR834360,MR760051,MR2431434}).

\begin{proposition} \label{prop:local-ps}
Assume \eqref{110}--\eqref{112}. If $\mu > 0$ and $s$ satisfies \eqref{118} or \eqref{117}, then $\E$ satisfies the {\em \PS{c}} condition for all $c < \frac{1}{N}\, S^{N/2}$.
\end{proposition}

\begin{proof}
Let $c < \frac{1}{N}\, S^{N/2}$ and let $\seq{u_j} \subset E^{p,q}_{b,\rad}(\R^N)$ be a \PS{c} sequence, i.e.,
\begin{equation} \label{81}
\E(u_j) = I_\sigma(u_j) - \frac{\mu}{s} \int_{\R^N} |u_j|^s\, dx - \frac{1}{2^\ast} \int_{\R^N} |u_j|^{2^\ast}\halfthin dx = c + \o(1)
\end{equation}
and
\begin{equation} \label{83}
\E'(u_j) v = \As(u_j) v - \mu \int_{\R^N} |u_j|^{s-2}\, u_j\halfthin v\, dx - \int_{\R^N} |u_j|^{2^\ast - 2}\, u_j\halfthin v\, dx = \o(\norm{v})
\end{equation}
for all $v \in E^{p,q}_{b,\rad}(\R^N)$ (see \eqref{100}). Taking $v = u_j$ in \eqref{83} gives
\begin{equation} \label{82}
\As(u_j) u_j - \mu \int_{\R^N} |u_j|^s\, dx - \int_{\R^N} |u_j|^{2^\ast}\halfthin dx = \o(\norm{u_j}).
\end{equation}

We begin by showing that $\seq{u_j}$ is bounded in $E^{p,q}_{b,\rad}(\R^N)$ if $s$ satisfies \eqref{118} or \eqref{117}. First suppose \eqref{118} holds. Dividing \eqref{82} by $s$ and subtracting from \eqref{81} gives
\begin{multline*}
\left(\frac{1}{2} - \frac{1}{s}\right) \int_{\R^N} |\nabla u_j|^2\, dx + a \left(\frac{1}{pq} - \frac{1}{s}\right) \int_0^\infty \rho^{-b}\, h_{u_j}^q(\rho)\, d\rho + \left(\frac{1}{s} - \frac{1}{2^\ast}\right) \int_{\R^N} |u_j|^{2^\ast}\halfthin dx\\[7.5pt]
\le C + \o(\norm{u_j})
\end{multline*}
for some constant $C > 0$ (see \eqref{134}). Since $2 < pq < s < 2^\ast$ by \eqref{111} and \eqref{118}, boundedness of $\seq{u_j}$ follows in this case. Now suppose \eqref{117} holds, but $\norm{u_j} \to \infty$ for a renamed subsequence. Set
\[
t_j = t_{u_j} = I_\sigma(u_j)^{- 1/\sigma}, \qquad \widetilde{u}_j = (u_j)_{t_j}, \qquad \tilde{t}_j = t_j^{-1} = I_\sigma(u_j)^{1/\sigma}
\]
(see \eqref{102}). Then $\widetilde{u}_j \in \M$ and
\begin{equation} \label{167}
u_j = (\widetilde{u}_j)_{\tilde{t}_j} = \tilde{t}_j^{\,\alpha}\, \widetilde{u}_j(\tilde{t}_j\, \cdot)
\end{equation}
by $(A_1)$, $(A_3)$, and \eqref{103}. Since $\M$ is a bounded manifold, $\seq{\widetilde{u}_j}$ is bounded. Since $\norm{u_j} \to \infty$, $I_\sigma(u_j) \to \infty$ by $(A_{10})$ and hence $\tilde{t}_j \to \infty$. By \eqref{38},
\begin{equation} \label{131}
\norm{u_j} = \|(\widetilde{u}_j)_{\tilde{t}_j}\| \le \tilde{t}_j^{\,\sigma/2}
\end{equation}
for all sufficiently large $j$ since $pq > 2$ by \eqref{111}. Using \eqref{167}, \eqref{104}, \eqref{105}, and \eqref{131}, we can write \eqref{81} and \eqref{82} as
\begin{equation} \label{168}
\tilde{t}_j^{\,\sigma} I_\sigma(\widetilde{u}_j) = \frac{\tilde{t}_j^{\,\gamma^\ast}}{2^\ast} \int_{\R^N} |\widetilde{u}_j|^{2^\ast}\halfthin dx + \frac{\mu\, \tilde{t}_j^{\,\gamma}}{s} \int_{\R^N} |\widetilde{u}_j|^s\, dx + c + \o(1)
\end{equation}
and
\begin{equation} \label{169}
\tilde{t}_j^{\,\sigma} \As(\widetilde{u}_j) \widetilde{u}_j = \tilde{t}_j^{\,\gamma^\ast} \int_{\R^N} |\widetilde{u}_j|^{2^\ast}\halfthin dx + \mu\, \tilde{t}_j^{\,\gamma} \int_{\R^N} |\widetilde{u}_j|^s\, dx + \o(\tilde{t}_j^{\,\sigma/2}),
\end{equation}
respectively, where
\begin{equation} \label{133}
\gamma = \alpha s - N, \qquad \gamma^\ast = \alpha\halfthin 2^\ast - N.
\end{equation}
By \eqref{55}, \eqref{35}, and \eqref{117}, $\gamma - \sigma = \alpha\, (s - \beta) \ge 0$ and $\gamma^\ast - \gamma = \alpha\, (2^\ast - s) > 0$, so
\begin{equation} \label{130}
\sigma \le \gamma < \gamma^\ast.
\end{equation}
Since $\tilde{t}_j \to \infty$, dividing \eqref{169} by $\tilde{t}_j^{\,\sigma}$ gives
\[
\tilde{t}_j^{\,\gamma^\ast - \sigma} \int_{\R^N} |\widetilde{u}_j|^{2^\ast}\halfthin dx + \mu\, \tilde{t}_j^{\,\gamma - \sigma} \int_{\R^N} |\widetilde{u}_j|^s\, dx = \As(\widetilde{u}_j) \widetilde{u}_j + \o(1).
\]
Since $\mu > 0$ and $\As$ maps bounded sets into bounded sets, this implies
\[
\int_{\R^N} |\widetilde{u}_j|^{2^\ast}\halfthin dx \le C\halfthin \tilde{t}_j^{\,-(\gamma^\ast - \sigma)}
\]
for some constant $C > 0$. Fix $r \in (2^{p,q}_{b,\ast},\beta)$ using \eqref{117}. Then $r \in (2^{p,q}_{b,\ast},2^\ast)$ and hence $\seq{\widetilde{u}_j}$ is bounded in $L^r(\R^N)$ by Theorem \ref{Theorem 1}. Now the interpolation inequality
\[
\norm[L^s(\R^N)]{\widetilde{u}_j} \le \norm[L^r(\R^N)]{\widetilde{u}_j}^{1 - \theta} \norm[L^{2^\ast}(\R^N)]{\widetilde{u}_j}^\theta,
\]
where $1/s = (1 - \theta)/r + \theta/2^\ast$, gives
\[
\tilde{t}_j^{\,\gamma - \sigma} \int_{\R^N} |\widetilde{u}_j|^s\, dx \le C\halfthin \tilde{t}_j^{\,- \alpha\, (2^\ast - s)(\beta - r)/(2^\ast - r)} \to 0
\]
since $r < \beta \le s < 2^\ast$. Now multiplying \eqref{168} by $2^\ast$, subtracting \eqref{169}, and dividing by $\tilde{t}_j^{\,\sigma}$ gives
\[
\frac{2}{N - 2} \int_{\R^N} |\nabla \widetilde{u}_j|^2\, dx + a \left(\frac{2^\ast}{pq} - 1\right) \int_0^\infty \rho^{-b}\, h_{\widetilde{u}_j}^q(\rho)\, d\rho = \o(1).
\]
This implies that $\widetilde{u}_j \to 0$ in $E^{p,q}_{b,\rad}(\R^N)$ since $pq < 2^\ast$ by \eqref{111}, contradicting $\widetilde{u}_j \in \M$. So $\seq{u_j}$ is bounded in this case also.

Since $\seq{u_j}$ is bounded and $E^{p,q}_{b,\rad}(\R^N)$ is reflexive by Theorem \ref{Theorem 4}, a renamed subsequence of $\seq{u_j}$ converges weakly to some $u \in E^{p,q}_{b,\rad}(\R^N)$. Since $s \in (2^{p,q}_{b,\ast},2^\ast)$, then $u_j$ also converges to $u$ strongly in $L^s(\R^N)$ by Theorem \ref{Theorem 1} and hence a.e.\! in $\R^N$ for a further subsequence. Now passing to the limit in \eqref{83} using Theorem \ref{thm:weak-cont} gives
\begin{equation} \label{132}
\As(u) v - \mu \int_{\R^N} |u|^{s-2}\, uv\, dx - \int_{\R^N} |u|^{2^\ast - 2}\, uv\, dx = 0
\end{equation}
for all $v \in E^{p,q}_{b,\rad}(\R^N)$, i.e., $u$ is a weak radial solution of equation \eqref{pert-crit}. Taking $v = u$ in \eqref{132} gives
\begin{equation} \label{84}
\As(u) u - \mu \int_{\R^N} |u|^s\, dx - \int_{\R^N} |u|^{2^\ast}\halfthin dx = 0.
\end{equation}

Next we show that $\E(u) \ge 0$. If $s$ satisfies \eqref{118}, dividing \eqref{84} by $s$ and subtracting from $\E(u)$ gives
\[
\E(u) = \left(\frac{1}{2} - \frac{1}{s}\right) \int_{\R^N} |\nabla u|^2\, dx + a \left(\frac{1}{pq} - \frac{1}{s}\right) \int_0^\infty \rho^{-b}\, h_u^q(\rho)\, d\rho + \left(\frac{1}{s} - \frac{1}{2^\ast}\right) \int_{\R^N} |u|^{2^\ast}\halfthin dx \ge 0
\]
since $2 < pq < s < 2^\ast$ (see \eqref{111}). If $s$ satisfies \eqref{117}, we use the scaling identity from Theorem \ref{thm:bm-pohozaev}. For the nonlinearity $f(t) = \mu\, |t|^{s-2}\, t + |t|^{2^\ast - 2}\, t$,
\[
\alpha\halfthin f(t)\, t - N F(t) = \mu \left(\alpha - \frac{N}{s}\right) |t|^s + \left(\alpha - \frac{N}{2^\ast}\right) |t|^{2^\ast} = \frac{\mu \gamma}{s}\, |t|^s + \frac{\gamma^\ast}{2^\ast}\, |t|^{2^\ast},
\]
where $\gamma$ and $\gamma^\ast$ are given in \eqref{133}, and hence \eqref{60} gives
\[
\sigma I_\sigma(u) = \frac{\mu \gamma}{s} \int_{\R^N} |u|^s\, dx + \frac{\gamma^\ast}{2^\ast} \int_{\R^N} |u|^{2^\ast}\halfthin dx.
\]
So
\[
\E(u) = \frac{\mu}{s}\, \bigg(\frac{\gamma}{\sigma} - 1\bigg) \int_{\R^N} |u|^s\, dx + \frac{1}{2^\ast} \left(\frac{\gamma^\ast}{\sigma} - 1\right) \int_{\R^N} |u|^{2^\ast}\halfthin dx \ge 0
\]
by \eqref{130}.

Now set $v_j = u_j - u$. We will show that $v_j \to 0$ in $E^{p,q}_{b,\rad}(\R^N)$ for a renamed subsequence. We have
\begin{equation} \label{85}
\int_{\R^N} |\nabla u_j|^2\, dx - \int_{\R^N} |\nabla u|^2\, dx = \int_{\R^N} |\nabla v_j|^2\, dx + \o(1)
\end{equation}
since $v_j \wto 0$ in $E^{p,q}_{b,\rad}(\R^N)$,
\begin{equation} \label{135}
\int_{\R^N} |u_j|^s\, dx - \int_{\R^N} |u|^s\, dx = \o(1)
\end{equation}
since $v_j \to 0$ in $L^s(\R^N)$, and
\begin{equation} \label{86}
\int_{\R^N} |u_j|^{2^\ast}\halfthin dx - \int_{\R^N} |u|^{2^\ast}\halfthin dx = \int_{\R^N} |v_j|^{2^\ast}\halfthin dx + \o(1)
\end{equation}
by the Br\'{e}zis-Lieb lemma. Since $u_j \to u$ in $L^s(\R^N)$ and $p < s$ by \eqref{106} and \eqref{118} or \eqref{117}, $u_j \to u$ in $L^p_\loc(\R^N)$ by the H\"{o}lder inequality, so
\begin{equation} \label{90}
\int_0^\infty \rho^{-b}\, h_{u_j}^q(\rho)\, d\rho - \int_0^\infty \rho^{-b}\, h_u^q(\rho)\, d\rho = \int_0^\infty \rho^{-b}\, h_{v_j}^q(\rho)\, d\rho + \o(1)
\end{equation}
by Theorem \ref{thm:bm-brezis-lieb}. Subtracting \eqref{84} from \eqref{82} and combining with \eqref{85}--\eqref{90} and \eqref{88} gives
\begin{multline} \label{87}
\int_{\R^N} |\nabla v_j|^2\, dx + a \int_0^\infty \rho^{-b}\, h_{v_j}^q(\rho)\, d\rho = \int_{\R^N} |v_j|^{2^\ast}\halfthin dx + \o(1)\\[7.5pt]
\le S^{- 2^\ast/2} \bigg(\int_{\R^N} |\nabla v_j|^2\, dx\bigg)^{2^\ast/2} + \o(1).
\end{multline}
So it suffices to show that $\int_{\R^N} |\nabla v_j|^2\, dx \to 0$ for a renamed subsequence. Suppose this is not the case. Then \eqref{87} gives
\begin{equation} \label{89}
\int_{\R^N} |\nabla v_j|^2\, dx \ge S^{N/2} + \o(1).
\end{equation}
Subtracting $\E(u)$ from \eqref{81} and combining with \eqref{85}--\eqref{90}, $\E(u) \ge 0$, and \eqref{87} gives
\begin{multline*}
c = \frac{1}{2} \int_{\R^N} |\nabla v_j|^2\, dx + \frac{a}{pq} \int_0^\infty \rho^{-b}\, h_{v_j}^q(\rho)\, d\rho - \frac{1}{2^\ast} \int_{\R^N} |v_j|^{2^\ast}\halfthin dx + \E(u) + \o(1)\\[7.5pt]
\ge \frac{1}{N} \int_{\R^N} |\nabla v_j|^2\, dx + a \left(\frac{1}{pq} - \frac{1}{2^\ast}\right) \int_0^\infty \rho^{-b}\, h_{v_j}^q(\rho)\, d\rho + \o(1).
\end{multline*}
Since $pq < 2^\ast$ by \eqref{111}, this together with \eqref{89} gives $c \ge \frac{1}{N}\, S^{N/2}$, a contradiction.
\end{proof}

First we consider the perturbed equation
\begin{equation} \label{136}
- \Delta u + \Phi_u(|x|)\, |u|^{p-2}\, u = \lambda\, |u|^{\beta - 2}\, u + |u|^{2^\ast - 2}\, u \quad \text{in } \R^N,
\end{equation}
where $\lambda > 0$. We assume that
\begin{equation} \label{137}
\beta > 2^{p,q}_{b,\ast}.
\end{equation}
Let $\lambda_k \nearrow \infty$ be the sequence of positive eigenvalues of the scaled eigenvalue problem \eqref{51} given in Theorem \ref{thm:bm-eigenvalues}. We will show that equation \eqref{136} has $m$ distinct pairs of nontrivial weak radial solutions for all $\lambda$ in a suitably small left neighborhood of any eigenvalue of multiplicity $m \ge 1$.

\begin{theorem} \label{thm:left-nbhd}
Assume \eqref{110}--\eqref{112} and \eqref{137}. If $\lambda_k = \cdots = \lambda_{k+m-1} < \lambda_{k+m}$ for some $k, m \ge 1$, then $\exists \delta_k > 0$ such that equation \eqref{136} has $m$ distinct pairs of nontrivial weak radial solutions at positive energy levels for all $\lambda \in (\lambda_k - \delta_k,\lambda_k)$. In particular, there exists a nontrivial solution for each $\lambda \in \bigcup_{k=1}^\infty (\lambda_k - \delta_k,\lambda_k)$.
\end{theorem}

\begin{proof}
In view of Proposition \ref{prop:local-ps}, we apply Theorem \ref{thm:even-non-min} with $c^\ast = \frac{1}{N}\, S^{N/2}$. Let $\eps \in (0,\lambda_{k+m} - \lambda_{k+m-1})$. Then
\[
i(\M \setminus \widetilde{\Psi}_{\lambda_{k+m-1} + \eps}) = k + m - 1
\]
by Theorem \ref{Theorem 13} \ref{Theorem 13.iii}. Since $\M \setminus \widetilde{\Psi}_{\lambda_{k+m-1} + \eps}$ is an open symmetric subset of $\M$, then it has a compact symmetric subset $C$ of index $k + m - 1$ (see the proof of Proposition 3.1 in \cite{MR2371112}). We apply Theorem \ref{thm:even-non-min} with $A_0 = C$ and $B_0 = \widetilde{\Psi}_{\lambda_k}$. We have either $\lambda_1 = \cdots = \lambda_k$, or $\lambda_{l-1} < \lambda_l = \cdots = \lambda_k$ for some $l$ with $2 \le l \le k$. In the former case, $B_0 = \widetilde{\Psi}_{\lambda_1} = \M$ by Theorem \ref{Theorem 13} \ref{Theorem 13.i} and hence
\[
i(\M \setminus B_0) = i(\emptyset) = 0 \le k - 1.
\]
In the latter case,
\[
i(\M \setminus B_0) = i(\M \setminus \widetilde{\Psi}_{\lambda_l}) = l - 1 \le k - 1
\]
by Theorem \ref{Theorem 13} \ref{Theorem 13.iii}.

Let $R > \rho > 0$ and let $A = \set{u_R : u \in A_0}$, $B = \set{u_\rho : u \in B_0}$, and $X = \{u_t : u \in A_0,\, 0 \le t \le R\}$. For $u \in \M$ and $t \ge 0$, we have
\begin{equation} \label{180}
\E(u_t) = t^\sigma \left(1 - \frac{\lambda}{\widetilde{\Psi}(u)} - \frac{t^{\gamma^\ast - \sigma}}{2^\ast} \int_{\R^N} |u|^{2^\ast}\halfthin dx\right),
\end{equation}
where $\gamma^\ast$ is given in \eqref{133}. By \eqref{55} and \eqref{35}, $\gamma^\ast - \sigma = \alpha\, (2^\ast - \beta) > 0$, so $\sigma < \gamma^\ast$. Since $\M$ is bounded, this gives
\[
\E(u_t) \ge t^\sigma \left(1 - \frac{\lambda}{\lambda_k} - C t^{\gamma^\ast - \sigma}\right) \quad \forall u \in B_0,\, t \ge 0
\]
for some constant $C > 0$. So
\[
\inf_{u \in B}\, \E(u) > 0
\]
if $\lambda < \lambda_k$ and $\rho$ is sufficiently small. For $u \in A_0 \subset \M \setminus \widetilde{\Psi}_{\lambda_{k+m-1} + \eps}$,
\[
J_\sigma(u) > \frac{1}{\lambda_{k+m-1} + \eps} > \frac{1}{\lambda_{k+m}}
\]
since $\eps < \lambda_{k+m} - \lambda_{k+m-1}$ (see \eqref{305}). Since $\M$ is bounded in $L^r(\R^N)$ for $r \in (2^{p,q}_{b,\ast},\beta)$ by Theorem \ref{Theorem 1}, the interpolation inequality
\[
\norm[L^\beta(\R^N)]{u} \le \norm[L^r(\R^N)]{u}^{1 - \theta} \norm[L^{2^\ast}(\R^N)]{u}^\theta,
\]
where $1/\beta = (1 - \theta)/r + \theta/2^\ast$, then implies that $\norm[L^{2^\ast}(\R^N)]{u}$ is bounded from below by a positive constant independent of $\eps$. Then \eqref{180} together with $\lambda_{k+m-1} = \lambda_k$ gives
\[
\E(u_t) \le t^\sigma \left(1 - \frac{\lambda}{\lambda_k + \eps} - \kappa\halfthin t^{\gamma^\ast - \sigma}\right) \quad \forall u \in A_0,\, t \ge 0
\]
for some constant $\kappa > 0$ independent of $\eps$. So
\[
\sup_{u \in A}\, \E(u) \le R^\sigma \left(1 - \kappa R^{\gamma^\ast - \sigma}\right) \le 0
\]
if $R$ is sufficiently large and
\[
\sup_{u \in X}\, \E(u) \le \max_{t \ge 0}\, t^\sigma \left(1 - \frac{\lambda}{\lambda_k + \eps} - \kappa\halfthin t^{\gamma^\ast - \sigma}\right) = C \left(1 - \frac{\lambda}{\lambda_k + \eps}\right)^{\gamma^\ast/(\gamma^\ast - \sigma)} < \frac{1}{N}\, S^{N/2}
\]
if $\lambda > \lambda_k - \delta_k$ with $\eps$ and $\delta_k > 0$ sufficiently small. The desired conclusion now follows from Theorem \ref{thm:even-non-min}.
\end{proof}

Now we consider the equation
\begin{equation} \label{138}
- \Delta u + \Phi_u(|x|)\, |u|^{p-2}\, u = \mu\, |u|^{s-2}\, u + |u|^{2^\ast - 2}\, u \quad \text{in } \R^N,
\end{equation}
where $\mu > 0$ and $s$ satisfies
\begin{equation} \label{139}
\max \set{2^{p,q}_{b,\ast},\beta} < s < 2^\ast
\end{equation}
and
\begin{equation} \label{700}
s > pq \hquad \text{if} \hquad \beta \le 2^{p,q}_{b,\ast}.
\end{equation}
Then $s$ satisfies \eqref{118} if $\beta \le 2^{p,q}_{b,\ast}$ and \eqref{117} if $\beta > 2^{p,q}_{b,\ast}$, so Proposition \ref{prop:local-ps} applies. We will show that this equation has arbitrarily many weak radial solutions for all sufficiently large $\mu$.

\begin{theorem} \label{thm:large-mu}
Assume \eqref{110}--\eqref{112}, \eqref{139}, and \eqref{700}. Then for any $m \ge 1$, $\exists \mu_m > 0$ such that equation \eqref{138} has $m$ distinct pairs of nontrivial weak radial solutions at positive energy levels for all $\mu > \mu_m$. In particular, the number of solutions goes to infinity as $\mu \to \infty$.
\end{theorem}

\begin{proof}
In view of Proposition \ref{prop:local-ps}, we apply Corollary \ref{Corollary 1} with $c^\ast = \frac{1}{N}\, S^{N/2}$. Let $A_0$ be a compact symmetric subset of $\M$ with $i(A_0) = m$. Let $R > \rho > 0$ and let $A = \set{u_R : u \in A_0}$ and $X = \set{u_t : u \in A_0,\, 0 \le t \le R}$. For $u \in \M$ and $t \ge 0$, we have
\begin{equation} \label{80}
\E(u_t) = t^\sigma \left(1 - \frac{\mu t^{\gamma - \sigma}}{s} \int_{\R^N} |u|^s\, dx - \frac{t^{\gamma^\ast - \sigma}}{2^\ast} \int_{\R^N} |u|^{2^\ast}\halfthin dx\right),
\end{equation}
where $\gamma$ and $\gamma^\ast$ are given in \eqref{133}. By \eqref{55}, \eqref{35}, and \eqref{139}, $\gamma - \sigma = \alpha\, (s - \beta) > 0$ and $\gamma^\ast - \gamma = \alpha\, (2^\ast - s) > 0$, so
\[
\sigma < \gamma < \gamma^\ast.
\]
Since $\M$ is bounded, it follows that
\[
\inf_{u \in \M}\, \E(u_\rho) > 0
\]
if $\rho$ is sufficiently small. Since $A_0$ is compact and $0 \notin A_0$, \eqref{80} also gives
\[
\E(u_t) \le t^\sigma \left(1 - \kappa_1\halfthin \mu t^{\gamma - \sigma} - \kappa_2\halfthin t^{\gamma^\ast - \sigma}\right) \quad \forall u \in A_0,\, t \ge 0
\]
for some constants $\kappa_1, \kappa_2 > 0$. So
\[
\sup_{u \in A}\, \E(u) \le R^\sigma \left(1 - \kappa_2\halfthin R^{\gamma^\ast - \sigma}\right) \le 0
\]
if $R$ is sufficiently large and
\[
\sup_{u \in X}\, \E(u) \le \max_{t \ge 0}\, t^\sigma \left(1 - \kappa_1\halfthin \mu t^{\gamma - \sigma}\right) = C \mu^{- \sigma/(\gamma - \sigma)} < \frac{1}{N}\, S^{N/2}
\]
if $\mu$ is sufficiently large. The desired conclusion now follows from Corollary \ref{Corollary 1}.
\end{proof}

\subsection{Br\'{e}zis-Nirenberg result for the ball-mass operator} \label{ssec:bn}

In this subsection we revisit the critical equation
\begin{equation} \label{bn-eq}
- \Delta u + \Phi_u(|x|)\, |u|^{p-2}\, u = \lambda\, |u|^{\beta - 2}\, u + |u|^{2^\ast - 2}\, u \quad \text{in } \R^N,
\end{equation}
where $N \ge 3$,
\begin{gather}
\label{300} 1 < p < 2^\ast,\\[10pt]
\label{301} 1 < q < \infty, \quad 2 < pq < 2^\ast,\\[10pt]
\label{302} 1 < b < b_\ast := 1 + Nq\, (1 - p/2^\ast),
\end{gather}
and we assume that $\beta > 2^{p,q}_{b,\ast}$ (see \eqref{31}). Let $\lambda_k \nearrow \infty$ be the sequence of positive eigenvalues of the scaled eigenvalue problem \eqref{51} given in Theorem \ref{thm:bm-eigenvalues}. We will obtain a nontrivial weak radial solution of equation \eqref{bn-eq} for all $\lambda > 0$ that does not belong to this sequence of eigenvalues. We have the following analog of the classical results of Br\'{e}zis and Nirenberg \cite{MR709644} and Capozzi et al.\! \cite{MR831041} for the ball-mass operator.

\begin{theorem} \label{thm:bn}
Assume \eqref{300}--\eqref{302} and
\begin{equation} \label{304}
\beta > \max \set{2^{p,q}_{b,\ast},2^\ast - 2}.
\end{equation}
Then equation \eqref{bn-eq} has a nontrivial weak radial solution $u$ with $0 < \E(u) < \frac{1}{N}\, S^{N/2}$ in the following cases:
\begin{enumroman}
\item \label{thm:bn.i} $0 < \lambda < \lambda_1$;
\item \label{thm:bn.ii} $\lambda_k < \lambda < \lambda_{k+1}$ for some $k \ge 1$.
\end{enumroman}
\end{theorem}

For $N \ge 4$,
\[
2^\ast - 2 = \frac{4}{N - 2} \le 2 < \beta
\]
by \eqref{303} and hence \eqref{304} reduces to $\beta > 2^{p,q}_{b,\ast}$. For $N = 3$, it also requires $\beta > 4$. We also note that the condition $q > 1$ in \eqref{301} is not an additional restriction. Indeed, $\beta > 2^{p,q}_{b,\ast}$ implies $b < (N - 1)(q - 1) + 1$ by \eqref{505}, \eqref{2}, and \eqref{36}, which together with $b > 1$ forces $q > 1$.

We will prove Theorem \ref{thm:bn} using the mountain pass theorem of Ambrosetti and Rabinowitz \cite{MR0370183} in case \ref{thm:bn.i} and the scaled linking theorem, Theorem \ref{Theorem 8}, in case \ref{thm:bn.ii} (for the classical linking arguments in the critical case see \cite{MR829403,MR831041,MR867663,MR779872}). The variational functional associated with equation \eqref{bn-eq} is
\[
\E(u) = I_\sigma(u) - \lambda J_\sigma(u) - \frac{1}{2^\ast} \int_{\R^N} |u|^{2^\ast}\halfthin dx, \quad u \in E^{p,q}_{b,\rad}(\R^N)
\]
(see \eqref{101} and \eqref{305}). For $u \in E^{p,q}_{b,\rad}(\R^N) \setminus \set{0}$ and $t \ge 0$, we have
\begin{equation} \label{506}
\E(u_t) = t^\sigma\halfthin \big(I_\sigma(u) - \lambda J_\sigma(u)\big) - \frac{t^{\gamma^\ast}}{2^\ast} \int_{\R^N} |u|^{2^\ast}\halfthin dx,
\end{equation}
where $\gamma^\ast$ is given in \eqref{133}. So
\[
\max_{t \ge 0}\, \E(u_t) = \left(1 - \frac{\sigma}{\gamma^\ast}\right)\! \left(\frac{2^\ast \sigma}{\gamma^\ast}\right)^{\sigma/(\gamma^\ast - \sigma)}\! \big(I_\sigma(u) - \lambda J_\sigma(u)\big)_+^{\gamma^\ast/(\gamma^\ast - \sigma)} \left(\int_{\R^N} |u|^{2^\ast}\halfthin dx\right)^{- \sigma/(\gamma^\ast - \sigma)}.
\]
Since $2^\ast \sigma = 2\, (\alpha\halfthin 2^\ast - N) = 2 \gamma^\ast$ (see \eqref{55}), this simplifies to
\begin{equation} \label{309}
\max_{t \ge 0}\, \E(u_t) = \frac{1}{N}\, Q(u)^{N/2},
\end{equation}
where
\[
Q(u) = \frac{2\, \big(I_\sigma(u) - \lambda J_\sigma(u)\big)_+}{\left(\dint_{\R^N} |u|^{2^\ast}\halfthin dx\right)^{2/2^\ast}}.
\]
We note that
\begin{equation} \label{519}
Q(u_t) = Q(u) \quad \forall u \in E^{p,q}_{b,\rad}(\R^N) \setminus \set{0},\, t > 0.
\end{equation}

Next we turn to the asymptotics of the truncated Talenti bubbles built from the extremals of the Sobolev inequality found by Talenti \cite{MR0463908}. Let $\varphi \in C^\infty_{c,\rad}(\R^N)$ with $0 \le \varphi \le 1$, $\varphi \equiv 1$ on $B_{1/2}$, and $\supp \varphi \subset B_1$. For $0 < \eps \le \ell \le 1$, set
\[
\varphi_{\eps,\ell}(x) = \varphi\!\left(\frac{x}{\ell}\right) \frac{\eps^{(N-2)/2}}{\big(\eps^2 + |x|^2\big)^{(N-2)/2}}.
\]
Then $\varphi_{\eps,\ell} \in C^\infty_{c,\rad}(\R^N)$ is supported in $B_\ell$ and we have the classical estimates
\begin{equation} \label{507}
\int_{\R^N} |\nabla \varphi_{\eps,\ell}|^2\, dx = S^{N/2} + \O\!\left((\eps/\ell)^{N-2}\right), \qquad \int_{\R^N} \varphi_{\eps,\ell}^{2^\ast}\halfthin dx = S^{N/2} + \O\!\left((\eps/\ell)^N\right)
\end{equation}
and
\begin{equation} \label{501}
\int_{\R^N} \varphi_{\eps,\ell}^s\, dx \asymp \begin{cases}
\eps^{s\, (N-2)/2}\, \ell^{N - s\, (N-2)} & \text{if } 1 \le s < \frac{N}{N - 2}\\[5pt]
\eps^{N/2} \abs{\log (\eps/\ell)} & \text{if } s = \frac{N}{N - 2}\\[5pt]
\eps^{\delta_s} & \text{if } s > \frac{N}{N - 2},
\end{cases}
\end{equation}
where
\begin{equation} \label{306}
\delta_s = N - \frac{s\, (N - 2)}{2}
\end{equation}
(see \cite{MR709644}).

\begin{lemma} \label{lem:bubble}
There exists a constant $C > 0$, independent of $\eps$ and $\ell$, such that
\begin{gather}
\label{313} \int_{\R^N} \varphi_{\eps,\ell}^\beta\, dx \ge \frac{1}{C}\, \eps^{\delta_\beta},\\[10pt]
\label{314} \int_0^\infty \rho^{-b}\, h_{\varphi_{\eps,\ell}}^q(\rho)\, d\rho \le C \eps^\nu\halfthin \ell^{b_\ast - b - \nu}\halfthin (\abs{\log (\eps/\ell)} + 1), \hquad \nu = \min \set{\frac{pq\, (N - 2)}{2},b_\ast - b}
\end{gather}
for all $0 < \eps \le \ell \le 1$. Moreover, $\delta_\beta < \min \set{N - 2,\nu}$.
\end{lemma}

\begin{proof}
Since $\varphi_{\eps,\ell}(x) = \ell^{-(N-2)/2}\, \varphi_{\eps/\ell,1}(x/\ell)$, $h_{\varphi_{\eps,\ell}}(\rho) = \ell^{\delta_p}\, h_{\varphi_{\eps/\ell,1}}(\rho/\ell)$, and $\delta_p\halfthin q + 1 = b_\ast$,
\[
\int_{\R^N} \varphi_{\eps,\ell}^\beta\, dx = \ell^{\delta_\beta} \int_{\R^N}
\varphi_{\eps/\ell,1}^\beta\, dx, \qquad \int_0^\infty \rho^{-b}\, h_{\varphi_{\eps,\ell}}^q(\rho)\, d\rho = \ell^{\,b_\ast -
b} \int_0^\infty \rho^{-b}\, h_{\varphi_{\eps/\ell,1}}^q(\rho)\, d\rho,
\]
so we may assume that $\ell = 1$.

We have
\begin{equation} \label{504}
\delta_\beta = N - \frac{\beta\, (N - 2)}{2} = 2 - \frac{N - 2}{\alpha}
\end{equation}
(see \eqref{31}), so $0 < \delta_\beta < 2$ by \eqref{35}. On the other hand, since $\beta > 4/(N - 2)$ by \eqref{304}, $\delta_\beta < N - 2$. So
\begin{equation} \label{316}
\delta_\beta < \min \set{2,N - 2} \le \frac{N}{2}
\end{equation}
and hence
\[
\beta = \frac{2\, (N - \delta_\beta)}{N - 2} > \frac{N}{N - 2}.
\]
The estimate \eqref{313} now follows from \eqref{501}.

For $0 < \rho < 1$,
\[
h_{\varphi_{\eps,1}}(\rho) = \int_{B_\rho} \varphi_{\eps,1}^p\, dy \le \int_{B_\rho} \frac{\eps^{p\,(N-2)/2}}{\big(\eps^2 + |y|^2\big)^{p\,(N-2)/2}}\, dy = \eps^{\delta_p}\, H\!\left(\frac{\rho}{\eps}\right),
\]
where
\[
H(\varrho) = \int_{B_\varrho} \frac{dy}{\big(1 + |y|^2\big)^{p\,(N-2)/2}}.
\]
So
\begin{equation} \label{503}
\int_0^1 \rho^{-b}\, h_{\varphi_{\eps,1}}^q(\rho)\, d\rho \le \eps^{\delta_p\halfthin q} \int_0^1 \rho^{-b}\, H^q\!\left(\frac{\rho}{\eps}\right) d\rho = \eps^{b_\ast - b} \int_0^{1/\eps} \varrho^{-b}\, H^q(\varrho)\, d\varrho.
\end{equation}
Using $H(\varrho) \le \omega_N\halfthin \varrho^N$ for $0 < \varrho \le 1$ and
\[
H(\varrho) \le \begin{cases}
C \varrho^{N - p\,(N-2)} & \text{if } p < \frac{N}{N - 2}\\[5pt]
C\, (1 + \abs{\log \varrho}) & \text{if } p \ge \frac{N}{N - 2}
\end{cases}
\]
for $\varrho > 1$ gives
\begin{align*}
\int_0^{1/\eps} \varrho^{-b}\, H^q(\varrho)\, d\varrho & \le \omega_N^q \int_0^1 \varrho^{Nq-b}\, d\varrho + \begin{cases}
C \int_1^{1/\eps} \varrho^{[N - p\,(N-2)]\, q - b}\, d\varrho & \text{if } p < \frac{N}{N - 2}\\[5pt]
C \int_1^{1/\eps} \varrho^{-b}\, (1 + \abs{\log \varrho})^q\, d\varrho & \text{if } p \ge \frac{N}{N - 2}
\end{cases}\\[7.5pt]
& \le C \left(1 + \int_1^{1/\eps} \varrho^{[N - p\,(N-2)]\, q - b}\, d\varrho\right) \hquad \text{since $1 < b < Nq + 1$ by \eqref{302}}\\[7.5pt]
& \le \begin{cases}
C\, (\eps^{-([N - p\,(N-2)]\, q + 1 - b)} + 1) & \text{if } b < [N - p\,(N-2)]\, q + 1\\[5pt]
C\, (\abs{\log \eps} + 1) & \text{if } b \ge [N - p\,(N-2)]\, q + 1.
\end{cases}
\end{align*}
Combining \eqref{503} with this estimate and noting that
\[
\delta_p\halfthin q - [N - p\, (N - 2)]\, q = \frac{pq\, (N - 2)}{2}
\]
gives
\begin{equation} \label{500}
\int_0^1 \rho^{-b}\, h_{\varphi_{\eps,1}}^q(\rho)\, d\rho \le \begin{cases}
C\, (\eps^{pq\,(N-2)/2} + \eps^{b_\ast - b}) & \text{if } b < [N - p\,(N-2)]\, q + 1\\[5pt]
C \eps^{b_\ast - b}\, (\abs{\log \eps} + 1) & \text{if } b \ge [N - p\,(N-2)]\, q + 1.
\end{cases}
\end{equation}
On the other hand, since
\[
h_{\varphi_{\eps,1}}(\rho) = \int_{\R^N} \varphi_{\eps,1}^p\, dy
\]
for $\rho \ge 1$ and $b > 1$,
\begin{equation} \label{502}
\int_1^\infty \rho^{-b}\, h_{\varphi_{\eps,1}}^q(\rho)\, d\rho \le \begin{cases}
C \eps^{pq\, (N-2)/2} & \text{if } p < \frac{N}{N - 2}\\[5pt]
C \eps^{\delta_p\halfthin q}\, (\abs{\log \eps} + 1)^q & \text{if } p \ge \frac{N}{N - 2}
\end{cases}
\end{equation}
by \eqref{501}. Since $\delta_p\halfthin q = b_\ast - 1 > b_\ast - b$, \eqref{314} follows from \eqref{500} and \eqref{502}.

Since $\delta_\beta < N - 2$ (see \eqref{316}) and $pq > 2$ by \eqref{301},
\[
\delta_\beta < \frac{pq\, (N - 2)}{2},
\]
so it only remains to show that $\delta_\beta < b_\ast - b$. By \eqref{504} and \eqref{55},
\[
\delta_\beta = \frac{2\, (b_\ast - b)}{\alpha\, (pq - 2)},
\]
and $b_\ast - b > 0$ by \eqref{302}, so it suffices to show that $\alpha\, (pq - 2) > 2$. By \eqref{37},
\[
\alpha\, (pq - 2) = N\halfthin (q - 1) + 3 - b.
\]
Since $\beta > 2^{p,q}_{b,\ast}$ by \eqref{304}, $b < (N - 1)(q - 1) + 1$ by \eqref{505}, \eqref{2}, and \eqref{36}, so
\[
N\halfthin (q - 1) + 3 - b > q + 1 > 2
\]
by \eqref{301}.
\end{proof}

It follows from \eqref{507} and Lemma \ref{lem:bubble} that there exists a constant $c_0 > 0$ such that
\begin{equation} \label{508}
Q(\varphi_{\eps,\ell}) = \frac{\left(\ds{\int_{\R^N} |\nabla \varphi_{\eps,\ell}|^2\, dx + \frac{2a}{pq} \int_0^\infty \rho^{-b}\, h_{\varphi_{\eps,\ell}}^q(\rho)\, d\rho - \frac{2 \lambda}{\beta} \int_{\R^N} \varphi_{\eps,\ell}^\beta\, dx}\right)_+}{\left(\dint_{\R^N} \varphi_{\eps,\ell}^{2^\ast}\halfthin dx\right)^{2/2^\ast}} \le S - c_0\, \lambda\, \eps^{\delta_\beta}
\end{equation}
for all sufficiently small $\eps > 0$, depending on $\ell$ and $\lambda$. We are now ready to prove case \ref{thm:bn.i} of Theorem \ref{thm:bn}.

\begin{proof}[Proof of Theorem \ref{thm:bn} \ref{thm:bn.i}]
We will show that the functional $\E$ has the mountain pass geometry with the mountain pass level below the compactness threshold $\frac{1}{N}\, S^{N/2}$ (see Proposition \ref{prop:local-ps}).

As in the proof of Theorem \ref{thm:left-nbhd},
\[
\E(u_t) \ge t^\sigma \left(1 - \frac{\lambda}{\lambda_1} - C t^{\gamma^\ast - \sigma}\right) \quad \forall u \in \M,\, t \ge 0
\]
for some constant $C > 0$. Since $\lambda < \lambda_1$ and $\gamma^\ast > \sigma$, it follows that if $\rho > 0$ is sufficiently small, then
\[
\inf_{u \in \bdry{U_\rho}}\, \E(u) > 0,
\]
where $U_\rho = \left\{u \in E^{p,q}_{b,\rad}(\R^N) : I_\sigma(u) < \rho^\sigma\right\}$ is a neighborhood of the origin since $I_\sigma$ is continuous with $I_\sigma(0) = 0$. On the other hand, for any $u \in E^{p,q}_{b,\rad}(\R^N) \setminus \set{0}$, $\E(u_t) \to - \infty$ as $t \to \infty$ by \eqref{506}. So $\E$ has the mountain pass geometry.

Let
\[
c := \inf_{\gamma \in \Gamma}\, \max_{t \in [0,1]}\, \E(\gamma(t)),
\]
where
\[
\Gamma = \set{\gamma \in C([0,1],E^{p,q}_{b,\rad}(\R^N)) : \gamma(0) = 0,\, \E(\gamma(1)) < 0},
\]
be the mountain pass level. By \eqref{309} and \eqref{508},
\[
c \le \max_{t \ge 0}\, \E((\varphi_{\eps,1})_t) = \frac{1}{N}\, Q(\varphi_{\eps,1})^{N/2} < \frac{1}{N}\, S^{N/2}
\]
for all sufficiently small $\eps > 0$.
\end{proof}

Now suppose $\lambda_k < \lambda < \lambda_{k+1}$ for some $k \ge 1$ and fix $\lambda' \in (\lambda_k,\lambda)$. We will apply the scaled linking theorem, Theorem \ref{Theorem 8}, with $e = \pi(\varphi_{\eps,\ell})$ and the homotopy
\begin{equation} \label{584}
H(u,\tau) = \pi(u_{1 - \tau} + (\varphi_{\eps,\ell})_\tau), \quad u \in A_0,\, \tau \in [0,1],
\end{equation}
which interpolates between a suitable subset $A_0$ of the sublevel set $\widetilde{\Psi}^{\lambda'}$ and the concentrating bubble $e$ along the fibers of the scaling (see \eqref{102}).

Set
\[
\mu = \frac{\lambda}{\lambda'} - 1 > 0.
\]
For $u \in \widetilde{\Psi}^{\lambda'}$,
\begin{equation} \label{517}
I_\sigma(u) - \lambda J_\sigma(u) = 1 - \frac{\lambda}{\widetilde{\Psi}(u)} \le 1 - \frac{\lambda}{\lambda'} = - \mu < 0,
\end{equation}
so it follows from \eqref{104} that
\[
Q(u_t) = 0 \quad \forall u \in \widetilde{\Psi}^{\lambda'},\, t > 0.
\]

Next we show that for all sufficiently large $t > 0$, $Q(u_t + \varphi_{\eps,\ell})$ vanishes identically on $\widetilde{\Psi}^{\lambda'}$.

\begin{lemma} \label{lem:T}
There exists $T \ge 1$, depending only on $\lambda$ and $\lambda'$, such that
\[
Q(u_t + \varphi_{\eps,\ell}) = 0 \quad \forall u \in \widetilde{\Psi}^{\lambda'},\, t \ge T,\, 0 < \eps \le \ell \le 1.
\]
\end{lemma}

\begin{proof}
For $t > 0$, $u_t + \varphi_{\eps,\ell} = (u + (\varphi_{\eps,\ell})_{1/t})_t$ and hence $Q(u_t + \varphi_{\eps,\ell}) = Q(u + (\varphi_{\eps,\ell})_{1/t})$ by \eqref{519}, so it suffices to show that there exists $0 < \tau_0 \le 1$, depending only on $\lambda$ and $\lambda'$, such that
\[
Q(u + (\varphi_{\eps,\ell})_\tau) = 0 \quad \forall u \in \widetilde{\Psi}^{\lambda'},\, \tau \le \tau_0,\, 0 < \eps \le \ell \le 1.
\]
Since $\norm{\varphi_{\eps,\ell}}$ is bounded for $0 < \eps \le \ell \le 1$ by \eqref{507} and Lemma \ref{lem:bubble}, the set $\bgset{u + (\varphi_{\eps,\ell})_\tau : u \in \widetilde{\Psi}^{\lambda'},\, \tau \in [0,1],\, 0 < \eps \le \ell \le 1}$ is bounded. Since $I_\sigma'$ and $J_\sigma'$ are bounded on bounded sets, $I_\sigma - \lambda J_\sigma$ is Lipschitz continuous on this set, say, with Lipschitz constant $L > 0$. So
\[
I_\sigma(u + (\varphi_{\eps,\ell})_\tau) - \lambda J_\sigma(u + (\varphi_{\eps,\ell})_\tau) \le I_\sigma(u) - \lambda J_\sigma(u) + L \norm{(\varphi_{\eps,\ell})_\tau} \le - \mu + L \tau^{\sigma/pq} \norm{\varphi_{\eps,\ell}}
\]
for all $u \in \widetilde{\Psi}^{\lambda'}$, $\tau \in [0,1]$, and $0 < \eps \le \ell \le 1$ by \eqref{517} and \eqref{38}. The desired conclusion follows from this since $\norm{\varphi_{\eps,\ell}}$ is bounded.
\end{proof}

We will construct a suitable set $A_0 \subset \widetilde{\Psi}^{\lambda'}$ with elements supported away from the origin. We begin by proving the following truncation lemma.

\begin{lemma} \label{lem:cut}
Let $\vartheta : [0,\infty) \to [0,1]$ be a smooth function such that $\vartheta \equiv 0$ on $[0,1/2]$, $\vartheta \equiv 1$ on $[1,\infty)$, and $|\vartheta'| \le 3$. For $v \in E^{p,q}_{b,\rad}(\R^N)$ and $m > 0$, set $\vartheta_m(x) = \vartheta(|x|/m)$ and $v_m = \vartheta_m\, v$. Then
\begin{enumroman}
\item \label{lem:cut.i} $v \mapsto v_m$ is a bounded linear mapping on $E^{p,q}_{b,\rad}(\R^N)$;
\item \label{lem:cut.ii} $v_m \to v$ in $E^{p,q}_{b,\rad}(\R^N)$ as $m \to 0$, uniformly on compact sets.
\end{enumroman}
\end{lemma}

\begin{proof}
\ref{lem:cut.i} Since $|v_m| \le |v|$, $h_{v_m} \le h_v$ and hence $\norm[L^{p,q}_b(\R^N)]{v_m} \le \norm[L^{p,q}_b(\R^N)]{v}$. Since $|\nabla v_m| \le \vartheta_m\, |\nabla v| + |\nabla \vartheta_m|\, |v| \le |\nabla v| + 3\, \mathbf{1}_{B_m}\, |v|/m$,
\[
\int_{\R^N} |\nabla v_m|^2\, dx \le 2 \left(\int_{\R^N} |\nabla v|^2\, dx + \frac{9}{m^2} \int_{B_m} v^2\, dx\right).
\]
By the H\"{o}lder and Sobolev inequalities,
\begin{equation} \label{511}
\int_{B_m} v^2\, dx \le \left(\omega_N\halfthin m^N\right)^{2/N}\! \left(\int_{B_m} |v|^{2^\ast}\halfthin dx\right)^{2/2^\ast} \le C\halfthin m^2 \int_{\R^N} |\nabla v|^2\, dx
\end{equation}
for some constant $C > 0$ independent of $m$. So
\[
\int_{\R^N} |\nabla v_m|^2\, dx \le C \int_{\R^N} |\nabla v|^2\, dx.
\]

\ref{lem:cut.ii} Since $\vartheta_m \equiv 1$ on $\R^N \setminus B_m$, $|v_m - v| = (1 - \vartheta_m)\, |v| \le \mathbf{1}_{B_m}\, |v|$ and hence
\[
h_{v_m - v}(\rho) = \int_{B_\rho} |v_m - v|^p\, dy \le \begin{cases}
h_v(\rho) & \text{if } 0 < \rho < m\\[5pt]
h_v(m) & \text{if } \rho \ge m.
\end{cases}
\]
By the H\"{o}lder and Sobolev inequalities,
\[
h_v(m) = \int_{B_m} |v|^p\, dy \le \left(\omega_N\halfthin m^N\right)^{1 - p/2^\ast}\! \left(\int_{B_m} |v|^{2^\ast}\halfthin dy\right)^{p/2^\ast} \le C\halfthin m^{N\, (1 - p/2^\ast)} \norm{v}^p.
\]
So
\begin{multline*}
\norm[L^{p,q}_b(\R^N)]{v_m - v}^{pq} = \int_0^\infty \rho^{-b}\, h_{v_m - v}^q(\rho)\, d\rho \le \int_0^m \rho^{-b}\, h_v^q(\rho)\, d\rho + \int_m^\infty \rho^{-b}\, h_v^q(m)\, d\rho\\[7.5pt]
\le \norm[X]{\mathbf{1}_{(0,m)}\, \Lambda v}^{pq} + C\halfthin m^{b_\ast - b} \norm{v}^{pq},
\end{multline*}
where $\Lambda$ is the isometry in \eqref{510}. The last expression is continuous in $v$ on $E^{p,q}_{b,\rad}(\R^N)$, nondecreasing in $m$, and converges to zero as $m \to 0$, so the convergence is uniform on compact sets by Dini's theorem. Since $|\nabla(v_m - v)| \le (1 - \vartheta_m)\, |\nabla v| + |\nabla \vartheta_m|\, |v| \le \mathbf{1}_{B_m}\, (|\nabla v| + 3\, |v|/m)$,
\begin{multline*}
\int_{\R^N} |\nabla(v_m - v)|^2\, dx \le 2 \left(\int_{B_m} |\nabla v|^2\, dx + \frac{9}{m^2} \int_{B_m} v^2\, dx\right)\\[7.5pt]
\le C \left[\int_{B_m} |\nabla v|^2\, dx + \left(\int_{B_m} |v|^{2^\ast}\halfthin dx\right)^{2/2^\ast}\right]
\end{multline*}
by \eqref{511} and Dini's theorem applies again.
\end{proof}

Next we prove the following lemma.

\begin{lemma} \label{lem:A0}
For all sufficiently small $m > 0$, $\widetilde{\Psi}^{\lambda'}$ has a compact symmetric subset $A_0$ of index $k$ such that
\begin{equation} \label{513}
\supp u \subset \R^N \setminus B_{m/4} \quad \forall u \in A_0.
\end{equation}
\end{lemma}

\begin{proof}
By Theorem \ref{Theorem 13} \ref{Theorem 13.iii}, $i(\M \setminus \widetilde{\Psi}_{\lambda'}) = k$. Since $\M \setminus \widetilde{\Psi}_{\lambda'}$ is an open symmetric subset of $\M$, then it has a compact symmetric subset $C$ of index $k$ (see the proof of Proposition 3.1 in \cite{MR2371112}). Let $v \mapsto v_m$ be the continuous mapping in Lemma \ref{lem:cut}. Since $C$ is compact, $v_m \to v$ as $m \to 0$, uniformly on $C$. Since the origin is not in the closed set $C$, it follows that for all sufficiently small $m > 0$, $v_m \ne 0$ for all $v \in C$. For such $m$, let
\[
A_0 = \set{\pi(v_m) : v \in C}.
\]
Since $C \to A_0,\, v \mapsto \pi(v_m)$ is an odd continuous map, $A_0$ is a compact symmetric subset of $\M$ and
\begin{equation} \label{514}
i(A_0) \ge i(C) = k
\end{equation}
by Proposition \ref{Proposition 7} $(i_2)$.

We have
\[
\widetilde{\Psi}(\pi(v_m)) = \frac{1}{J_\sigma((v_m)_{I_\sigma(v_m)^{- 1/\sigma}})} = \frac{I_\sigma(v_m)}{J_\sigma(v_m)}
\]
by \eqref{104}. On bounded sets, $I_\sigma'$ and $J_\sigma'$ are bounded and hence $I_\sigma$ and $J_\sigma$ are Lipschitz continuous, so
\[
I_\sigma(v_m) \to I_\sigma(v) = 1, \qquad J_\sigma(v_m) \to J_\sigma(v) = \frac{1}{\widetilde{\Psi}(v)}
\]
uniformly on $C$. Since $1/\widetilde{\Psi} > 1/\lambda' > 0$ on $C$, it follows that $\widetilde{\Psi}(\pi(v_m)) \to \widetilde{\Psi}(v)$ uniformly on $C$. Since $\widetilde{\Psi}(v) < \lambda'$ for all $v$ in the compact set $C$, then for all sufficiently small $m > 0$, $\widetilde{\Psi}(\pi(v_m)) \le \lambda'$ for all $v \in C$, i.e., $A_0 \subset \widetilde{\Psi}^{\lambda'}$. Then
\[
i(A_0) \le i(\widetilde{\Psi}^{\lambda'}) = k
\]
by Proposition \ref{Proposition 7} $(i_2)$ and Theorem \ref{Theorem 13} \ref{Theorem 13.iii}, which together with \eqref{514} gives $i(A_0) = k$.

Finally, \eqref{513} follows since $\supp v_m \subset \R^N \setminus B_{m/2}$,
\[
\pi(v_m)(x) = (v_m)_{I_\sigma(v_m)^{- 1/\sigma}}(x) = I_\sigma(v_m)^{- \alpha/\sigma}\, v_m(I_\sigma(v_m)^{- 1/\sigma}\halfthin x),
\]
and $I_\sigma(v_m) \to 1$ uniformly on $C$.
\end{proof}

Now we let $T \ge 1$ be as in Lemma \ref{lem:T}, $\ell > 0$ be so small that $m = 4T \ell$ is admissible in Lemma \ref{lem:A0}, and take $A_0$ to be the corresponding compact symmetric subset of $\widetilde{\Psi}^{\lambda'}$ of index $k$ such that
\begin{equation} \label{533}
\supp u \subset \R^N \setminus B_{T \ell} \quad \forall u \in A_0.
\end{equation}
We have the following interaction estimates.

\begin{lemma} \label{lem:interaction}
For all $u \in A_0$ and $0 \le t < T$,
\begin{gather}
\label{541} \int_{\R^N} |\nabla(u_t + \varphi_{\eps,\ell})|^2\, dx = \int_{\R^N} |\nabla u_t|^2\, dx + \int_{\R^N} |\nabla \varphi_{\eps,\ell}|^2\, dx,\\[10pt]
\label{548} \int_{\R^N} |u_t + \varphi_{\eps,\ell}|^\beta\, dx = \int_{\R^N} |u_t|^\beta\, dx + \int_{\R^N} \varphi_{\eps,\ell}^\beta\, dx,\\[10pt]
\label{542} \int_{\R^N} |u_t + \varphi_{\eps,\ell}|^{2^\ast}\halfthin dx = \int_{\R^N} |u_t|^{2^\ast}\halfthin dx + \int_{\R^N} \varphi_{\eps,\ell}^{2^\ast}\halfthin dx.
\end{gather}
For all $u \in A_0$ and all $t \ge 0$,
\begin{multline} \label{524}
\int_0^\infty \rho^{-b}\, h_{u_t + \varphi_{\eps,\ell}}^q(\rho)\, d\rho \le \int_0^\infty \rho^{-b}\, h_{u_t}^q(\rho)\, d\rho + \int_0^\infty \rho^{-b}\, h_{\varphi_{\eps,\ell}}^q(\rho)\, d\rho\\[7.5pt]
+ C \left(t^{\sigma\,(pq-1)/pq} \norm[L^{p,q}_b(\R^N)]{\varphi_{\eps,\ell}} + t^{\sigma/pq} \norm[L^{p,q}_b(\R^N)]{\varphi_{\eps,\ell}}^{pq-1}\right)
\end{multline}
for some constant $C > 0$ independent of $\eps$ and $\ell$.
\end{lemma}

\begin{proof}
For $t = 0$, $u_t = 0$ and \eqref{541}--\eqref{524} are trivial. For $u \in A_0$ and $0 < t < T$,
\[
\supp u_t \subset \R^N \setminus B_{T \ell/t} \subset \R^N \setminus B_\ell
\]
by \eqref{533}, while $\supp \varphi_{\eps,\ell} \subset B_\ell$. So $u_t$ and $\varphi_{\eps,\ell}$ have disjoint supports and hence \eqref{541}--\eqref{542} hold. For $u \in A_0$ and $t \ge 0$,
\[
\int_0^\infty \rho^{-b}\, h_{u_t + \varphi_{\eps,\ell}}^q(\rho)\, d\rho = \norm[L^{p,q}_b(\R^N)]{u_t + \varphi_{\eps,\ell}}^{pq} \le \big(\norm[L^{p,q}_b(\R^N)]{u_t} + \norm[L^{p,q}_b(\R^N)]{\varphi_{\eps,\ell}}\big)^{pq}.
\]
Combining this with the elementary inequality
\[
(a + b)^s \le a^s + b^s + s\, 2^{s-1} \left(ab^{s-1} + a^{s-1} b\right) \quad \forall a, b \ge 0,\, s > 1
\]
gives \eqref{524} since $\norm[L^{p,q}_b(\R^N)]{u_t} = t^{\sigma/pq} \norm[L^{p,q}_b(\R^N)]{u}$ as in \eqref{38} and $A_0$ is a subset of the bounded set $\M$.
\end{proof}

Next we estimate $Q(u_t + \varphi_{\eps,\ell})$ on $A_0 \times [0,T)$.

\begin{lemma} \label{lem:level}
There exists $\eps_0 \in (0,\ell]$, depending only on $\lambda$, $\lambda'$, and $\ell$, such that
\[
Q(u_t + \varphi_{\eps,\ell}) \le S - \frac{c_0}{2}\, \lambda\, \eps^{\delta_\beta} \quad \forall u \in A_0,\, 0 \le t < T,\, 0 < \eps < \eps_0.
\]
\end{lemma}

\begin{proof}
Let $u \in A_0$, let $0 \le t < T$, and set $w = u_t + \varphi_{\eps,\ell}$. By Lemma \ref{lem:interaction},
\begin{multline} \label{543}
I_\sigma(w) - \lambda J_\sigma(w) \le I_\sigma(u_t) - \lambda J_\sigma(u_t) + I_\sigma(\varphi_{\eps,\ell}) - \lambda J_\sigma(\varphi_{\eps,\ell})\\[7.5pt]
+ C \left(t^{\sigma\,(pq-1)/pq}\, \eta_{\eps,\ell} + t^{\sigma/pq}\, \eta_{\eps,\ell}^{\,pq-1}\right),
\end{multline}
where $\eta_{\eps,\ell} = \norm[L^{p,q}_b(\R^N)]{\varphi_{\eps,\ell}}$ and $C > 0$ is a constant independent of $\eps$ and $\ell$. By \eqref{104} and \eqref{517},
\begin{equation} \label{544}
I_\sigma(u_t) - \lambda J_\sigma(u_t) = t^\sigma\halfthin \big(I_\sigma(u) - \lambda J_\sigma(u)\big) \le - \mu\halfthin t^\sigma
\end{equation}
since $A_0 \subset \widetilde{\Psi}^{\lambda'}$. On the other hand, by Young's inequality,
\begin{equation} \label{545}
C \left(t^{\sigma\,(pq-1)/pq}\, \eta_{\eps,\ell} + t^{\sigma/pq}\, \eta_{\eps,\ell}^{\,pq-1}\right) \le \mu\halfthin t^\sigma + \frac{C'}{2}\, \eta_{\eps,\ell}^{\,pq}
\end{equation}
with $C' > 0$ depending on $\mu$. Combining \eqref{543}--\eqref{545} and taking positive parts gives
\[
2\, \big(I_\sigma(w) - \lambda J_\sigma(w)\big)_+ \le 2\, \big(I_\sigma(\varphi_{\eps,\ell}) - \lambda J_\sigma(\varphi_{\eps,\ell})\big)_+ + C' \eta_{\eps,\ell}^{\,pq}.
\]
Since
\[
\int_{\R^N} |w|^{2^\ast}\halfthin dx \ge \int_{\R^N} \varphi_{\eps,\ell}^{2^\ast}\halfthin dx
\]
by \eqref{542}, it follows that
\[
Q(w) \le Q(\varphi_{\eps,\ell}) + \frac{C' \eta_{\eps,\ell}^{\,pq}}{\left(\dint_{\R^N} \varphi_{\eps,\ell}^{2^\ast}\halfthin dx\right)^{2/2^\ast}}.
\]
Combining this with \eqref{508}, \eqref{314}, and \eqref{507} gives
\[
Q(w) \le S - c_0\, \lambda\, \eps^{\delta_\beta} + C'' \eps^\nu\halfthin (\abs{\log \eps} + 1)
\]
for all sufficiently small $\eps > 0$, with $C'' > 0$ depending on $\ell$. Since $\delta_\beta < \nu$ by Lemma \ref{lem:bubble}, the desired conclusion follows.
\end{proof}

Note that the scaling-based homotopy $H$ in \eqref{584} is well-defined, continuous, and satisfies the hypotheses of Theorem \ref{Theorem 8}. For $u \in A_0$ and $\tau \in [0,1]$, set $w_{u,\tau} = u_{1 - \tau} + (\varphi_{\eps,\ell})_\tau$. Then $w_{u,0} = u \ne 0$, while for $\tau > 0$, $u_{1 - \tau}$ vanishes in a neighborhood of the origin by \eqref{533} and hence
\[
w_{u,\tau}(0) = \tau^\alpha\, \varphi_{\eps,\ell}(0) = \tau^\alpha\, \eps^{-(N-2)/2} > 0.
\]
So $w_{u,\tau} \ne 0$ and hence $H$ is well-defined. Since the scaling \eqref{32} is continuous by Proposition \ref{prop:continuous} and $\pi$ is continuous on $E^{p,q}_{b,\rad}(\R^N) \setminus \set{0}$, $H \in C(A_0 \times [0,1],\M)$. For all $u \in A_0$, $H(u,0) = \pi(u) = u$ and $H(u,1) = \pi(\varphi_{\eps,\ell}) = e$. Now we estimate $\E$ on $H$.

\begin{lemma} \label{lem:sup}
For all sufficiently small $\eps > 0$,
\begin{equation} \label{546}
\sup \set{\E(H(u,\tau)_t) : u \in A_0,\, \tau \in [0,1],\, t \ge 0} < \frac{1}{N}\, S^{N/2}.
\end{equation}
\end{lemma}

\begin{proof}
Let $u \in A_0$ and $\tau \in [0,1]$. Since
\[
H(u,\tau) = \pi(w_{u,\tau}) = (w_{u,\tau})_s, \quad s = I_\sigma(w_{u,\tau})^{- 1/\sigma},
\]
the fibers through $H(u,\tau)$ and $w_{u,\tau}$ coincide, so
\[
\sup_{t \ge 0}\, \E(H(u,\tau)_t) = \max_{t \ge 0}\, \E((w_{u,\tau})_t) = \frac{1}{N}\, Q(w_{u,\tau})^{N/2}
\]
by \eqref{309}. By \eqref{517}, $Q(w_{u,0}) = Q(u) = 0$ since $A_0 \subset \widetilde{\Psi}^{\lambda'}$. For $\tau > 0$, $w_{u,\tau} = (u_{(1 - \tau)/\tau} + \varphi_{\eps,\ell})_\tau$ and hence
\[
Q(w_{u,\tau}) = Q(u_{(1 - \tau)/\tau} + \varphi_{\eps,\ell})
\]
by \eqref{519}. If $(1 - \tau)/\tau \ge T$, then $Q(u_{(1 - \tau)/\tau} + \varphi_{\eps,\ell}) = 0$ by Lemma \ref{lem:T}. If $(1 - \tau)/\tau < T$, then
\[
Q(u_{(1 - \tau)/\tau} + \varphi_{\eps,\ell}) \le S - \frac{c_0}{2}\, \lambda\, \eps^{\delta_\beta}
\]
for all sufficiently small $\eps > 0$ by Lemma \ref{lem:level}. So
\[
\sup \set{\E(H(u,\tau)_t) : u \in A_0,\, \tau \in [0,1],\, t \ge 0} \le \frac{1}{N} \left(S - \frac{c_0}{2}\, \lambda\, \eps^{\delta_\beta}\right)^{N/2} < \frac{1}{N}\, S^{N/2}. \qedhere
\]
\end{proof}

We are now ready to prove case \ref{thm:bn.ii} of Theorem \ref{thm:bn}.

\begin{proof}[Proof of Theorem \ref{thm:bn} \ref{thm:bn.ii}]
Fix $\lambda'' \in (\lambda,\lambda_{k+1})$ and let $B_0 = \widetilde{\Psi}_{\lambda''}$. Then $A_0 \cap B_0 = \emptyset$ since $\lambda' < \lambda''$ and
\[
i(A_0) = i(\M \setminus B_0) = k
\]
by Lemma \ref{lem:A0} and Theorem \ref{Theorem 13} \ref{Theorem 13.iii}.

Fix $\eps > 0$ so small that \eqref{546} holds. Let $e = \pi(\varphi_{\eps,\ell})$, let $R > \rho > 0$, and let $A$, $B$, and $X$ be as in Theorem \ref{Theorem 8}. Since $A_0 \subset \widetilde{\Psi}^{\lambda'}$,
\begin{equation} \label{581}
\E(u_t) \le 0 \quad \forall u \in A_0,\, t \ge 0
\end{equation}
by \eqref{104} and \eqref{517}. We have
\begin{equation} \label{580}
\E(u_t) = t^\sigma \left(1 - \frac{\lambda}{\widetilde{\Psi}(u)} - \frac{t^{\gamma^\ast - \sigma}}{2^\ast} \int_{\R^N} |u|^{2^\ast}\halfthin dx\right), \quad u \in \M,\, t \ge 0,
\end{equation}
where $\gamma^\ast$ is given in \eqref{133}. By \eqref{55}, \eqref{35}, and \eqref{303}, $\gamma^\ast - \sigma = \alpha\, (2^\ast - \beta) > 0$, so $\sigma < \gamma^\ast$. Since $\lambda > 0$, \eqref{580} gives
\begin{equation} \label{582}
\E(H(u,\tau)_R) \le R^\sigma \left(1 - \frac{R^{\gamma^\ast - \sigma}}{2^\ast} \norm[L^{2^\ast}(\R^N)]{H(u,\tau)}^{2^\ast}\right) \quad \forall u \in A_0,\, \tau \in [0,1].
\end{equation}
Since $A_0$ is compact, so is the set $\set{H(u,\tau) : u \in A_0,\, \tau \in [0,1]}$. Since the origin is not in this set, it follows that $\norm[L^{2^\ast}(\R^N)]{\cdot}$ is bounded from below by a positive constant there. So it follows from \eqref{581} and \eqref{582} that
\[
\sup_{u \in A}\, \E(u) \le 0
\]
if $R$ is sufficiently large. On the other hand, \eqref{580} gives
\[
\E(u_\rho) \ge \rho^\sigma \left(1 - \frac{\lambda}{\lambda''} - C \rho^{\gamma^\ast - \sigma}\right) \quad \forall u \in B_0
\]
for some constant $C > 0$ since $\M$ is bounded. Since $\lambda < \lambda''$, it follows that
\begin{equation} \label{547}
\inf_{u \in B}\, \E(u) > 0
\end{equation}
if $\rho$ is sufficiently small. By \eqref{546},
\begin{equation} \label{583}
\sup_{u \in X}\, \E(u) < \frac{1}{N}\, S^{N/2}.
\end{equation}
By \eqref{547} and \eqref{583}, the minimax level $c$ in Theorem \ref{Theorem 8} satisfies
\[
0 < c < \frac{1}{N}\, S^{N/2}.
\]
So $\E$ satisfies the \PS{c} condition by Proposition \ref{prop:local-ps} and hence $c$ is a critical value of $\E$.
\end{proof}

\begin{remark} \label{rmk:resonant}
Theorem \ref{thm:bn} does not cover the resonant case $\lambda = \lambda_k$. The obstruction is not in the linking geometry. The index computation $i(A_0) = i(\M \setminus B_0) = k$ and the estimates $\sup_{u \in A}\, \E(u) \le 0 < \inf_{u \in B}\, \E(u)$ remain valid at $\lambda = \lambda_k$ and Proposition \ref{prop:local-ps} is independent of $\lambda$. The difficulty lies in the construction of the set $A_0$ so that the linking minimax level $c < \frac{1}{N}\, S^{N/2}$. Whether equation \eqref{bn-eq} has a nontrivial weak radial solution when $\lambda = \lambda_k$ is an open problem.
\end{remark}

\begin{remark} \label{rmk:sign}
In case \ref{thm:bn.i} of Theorem \ref{thm:bn}, the solution may be taken to be nonnegative. Indeed, a straightforward modification of the above arguments shows that the functional
\[
\E_+(u) = I_\sigma(u) - \frac{\lambda}{\beta} \int_{\R^N} u_+^\beta\, dx - \frac{1}{2^\ast} \int_{\R^N} u_+^{2^\ast}\halfthin dx, \quad u \in E^{p,q}_{b,\rad}(\R^N),
\]
where $u_\pm = \max \set{\pm u,0}$, has a critical point $u$ with $0 < \E_+(u) < \frac{1}{N}\, S^{N/2}$ (note that $\varphi_{\eps,\ell} \ge 0$), and testing the corresponding equation with $- u_-$ gives
\[
\int_{\R^N} |\nabla u_-|^2\, dx + \int_{\R^N} \Phi_u(|x|)\, u_-^p\, dx = 0,
\]
so $u \ge 0$. If $p \ge 2$, then $- \Delta u + c(x)\, u \ge 0$ with $c(x) = \Phi_u(|x|)\, u^{p-2} \in L^\infty_\loc(\R^N \setminus \set{0})$ by Theorem \ref{thm:bm-regularity}, so the strong maximum principle gives $u > 0$ in $\R^N \setminus \set{0}$ (see, e.g., \cite[Theorem 9.6]{MR1814364}). If $1 < p < 2$, the absorption term $\Phi_u(|x|)\, u^{p-1}$ is sublinear near $u = 0$, the strong maximum principle of V\'{a}zquez \cite{MR768629} does not apply, and we do not know whether $u$ is positive. We note that this is the case for all the Schr\"{o}dinger-Poisson-Slater equations covered by Theorem \ref{thm:sps-bn} below since $p < N/(N - 2) \le 2$ when $N \ge 4$.
\end{remark}

\subsection{\hspace{-3.08pt}Applications to the defocusing inverse-power Schr\"{o}dinger equation} \label{ssec:app-ip}

When $q = 1$ and $a = b - 1$, equation \eqref{1} reduces to the defocusing inverse-power Schr\"{o}dinger equation
\begin{equation} \label{600}
- \Delta u + \frac{|u|^{p-2}\, u}{|x|^{b-1}} = f(u) \quad \text{in } \R^N,
\end{equation}
where $N \ge 2$ and
\begin{equation} \label{601}
2 < p < \begin{cases}
\infty & \text{if } N = 2\\[5pt]
2^\ast & \text{if } N \ge 3,
\end{cases} \qquad 1 < b < b_\ast := 1 + N - \frac{(N - 2)\, p}{2}.
\end{equation}
By Fubini's theorem,
\begin{equation} \label{602}
\int_0^\infty \rho^{-b}\, h_u(\rho)\, d\rho = \int_{\R^N} |u(y)|^p \int_{|y|}^\infty \rho^{-b}\, d\rho\, dy = \frac{1}{b - 1} \int_{\R^N} \frac{|u|^p}{|x|^{b-1}}\, dx,
\end{equation}
so in this case the ball-mass term is local, $E^{p,1}_b(\R^N)$ is the weighted Sobolev space of weakly differentiable functions with $\nabla u \in L^2(\R^N)$ and $|x|^{-(b-1)/p}\, u \in L^p(\R^N)$, and \eqref{101} gives
\[
I_\sigma(u) = \frac{1}{2} \int_{\R^N} |\nabla u|^2\, dx + \frac{1}{p} \int_{\R^N} \frac{|u|^p}{|x|^{b-1}}\, dx, \quad u \in E^{p,1}_{b,\rad}(\R^N).
\]
We have
\begin{equation} \label{603}
\beta = \frac{2\, (p + 1 - b)}{3 - b}, \qquad 2^{p,1}_{b,\ast} = \frac{2\, [(N - 1)\, p + b - 1]}{2N - 1 - b}
\end{equation}
(see \eqref{31} and \eqref{30}). It follows from \eqref{505}, \eqref{601}, and \eqref{36} that
\begin{equation} \label{605}
\beta < 2^{p,1}_{b,\ast}.
\end{equation}

Theorem \ref{thm:pure-power} together with \eqref{605} gives the following multiplicity result for the pure power equation
\begin{equation} \label{604}
- \Delta u + \frac{|u|^{p-2}\, u}{|x|^{b-1}} = |u|^{s-2}\, u \quad \text{in } \R^N,
\end{equation}
where
\[
\E(u) = \frac{1}{2} \int_{\R^N} |\nabla u|^2\, dx + \frac{1}{p} \int_{\R^N} \frac{|u|^p}{|x|^{b-1}}\, dx - \frac{1}{s} \int_{\R^N} |u|^s\, dx, \quad u \in E^{p,1}_{b,\rad}(\R^N).
\]

\begin{theorem} \label{Theorem 4.20}
Assume \eqref{601}. If
\[
2^{p,1}_{b,\ast} < s < \begin{cases}
\infty & \text{if } N = 2\\[5pt]
2^\ast & \text{if } N \ge 3,
\end{cases}
\]
then equation \eqref{604} has an infinite sequence of nontrivial weak radial solutions $u_k$ such that $\E(u_k) > 0$ and $\E(u_k) \nearrow \infty$.
\end{theorem}

\begin{remark} \label{rmk:ip-pure}
Theorem \ref{Theorem 4.20} is contained in \cite[Theorem 5.2 $(iii)$]{GlPeRi}, applied with $\lambda = 0$, $\eta = 0$, and $h(t) = |t|^{s-2}\, t$. Indeed, when the parameters $q$ and $b$ of \cite{GlPeRi} are replaced by $p$ and $b - 1$, respectively, the energy space used there coincides with $E^{p,1}_b(\R^N)$, the scaling used there coincides with \eqref{103}, and the radial critical exponent $2^{\rm rad}_{b-1,p,0}$ defined there coincides with $2^{p,1}_{b,\ast}$.
\end{remark}

Theorem \ref{thm:pure-crit-nonex} gives the following nonexistence result for the pure critical power equation
\begin{equation} \label{606}
- \Delta u + \frac{|u|^{p-2}\, u}{|x|^{b-1}} = |u|^{2^\ast - 2}\, u \quad \text{in } \R^N,
\end{equation}
where $N \ge 3$.

\begin{theorem} \label{Theorem 4.21}
Assume \eqref{601}. Then equation \eqref{606} has no nontrivial weak radial solution.
\end{theorem}

\begin{remark} \label{rmk:ip-nonex}
Theorem \ref{Theorem 4.21} is contained in \cite[Proposition 2.6]{GlPeRi}, applied with $\eta = 0$ and $r = 2^\ast$, where radial symmetry is not assumed.
\end{remark}

Theorem \ref{thm:large-mu} together with \eqref{404} and \eqref{605} gives the following multiplicity result for the perturbed critical equation
\begin{equation} \label{607}
- \Delta u + \frac{|u|^{p-2}\, u}{|x|^{b-1}} = \mu\, |u|^{s-2}\, u + |u|^{2^\ast - 2}\, u \quad \text{in } \R^N,
\end{equation}
where $N \ge 3$.

\begin{theorem} \label{Theorem 4.22}
Assume \eqref{601}. If $2^{p,1}_{b,\ast} < s < 2^\ast$, then for any $m \ge 1$, $\exists \mu_m > 0$ such that equation \eqref{607} has $m$ distinct pairs of nontrivial weak radial solutions at positive energy levels for all $\mu > \mu_m$. In particular, the number of solutions goes to infinity as $\mu \to \infty$.
\end{theorem}

\begin{remark} \label{rmk:ip-large}
Theorem \ref{Theorem 4.22} does not seem to follow from the results of \cite{GlPeRi}. In the notation of \cite{GlPeRi}, equation \eqref{607} corresponds to the choice $\eta_1 = \eta_2 = 0$ of the weights on the critical term and on the perturbation, and this choice is excluded by the assumptions of \cite[Theorem 1.5 and Theorem 5.3]{GlPeRi} since $b > 1$.
\end{remark}

\subsection{Applications to the Chern-Simons-Schr\"{o}dinger equation} \label{ssec:app-css}

When $N = 2$, $p = 2$, $q = 3$, $a = 1/8 \pi^2$, $b = 3$, and $u = u(r)$, equation \eqref{1} reduces to the Chern-Simons-Schr\"{o}dinger equation
\begin{equation} \label{620}
- \Delta u + \left(\frac{\mathfrak{h}_u^2(|x|)}{|x|^2} + \int_{|x|}^\infty \frac{\mathfrak{h}_u(\rho)}{\rho}\, u^2(\rho)\, d\rho\right) u = f(u) \quad \text{in } \R^2,
\end{equation}
where $\mathfrak{h}_u$ is given in \eqref{623}. By \eqref{140} and \eqref{621},
\[
I_\sigma(u) = \frac{1}{2} \int_{\R^2} |\nabla u|^2\, dx + \frac{1}{6} \int_{\R^2} \left(\frac{\mathfrak{h}_u^2(|x|)}{|x|^2} + \int_{|x|}^\infty \frac{\mathfrak{h}_u(\rho)}{\rho}\, u^2(\rho)\, d\rho\right) u^2\, dx.
\]
By Fubini's theorem and \eqref{622},
\begin{multline*}
\int_{\R^2} \left(\int_{|x|}^\infty \frac{\mathfrak{h}_u(\rho)}{\rho}\, u^2(\rho)\, d\rho\right) u^2\, dx = 2 \pi \int_0^\infty \frac{\mathfrak{h}_u(\rho)}{\rho}\, u^2(\rho) \left(\int_0^\rho \tau\, u^2(\tau)\, d\tau\right) d\rho\\[7.5pt]
= 4 \pi \int_0^\infty \frac{\mathfrak{h}_u^2(\rho)}{\rho}\, u^2(\rho)\, d\rho = 2 \int_{\R^2} \frac{\mathfrak{h}_u^2(|x|)}{|x|^2}\, u^2\, dx
\end{multline*}
and hence
\[
I_\sigma(u) = \frac{1}{2} \int_{\R^2} |\nabla u|^2\, dx + \frac{1}{2} \int_{\R^2} \frac{\mathfrak{h}_u^2(|x|)}{|x|^2}\, u^2\, dx, \quad u \in E^{2,3}_{3,\rad}(\R^2).
\]
We have
\begin{equation} \label{608}
\beta = 4 = 2^{2,3}_{3,\ast}
\end{equation}
(see \eqref{31} and \eqref{30}).

Theorem \ref{thm:pure-power} together with \eqref{608} gives the following multiplicity result for the pure power equation
\begin{equation} \label{609}
- \Delta u + \left(\frac{\mathfrak{h}_u^2(|x|)}{|x|^2} + \int_{|x|}^\infty \frac{\mathfrak{h}_u(\rho)}{\rho}\, u^2(\rho)\, d\rho\right) u = |u|^{s-2}\, u \quad \text{in } \R^2,
\end{equation}
where
\[
\E(u) = \frac{1}{2} \int_{\R^2} |\nabla u|^2\, dx + \frac{1}{2} \int_{\R^2} \frac{\mathfrak{h}_u^2(|x|)}{|x|^2}\, u^2\, dx - \frac{1}{s} \int_{\R^2} |u|^s\, dx, \quad u \in E^{2,3}_{3,\rad}(\R^2).
\]

\begin{theorem} \label{thm:css-pure}
If $s > 4$, then equation \eqref{609} has an infinite sequence of nontrivial weak radial solutions $u_k$ such that $\E(u_k) > 0$ and $\E(u_k) \nearrow \infty$.
\end{theorem}

\begin{remark} \label{rmk:css-pure}
To our knowledge, the zero-mass pure power equation \eqref{609} has not been considered previously. In \cite{MR2948224,MR3415024,MR3353806} a frequency $\omega > 0$, or equivalently a negative linear part in the nonlinearity, is present and the problem is posed in $H^1_\rad(\R^2)$, while in the zero-mass works \cite{MR4708596,MR4968174,MR4502773} an additional term $- a\, |u|^{r-2}\, u$ with $a > 0$ is added to the right-hand side (see Subsection \ref{ssec:related}). Neither device is needed here. We note that the exponent $4$ plays a distinguished role also in the positive-mass case (see \cite[Theorem 1.2]{MR2948224}).
\end{remark}

\subsection{Applications to the Schr\"{o}dinger-Poisson-Slater equation} \label{ssec:app-sps}

When $N \ge 3$, $q = 2$, $a = N - 2$, $b = N - 1$, and $u = u(r)$, equation \eqref{1} reduces to the Schr\"{o}dinger-Poisson-Slater equation
\begin{equation} \label{624}
- \Delta u + \left(\frac{1}{|x|^{N-2}} \star |u|^p\right) |u|^{p-2}\, u = f(u) \quad \text{in } \R^N,
\end{equation}
where
\begin{equation} \label{610}
1 < p < \frac{N + 2}{N - 2}.
\end{equation}
By \eqref{140} and \eqref{49},
\[
I_\sigma(u) = \frac{1}{2} \int_{\R^N} |\nabla u|^2\, dx + \frac{1}{2p} \int_{\R^N} \int_{\R^N} \frac{|u(x)|^p\, |u(y)|^p}{|x - y|^{N-2}}\, dx\, dy, \quad u \in E^{p,2}_{N-1,\rad}(\R^N).
\]
We have
\begin{equation} \label{612}
\beta = p + 1, \qquad 2^{p,2}_{N-1,\ast} = \frac{2\, [2\, (N - 1)\, p + N - 2]}{3N - 2}
\end{equation}
(see \eqref{31} and \eqref{30}). It follows from \eqref{505} and \eqref{36} that
\begin{equation} \label{613}
\beta > 2^{p,2}_{N-1,\ast}.
\end{equation}

Theorem \ref{thm:bm-eigenvalues} together with \eqref{613} gives the following result for the scaled eigenvalue problem
\begin{equation} \label{614}
- \Delta u + \left(\frac{1}{|x|^{N-2}} \star |u|^p\right) |u|^{p-2}\, u = \lambda\, |u|^{p-1}\, u \quad \text{in } \R^N,
\end{equation}
where $N \ge 3$.

\begin{theorem} \label{thm:sps-eigenvalues}
Assume \eqref{610}. Then equation \eqref{614} has a sequence of positive eigenvalues $\lambda_k \nearrow \infty$ with radial eigenfunctions.
\end{theorem}

\begin{remark} \label{rmk:sps-eigen}
For $N = 3$ and $p = 2$, Theorem \ref{thm:sps-eigenvalues} is due to Ianni and Ruiz \cite[Theorem 1.3]{MR2902293} (see also \cite{MR5043800}). For general $N$ and $p$ satisfying \eqref{610}, the sequence $\lambda_k$ was constructed in \cite[Section 2]{MarMePe} in the Coulomb-Sobolev setting, where $\beta = p + 1$ appears as the Coulomb-Sobolev critical exponent of \cite{MR3568051}.
\end{remark}

Theorem \ref{thm:pure-power} together with \eqref{613} gives the following multiplicity result for the pure power equation
\begin{equation} \label{615}
- \Delta u + \left(\frac{1}{|x|^{N-2}} \star |u|^p\right) |u|^{p-2}\, u = |u|^{s-2}\, u \quad \text{in } \R^N,
\end{equation}
where $N \ge 3$ and
\[
\E(u) = \frac{1}{2} \int_{\R^N} |\nabla u|^2\, dx + \frac{1}{2p} \int_{\R^N} \int_{\R^N} \frac{|u(x)|^p\, |u(y)|^p}{|x - y|^{N-2}}\, dx\, dy - \frac{1}{s} \int_{\R^N} |u|^s\, dx, \quad u \in E^{p,2}_{N-1,\rad}(\R^N).
\]

\begin{theorem} \label{thm:sps-pure}
Assume \eqref{610}. If $s \ne p + 1$ satisfies $2^{p,2}_{N-1,\ast} < s < 2^\ast$, then equation \eqref{615} has an infinite sequence of nontrivial weak radial solutions $u_k$.
\begin{enumroman}
\item If $s < p + 1$, then $\E(u_k) < 0$ and $\E(u_k) \nearrow 0$.
\item If $s > p + 1$, then $\E(u_k) > 0$ and $\E(u_k) \nearrow \infty$.
\end{enumroman}
\end{theorem}

\begin{remark} \label{rmk:sps-pure}
For $N = 3$ and $p = 2$, both alternatives of Theorem \ref{thm:sps-pure} are known. In the subscaled range $18/7 < s < 3$, infinitely many solutions at negative energy levels were obtained in \cite[Corollary 1.31]{MR5043800}, following the existence of a positive solution in \cite[Theorem 1.3]{MR2679375}, and in the superscaled range $3 < s < 6$, infinitely many radial solutions with diverging energies were obtained in \cite[Theorem 1.2]{MR2902293}. For general $N$ and $p$, the existence of a radial groundstate in the whole range of Theorem \ref{thm:sps-pure} follows from \cite[Theorem 5]{MR3568051}, and existence results for more general nonlinearities in both regimes are given in \cite{MarMePe}. The infinite multiplicity for general $N$ and $p$ appears to be new.
\end{remark}

Now we strengthen \eqref{610} to
\begin{equation} \label{616}
1 < p < \frac{N}{N - 2}.
\end{equation}
Theorem \ref{thm:pure-crit-nonex} gives the following nonexistence result for the pure critical power equation
\begin{equation} \label{617}
- \Delta u + \left(\frac{1}{|x|^{N-2}} \star |u|^p\right) |u|^{p-2}\, u = |u|^{2^\ast - 2}\, u \quad \text{in } \R^N,
\end{equation}
where $N \ge 3$.

\begin{theorem} \label{thm:sps-nonex}
Assume \eqref{616}. Then equation \eqref{617} has no nontrivial weak radial solution.
\end{theorem}

\begin{remark} \label{rmk:sps-nonex}
For $N = 3$ and $p = 2$, the corresponding nonexistence result for solutions in $H^2_\loc(\R^3)$ was obtained in \cite[Corollary 2.6]{MR2902293}.
\end{remark}

Theorem \ref{thm:left-nbhd} together with \eqref{613} gives the following multiplicity result for the perturbed critical equation
\begin{equation} \label{618}
- \Delta u + \left(\frac{1}{|x|^{N-2}} \star |u|^p\right) |u|^{p-2}\, u = \lambda\, |u|^{p-1}\, u + |u|^{2^\ast - 2}\, u \quad \text{in } \R^N,
\end{equation}
where $N \ge 3$, $\lambda > 0$, and $\lambda_k \nearrow \infty$ is the sequence of positive eigenvalues of problem \eqref{614} given in Theorem \ref{thm:sps-eigenvalues}. 

\begin{theorem} \label{thm:sps-left}
Assume \eqref{616}. If $\lambda_k = \cdots = \lambda_{k+m-1} < \lambda_{k+m}$ for some $k, m \ge 1$, then $\exists \delta_k > 0$ such that equation \eqref{618} has $m$ distinct pairs of nontrivial weak radial solutions at positive energy levels for all $\lambda \in (\lambda_k - \delta_k,\lambda_k)$. In particular, there exists a nontrivial solution for each $\lambda \in \bigcup_{k=1}^\infty (\lambda_k - \delta_k,\lambda_k)$.
\end{theorem}

\begin{remark} \label{rmk:sps-left}
Theorem \ref{thm:sps-left} is known. For $N = 3$ and $p = 2$ it is \cite[Theorem 1.25]{MR5043800}, and for general $N$ and $p$ satisfying \eqref{616} it is \cite[Theorem 2.12]{MarMePe}. The existence of a positive solution of equation \eqref{618} for $N = 3$, $p = 2$, and some $\lambda > 0$ was proved earlier in \cite[Theorem 1.4]{MR3912770}.
\end{remark}

Theorem \ref{thm:large-mu} together with \eqref{613} and \eqref{616} gives the following multiplicity result for the perturbed critical equation
\begin{equation} \label{619}
- \Delta u + \left(\frac{1}{|x|^{N-2}} \star |u|^p\right) |u|^{p-2}\, u = \mu\, |u|^{s-2}\, u + |u|^{2^\ast - 2}\, u \quad \text{in } \R^N,
\end{equation}
where $N \ge 3$.

\begin{theorem} \label{thm:sps-large}
Assume \eqref{616}. If $p + 1 < s < 2^\ast$, then for any $m \ge 1$, $\exists \mu_m > 0$ such that equation \eqref{619} has $m$ distinct pairs of nontrivial weak radial solutions at positive energy levels for all $\mu > \mu_m$. In particular, the number of solutions goes to infinity as $\mu \to \infty$.
\end{theorem}

\begin{remark} \label{rmk:sps-large}
Theorem \ref{thm:sps-large} is known. For $N = 3$ and $p = 2$ it is \cite[Theorem 1.28]{MR5043800}, and for general $N$ and $p$ satisfying \eqref{616} it is \cite[Corollary 2.15]{MarMePe}. For $N = 3$ and $p = 2$, the existence of a positive solution of equation \eqref{619} was proved earlier in \cite[Theorem 1.1]{MR3912770} for $s \in (3,4]$ and all sufficiently large $\mu$, and for $s \in (4,6)$ and all $\mu > 0$.
\end{remark}

When $N \ge 4$, Theorem \ref{thm:bn} together with \eqref{613} gives the following analog of the classical results of Br\'{e}zis and Nirenberg \cite{MR709644} and Capozzi et al.\! \cite{MR831041} for the Schr\"{o}dinger-Poisson-Slater equation (see \eqref{88}).

\begin{theorem} \label{thm:sps-bn}
Let $N \ge 4$ and assume \eqref{616}. Then equation \eqref{618} has a nontrivial weak radial solution $u$ with $0 < \E(u) < \frac{1}{N}\, S^{N/2}$ in the following cases:
\begin{enumroman}
\item $0 < \lambda < \lambda_1$;
\item $\lambda_k < \lambda < \lambda_{k+1}$ for some $k \ge 1$.
\end{enumroman}
\end{theorem}

\begin{remark} \label{rmk:sps-bn}
Theorem \ref{thm:sps-bn} appears to be new. The results of \cite[Theorem 1.25]{MR5043800} and \cite[Theorem 2.12]{MarMePe} recalled in Remark \ref{rmk:sps-left} give solutions of equation \eqref{618} only for $\lambda$ in suitably small left neighborhoods of the eigenvalues $\lambda_k$, whereas Theorem \ref{thm:sps-bn} gives a solution for every $\lambda > 0$ that is not an eigenvalue. The restriction $N \ge 4$ cannot be removed within the present approach since for $N = 3$ condition \eqref{304} requires $p > 3$, which is incompatible with \eqref{616}. In particular, the case $N = 3$, $p = 2$ is not covered.
\end{remark}

\subsection{Concluding remarks and open problems} \label{section:open}

We conclude with some problems that are left open by the present work.

\begin{enumerate}
\item \emph{Resonant case.} Does equation \eqref{bn-eq} have a nontrivial weak radial solution when $\lambda = \lambda_k$? As mentioned in Remark \ref{rmk:resonant}, the linking geometry persists at $\lambda = \lambda_k$, and the difficulty lies in keeping the minimax level below $\frac{1}{N}\, S^{N/2}$.

\item \emph{Br\'{e}zis-Nirenberg problem in dimension three.} For $N = 3$, condition \eqref{304} requires $\beta > 4$, which for the Schr\"{o}dinger-Poisson-Slater equation means $p > 3$ and is incompatible with \eqref{616} (see Remark \ref{rmk:sps-bn}). This restriction enters through Lemma \ref{lem:bubble}, where the gain $\eps^{\delta_\beta}$ from the subcritical term must dominate the error $\O(\eps^{N-2})$ in the gradient term. Whether Theorem \ref{thm:sps-bn} holds for $N = 3$, in particular in the physical case $p = 2$, is open.

\item \emph{Endpoint embedding.} Theorem \ref{Theorem 1} leaves open the embedding at $s = 2^{p,q}_{b,\ast}$. It holds when $q = 1$ \cite[Lemma 5.1]{GlPeRi} and fails when $q = 2$ and $b = N - 1$ \cite[Theorem 4]{MR3568051}. What happens for general $(p,q,b)$?

\item \emph{Borderline Trudinger-Moser inequality.} Does Theorem \ref{Theorem 7} hold for $\nu = 4 \pi$ (see Remark \ref{rmk:tm-sharp})? Theorem \ref{Theorem 7} also suggests that zero-mass Chern-Simons-Schr\"{o}dinger equations with critical exponential growth can be treated directly in $E^{2,3}_{3,\rad}(\R^2)$, without the additional term used in \cite{MR4708596,MR4968174,MR4502773}.

\item \emph{Positivity.} Are the nonnegative solutions obtained in Remark \ref{rmk:sign} positive when $1 < p < 2$, or can they have dead cores? This question is relevant for the Schr\"{o}dinger-Poisson-Slater equation, where $p < 2$ in all the cases covered by Theorem \ref{thm:sps-bn}.

\item \emph{Scaled eigenvalues at the borderline.} When $\beta = 2^{p,q}_{b,\ast}$, as in the Chern-Simons-Schr\"{o}dinger case, Theorem \ref{thm:bm-eigenvalues} does not apply. Is there a general theory of scaled eigenvalues in the borderline case $\beta = 2^{p,q}_{b,\ast}$?

\item \emph{Nonradial solutions.} The ball-mass term is anchored at the origin, and by Theorem \ref{thm:no-embedding} the space $E^{p,q}_b(\R^N)$ has no Lebesgue embeddings other than the Sobolev embedding. For $q = 2$ and $b = N - 1$ the Coulomb-Sobolev space provides a translation-invariant framework for nonradial solutions \cite{MR3568051}, but for $q \ne 1, 2$ no such framework seems to be available.
\end{enumerate}

\def\cprime{$''$}


\begin{thebibliography}{10}

\bibitem{MR2902133}
Claudianor~O. Alves, Marco A.~S. Souto, and Marcelo Montenegro.
\newblock Existence of solution for two classes of elliptic problems in
  {$\R^N$} with zero mass.
\newblock {\em J. Differential Equations}, 252(10):5735--5750, 2012.

\bibitem{MR829403}
A.~Ambrosetti and M.~Struwe.
\newblock A note on the problem {$-\Delta u=\lambda u+u\vert u\vert \sp {2\sp
  \ast-2}$}.
\newblock {\em Manuscripta Math.}, 54(4):373--379, 1986.

\bibitem{MR2465993}
Antonio Ambrosetti.
\newblock On {S}chr\"{o}dinger-{P}oisson systems.
\newblock {\em Milan J. Math.}, 76:257--274, 2008.

\bibitem{MR0370183}
Antonio Ambrosetti and Paul~H. Rabinowitz.
\newblock Dual variational methods in critical point theory and applications.
\newblock {\em J. Functional Analysis}, 14:349--381, 1973.

\bibitem{MR3385192}
Marino Badiale, Michela Guida, and Sergio Rolando.
\newblock Compactness and existence results in weighted {S}obolev spaces of
  radial functions. {P}art {I}: compactness.
\newblock {\em Calc. Var. Partial Differential Equations}, 54(1):1061--1090,
  2015.

\bibitem{MR3576582}
Marino Badiale, Michela Guida, and Sergio Rolando.
\newblock Compactness and existence results in weighted {S}obolev spaces of
  radial functions. {P}art {II}: existence.
\newblock {\em NoDEA Nonlinear Differential Equations Appl.}, 23(6):Art. 67,
  34, 2016.

\bibitem{MR3852465}
Jacopo Bellazzini, Marco Ghimenti, Carlo Mercuri, Vitaly Moroz, and Jean
  Van~Schaftingen.
\newblock Sharp {G}agliardo-{N}irenberg inequalities in fractional
  {C}oulomb-{S}obolev spaces.
\newblock {\em Trans. Amer. Math. Soc.}, 370(11):8285--8310, 2018.

\bibitem{MR2187794}
Vieri Benci, Carlo~R. Grisanti, and Anna~Maria Micheletti.
\newblock Existence of solutions for the nonlinear {S}chr\"odinger equation
  with {$V(\infty)=0$}.
\newblock In {\em Contributions to nonlinear analysis}, volume~66 of {\em
  Progr. Nonlinear Differential Equations Appl.}, pages 53--65. Birkh\"auser,
  Basel, 2006.

\bibitem{MR695535}
H.~Berestycki and P.-L. Lions.
\newblock Nonlinear scalar field equations. {I}. {E}xistence of a ground state.
\newblock {\em Arch. Rational Mech. Anal.}, 82(4):313--345, 1983.

\bibitem{MR695536}
H.~Berestycki and P.-L. Lions.
\newblock Nonlinear scalar field equations. {II}. {E}xistence of infinitely
  many solutions.
\newblock {\em Arch. Rational Mech. Anal.}, 82(4):347--375, 1983.

\bibitem{MR2013491}
Olivier Bokanowski, Jos\'{e}~L. L\'{o}pez, and Juan Soler.
\newblock On an exchange interaction model for quantum transport: the
  {S}chr\"{o}dinger-{P}oisson-{S}later system.
\newblock {\em Math. Models Methods Appl. Sci.}, 13(10):1397--1412, 2003.

\bibitem{MR1702877}
Olivier Bokanowski and Norbert~J. Mauser.
\newblock Local approximation for the {H}artree-{F}ock exchange potential: a
  deformation approach.
\newblock {\em Math. Models Methods Appl. Sci.}, 9(6):941--961, 1999.

\bibitem{MR539217}
Ha\"{\i}m Br\'{e}zis and Tosio Kato.
\newblock Remarks on the {S}chr\"{o}dinger operator with singular complex
  potentials.
\newblock {\em J. Math. Pures Appl. (9)}, 58(2):137--151, 1979.

\bibitem{MR699419}
Ha{\"{\i}}m Br{\'e}zis and Elliott Lieb.
\newblock A relation between pointwise convergence of functions and convergence
  of functionals.
\newblock {\em Proc. Amer. Math. Soc.}, 88(3):486--490, 1983.

\bibitem{MR709644}
Ha{\"{\i}}m Br{\'e}zis and Louis Nirenberg.
\newblock Positive solutions of nonlinear elliptic equations involving critical
  {S}obolev exponents.
\newblock {\em Comm. Pure Appl. Math.}, 36(4):437--477, 1983.

\bibitem{MR2948224}
Jaeyoung Byeon, Hyungjin Huh, and Jinmyoung Seok.
\newblock Standing waves of nonlinear {S}chr\"odinger equations with the gauge
  field.
\newblock {\em J. Funct. Anal.}, 263(6):1575--1608, 2012.

\bibitem{MR768824}
L.~Caffarelli, R.~Kohn, and L.~Nirenberg.
\newblock First order interpolation inequalities with weights.
\newblock {\em Compositio Math.}, 53(3):259--275, 1984.

\bibitem{MR1163431}
D.~M. Cao.
\newblock Nontrivial solution of semilinear elliptic equation with critical
  exponent in {${\bf R}^2$}.
\newblock {\em Comm. Partial Differential Equations}, 17(3-4):407--435, 1992.

\bibitem{MR831041}
A.~Capozzi, D.~Fortunato, and G.~Palmieri.
\newblock An existence result for nonlinear elliptic problems involving
  critical {S}obolev exponent.
\newblock {\em Ann. Inst. H. Poincar\'e Anal. Non Lin\'eaire}, 2(6):463--470,
  1985.

\bibitem{MR878016}
Lennart Carleson and Sun-Yung~A. Chang.
\newblock On the existence of an extremal function for an inequality of {J}.\
  {M}oser.
\newblock {\em Bull. Sci. Math. (2)}, 110(2):113--127, 1986.

\bibitem{MR867663}
G.~Cerami, S.~Solimini, and M.~Struwe.
\newblock Some existence results for superlinear elliptic boundary value
  problems involving critical exponents.
\newblock {\em J. Funct. Anal.}, 69(3):289--306, 1986.

\bibitem{MR779872}
Giovanna Cerami, Donato Fortunato, and Michael Struwe.
\newblock Bifurcation and multiplicity results for nonlinear elliptic problems
  involving critical {S}obolev exponents.
\newblock {\em Ann. Inst. H. Poincar\'e Anal. Non Lin\'eaire}, 1(5):341--350,
  1984.

\bibitem{MR3415024}
Patricia~L. Cunha, Pietro d'Avenia, Alessio Pomponio, and Gaetano Siciliano.
\newblock A multiplicity result for {C}hern-{S}imons-{S}chr\"odinger equation
  with a general nonlinearity.
\newblock {\em NoDEA Nonlinear Differential Equations Appl.}, 22(6):1831--1850,
  2015.

\bibitem{MR2371112}
Marco Degiovanni and Sergio Lancelotti.
\newblock Linking over cones and nontrivial solutions for {$p$}-{L}aplace
  equations with {$p$}-superlinear nonlinearity.
\newblock {\em Ann. Inst. H. Poincar\'e Anal. Non Lin\'eaire}, 24(6):907--919,
  2007.

\bibitem{MR4292779}
Tomas Dutko, Carlo Mercuri, and Teresa~Megan Tyler.
\newblock Groundstates and infinitely many high energy solutions to a class of
  nonlinear {S}chr\"odinger-{P}oisson systems.
\newblock {\em Calc. Var. Partial Differential Equations}, 60(5):Paper No. 174,
  46, 2021.

\bibitem{MR0478189}
Edward~R. Fadell and Paul~H. Rabinowitz.
\newblock Generalized cohomological index theories for {L}ie group actions with
  an application to bifurcation questions for {H}amiltonian systems.
\newblock {\em Invent. Math.}, 45(2):139--174, 1978.

\bibitem{MR2379460}
Fran\c~cois Genoud and Charles~A. Stuart.
\newblock Schr\"odinger equations with a spatially decaying nonlinearity:
  existence and stability of standing waves.
\newblock {\em Discrete Contin. Dyn. Syst.}, 21(1):137--186, 2008.

\bibitem{MR1814364}
David Gilbarg and Neil~S. Trudinger.
\newblock {\em Elliptic partial differential equations of second order}.
\newblock Classics in Mathematics. Springer-Verlag, Berlin, 2001.
\newblock Reprint of the 1998 edition.

\bibitem{Gill}
T.~S. Gill.
\newblock Optical guiding of laser beam in nonuniform plasma.
\newblock {\em Pramana J. Phys.}, 55(5--6):835--842, 2000.

\bibitem{GlPeRi}
E. Gloss, K. Perera, and B. Ribeiro.
\newblock Inhomogeneous nonlinear {S}chr\"{o}dinger equations with competing
  singular nonlinearities.
\newblock preprint, \href{https://arxiv.org/abs/2601.02909}{\tt
  arXiv:2601.02909}.

\bibitem{MR2902293}
Isabella Ianni and David Ruiz.
\newblock Ground and bound states for a static
  {S}chr\"odinger-{P}oisson-{S}later problem.
\newblock {\em Commun. Contemp. Math.}, 14(1):1250003, 22, 2012.

\bibitem{MR2834784}
Norihisa Ikoma.
\newblock On radial solutions of inhomogeneous nonlinear scalar field
  equations.
\newblock {\em J. Math. Anal. Appl.}, 386(2):744--762, 2012.

\bibitem{MR1084552}
R.~Jackiw and So-Young Pi.
\newblock Classical and quantal nonrelativistic {C}hern-{S}imons theory.
\newblock {\em Phys. Rev. D (3)}, 42(10):3500--3513, 1990.

\bibitem{MR1817225}
Elliott~H. Lieb and Michael Loss.
\newblock {\em Analysis}, volume~14 of {\em Graduate Studies in Mathematics}.
\newblock American Mathematical Society, Providence, RI, second edition, 2001.

\bibitem{MR636734}
P.-L. Lions.
\newblock Some remarks on {H}artree equation.
\newblock {\em Nonlinear Anal.}, 5(11):1245--1256, 1981.

\bibitem{MR778970}
P.-L. Lions.
\newblock The concentration-compactness principle in the calculus of
  variations. {T}he locally compact case. {I}.
\newblock {\em Ann. Inst. H. Poincar\'e Anal. Non Lin\'eaire}, 1(2):109--145,
  1984.

\bibitem{MR834360}
P.-L. Lions.
\newblock The concentration-compactness principle in the calculus of
  variations. {T}he limit case. {I}.
\newblock {\em Rev. Mat. Iberoamericana}, 1(1):145--201, 1985.

\bibitem{LiuTripathi}
C.~S. Liu and V.~K. Tripathi.
\newblock Laser guiding in an axially nonuniform plasma channel.
\newblock {\em Phys. Plasmas}, 1(9):3100--3103, 1994.

\bibitem{MR3912770}
Zhisu Liu, Zhitao Zhang, and Shuibo Huang.
\newblock Existence and nonexistence of positive solutions for a static
  {S}chr\"odinger-{P}oisson-{S}later equation.
\newblock {\em J. Differential Equations}, 266(9):5912--5941, 2019.

\bibitem{MarMePe}
Artur~Jorge Marinho, Carlo Mercuri, and Kanishka Perera.
\newblock New solutions to {S}chr\"{o}dinger-{P}oisson-{S}later equations in
  {C}oulomb-{S}obolev spaces.
\newblock preprint, \href{https://arxiv.org/abs/2602.12784}{\tt
  arXiv:2602.12784}.

\bibitem{MR1836081}
N.~J. Mauser.
\newblock The {S}chr\"{o}dinger-{P}oisson-{$X\alpha$} equation.
\newblock {\em Appl. Math. Lett.}, 14(6):759--763, 2001.

\bibitem{MR3568051}
Carlo Mercuri, Vitaly Moroz, and Jean Van~Schaftingen.
\newblock Groundstates and radial solutions to nonlinear
  {S}chr\"odinger-{P}oisson-{S}later equations at the critical frequency.
\newblock {\em Calc. Var. Partial Differential Equations}, 55(6):Art. 146, 58,
  2016.

\bibitem{MR5043800}
Carlo Mercuri and Kanishka Perera.
\newblock Variational methods for scaled functionals with applications to the
  {S}chr\"odinger-{P}oisson-{S}later equation.
\newblock {\em J. Math. Pures Appl. (9)}, 212:Paper No. 103885, 70, 2026.

\bibitem{MR3210961}
Vitaly Moroz and Cyrill~B. Muratov.
\newblock Asymptotic properties of ground states of scalar field equations with
  a vanishing parameter.
\newblock {\em J. Eur. Math. Soc. (JEMS)}, 16(5):1081--1109, 2014.

\bibitem{MR0301504}
J.~Moser.
\newblock A sharp form of an inequality by {N}. {T}rudinger.
\newblock {\em Indiana Univ. Math. J.}, 20:1077--1092, 1970/71.

\bibitem{MR4999814}
Kanishka Perera.
\newblock Abstract multiplicity theorems and applications to critical growth
  problems.
\newblock {\em J. Anal. Math.}, 157(1):211--223, 2025.

\bibitem{MR2640827}
Kanishka Perera, Ravi~P. Agarwal, and Donal O'Regan.
\newblock {\em Morse theoretic aspects of {$p$}-{L}aplacian type operators},
  volume 161 of {\em Mathematical Surveys and Monographs}.
\newblock American Mathematical Society, Providence, RI, 2010.

\bibitem{MR3353806}
Alessio Pomponio and David Ruiz.
\newblock A variational analysis of a gauged nonlinear {S}chr\"odinger
  equation.
\newblock {\em J. Eur. Math. Soc. (JEMS)}, 17(6):1463--1486, 2015.

\bibitem{MR2679375}
David Ruiz.
\newblock On the {S}chr\"odinger-{P}oisson-{S}later system: behavior of
  minimizers, radial and nonradial cases.
\newblock {\em Arch. Ration. Mech. Anal.}, 198(1):349--368, 2010.

\bibitem{MR2032129}
\'{O}scar S\'{a}nchez and Juan Soler.
\newblock Long-time dynamics of the {S}chr\"{o}dinger-{P}oisson-{S}later
  system.
\newblock {\em J. Statist. Phys.}, 114(1-2):179--204, 2004.

\bibitem{MR4708596}
Liejun Shen.
\newblock Zero-mass gauged {S}chr\"odinger equations with supercritical
  exponential growth.
\newblock {\em J. Differential Equations}, 393:204--237, 2024.

\bibitem{MR4968174}
Liejun Shen and Marco Squassina.
\newblock Generalized {C}hern-{S}imons-{S}chr\"odinger system with critical
  exponential growth: the zero-mass case.
\newblock {\em Asymptot. Anal.}, 143(2):746--767, 2025.

\bibitem{Slater}
J.~C. Slater.
\newblock A simplification of the {H}artree-{F}ock method.
\newblock {\em Phys. Rev.}, 81(3):385--390, 1951.

\bibitem{MR0454365}
Walter~A. Strauss.
\newblock Existence of solitary waves in higher dimensions.
\newblock {\em Comm. Math. Phys.}, 55(2):149--162, 1977.

\bibitem{MR760051}
Michael Struwe.
\newblock A global compactness result for elliptic boundary value problems
  involving limiting nonlinearities.
\newblock {\em Math. Z.}, 187(4):511--517, 1984.

\bibitem{MR2431434}
Michael Struwe.
\newblock {\em Variational methods}, volume~34 of {\em Ergebnisse der
  Mathematik und ihrer Grenzgebiete. 3. Folge. A Series of Modern Surveys in
  Mathematics [Results in Mathematics and Related Areas. 3rd Series. A Series
  of Modern Surveys in Mathematics]}.
\newblock Springer-Verlag, Berlin, fourth edition, 2008.
\newblock Applications to nonlinear partial differential equations and
  Hamiltonian systems.

\bibitem{MR2334597}
Jiabao Su, Zhi-Qiang Wang, and Michel Willem.
\newblock Weighted {S}obolev embedding with unbounded and decaying radial
  potentials.
\newblock {\em J. Differential Equations}, 238(1):201--219, 2007.

\bibitem{MR0463908}
Giorgio Talenti.
\newblock Best constant in {S}obolev inequality.
\newblock {\em Ann. Mat. Pura Appl. (4)}, 110:353--372, 1976.

\bibitem{MR0216286}
Neil~S. Trudinger.
\newblock On imbeddings into {O}rlicz spaces and some applications.
\newblock {\em J. Math. Mech.}, 17:473--483, 1967.

\bibitem{MR768629}
J.~L. V{\'a}zquez.
\newblock A strong maximum principle for some quasilinear elliptic equations.
\newblock {\em Appl. Math. Optim.}, 12(3):191--202, 1984.

\bibitem{MR1400007}
Michel Willem.
\newblock {\em Minimax theorems}.
\newblock Progress in Nonlinear Differential Equations and their Applications,
  24. Birkh\"auser Boston Inc., Boston, MA, 1996.

\bibitem{MR4502773}
Ning Zhang, Xianhua Tang, and Sitong Chen.
\newblock Mountain-pass type solutions for the {C}hern-{S}imons-{S}chr\"odinger
  equation with zero mass potential and critical exponential growth.
\newblock {\em J. Geom. Anal.}, 33(1):Paper No. 12, 28, 2023.

\end{thebibliography}
\end{document}